\documentclass[reqno]{amsart}

\usepackage{amsmath,amssymb,amscd,amsthm,accents} \usepackage{graphicx}
\DeclareGraphicsExtensions{.eps} \usepackage{mathrsfs}
\usepackage[mathcal]{eucal}
\usepackage{enumitem}

\def\sideremark#1{\ifvmode\leavevmode\fi\vadjust{\vbox to0pt{\vss 
			\hbox to 0pt{\hskip\hsize\hskip1em          
				\vbox{\hsize3cm\tiny\raggedright\pretolerance10000%             
					\noindent #1\hfill}\hss}\vbox to8pt{\vfil}\vss}}}%

\newtheorem{Thm}{Theorem}{\bfseries}{\itshape}
\newtheorem*{Thm*}{Theorem}{\bfseries}{\itshape}
\newtheorem{Cor}{Corollary}{\bfseries}{\itshape}
\newtheorem{Prop}[Cor]{Proposition}{\bfseries}{\itshape}
\newtheorem{Lem}[Cor]{Lemma}{\bfseries}{\itshape}
\newtheorem*{Lem*}{Lemma}{\bfseries}{\itshape}
\newtheorem{Fact}[Cor]{Fact}{\bfseries}{\itshape}
\newtheorem{Conj}[Cor]{Conjecture}{\bfseries}{\itshape}
\newtheorem{Def}[Cor]{Definition}{\bfseries}{\rmfamily}
{\scshape}{\rmfamily}
\newtheorem{Rem}[Cor]{Remark}{\scshape}{\rmfamily}
\newtheorem{Claim}{Claim}{\bfseries}{\itshape}

\renewcommand\ge{\geqslant} \renewcommand\le{\leqslant}
\let\tildeaccent=\~ \let\hataccent=\^
\renewcommand\~[1]{\widetilde{#1}}

\def\<{\left<} \def\>{\right>} \def\({\left(} \def\){\right)}

 \def\norm#1{\Vert #1\Vert} 

\let\parasymbol=\S \def\secref#1{\parasymbol\ref{#1}}

\def\pd#1#2{\frac{\partial#1}{\partial#2}}

\let\polishL=l \def\Zoladek.{\.Zol\c adek}

 \def\Mat{\operatorname{Mat}}
 
\def\Re{\operatorname{Re}} \def\Im{\operatorname{Im}}
\def\Arg{\operatorname{Arg}} \def\dist{\operatorname{dist}}
\def\length{\operatorname{length}}
\def\diam{\operatorname{diam}} 
\def\GL{\operatorname{GL}} 
\def\etc.{\emph{etc}.}

\def\:{\colon} \def\R{{\mathbb R}} \def\C{{\mathbb C}} \def\Z{{\mathbb
    Z}} \def\N{{\mathbb N}} \def\Q{{\mathbb Q}} \def\P{{\mathbb P}}
\def\H{{\mathbb H}}
\def\D{{\mathbb D}}
\def\V{{\mathbb V}}
\def\F{{\mathbb F}}
\def\A{{\mathbb A}}
\def\K{{\mathbb K}}

\let\PolishL=\L % remember polish L
\def\L{{\mathbb L}}

\def\I{{\mathbb I}} \def\e{\varepsilon} \def\S{\varSigma}

\def\poly{\operatorname{poly}}

 \def\d{\,\mathrm d}

 \def\Lojas.{\PolishL oja\-sie\-wicz}
 
\def\cP{{\mathcal P}}

\def\cF{{\mathcal F}} \def\cL{{\mathcal L}} 
  
  \def\cC{{\mathcal C}}

\def\cS{{\mathcal S}} 
\def\cE{{\mathcal E}}
\def\cA{{\mathcal A}}
\def\cO{{\mathcal O}}
\def\cX{{\mathcal X}}
\def\cY{{\mathcal Y}}
\def\cV{{\mathcal V}}

\def\sF{{\mathscr F}}

\def\trans{\pitchfork}

\def\rest#1{{\vert_{#1}}}

\def\w{\omega}

\def\id{\operatorname{id}}

\def\vx{{\mathbf x}}
\def\vz{{\mathbf z}}
\def\vw{{\mathbf w}}
\def\vy{{\mathbf y}}

\def\valpha{{\boldsymbol\alpha}}
\def\vbeta{{\boldsymbol\beta}}

\def\vnu{{\boldsymbol\nu}}
\def\vxi{{\boldsymbol\xi}}

\def\an{{\mathrm{an}}}

\def\he#1{{\{#1\}}}
\def\hrho{{\he\rho}}

\def\hsigma{{\he\sigma}}

\def\alg{\mathrm{alg}}
\def\trans{\mathrm{trans}}

\DeclareMathOperator{\Der}{Der}
\DeclareMathOperator{\Spec}{Spec}

\def\LN{{\mathrm{LN}}}
\def\PF{{\mathrm{PF}}}
\def\rPF{{\mathrm{rPF}}}
\def\qf{{\mathrm{qf}}}
\def\SC{{\mathrm{SC}}}
\def\so{\raisebox{.5ex}{\scalebox{0.6}{\#}}\kern-.02em{}o}

\DeclareMathOperator{\eff}{Eff}
\DeclareMathOperator{\gr}{Gr}
\DeclareMathOperator{\End}{End}

\begin{document}

% +Title
\title{Log-Noetherian Functions}

\author{Gal Binyamini} \address{Weizmann Institute of Science,
  Rehovot, Israel} \email{gal.binyamini@weizmann.ac.il} \thanks{
  Funded by the European Union (ERC, SharpOS, 101087910), and by the
  ISRAEL SCIENCE FOUNDATION (grant No. 2067/23) and by the Shimon and
  Golde Picker - Weizmann Annual Grant.}


\date{\today}

\begin{abstract}
  We introduce the class of \emph{Log-Noetherian} (LN)
  functions. These are holomorphic solutions to algebraic differential
  equations (in several variables) with logarithmic singularities. We
  prove an upper bound on the number of solutions for systems of LN
  equations, resolving in particular Khovanskii's conjecture for
  Noetherian functions. Consequently, we show that the structure
  $\R_\LN$ generated by LN-functions, as well as its expansion
  $\R_{\LN,\exp}$, are effectively o-minimal: definable sets in these
  structures admit effective bounds on their complexity in terms of
  the complexity of the defining formulas.

  We show that $\R_{\LN,\exp}$ contains the horizontal sections of
  regular flat connections with quasiunipotent monodromy over
  algebraic varieties. It therefore contains the universal covers of
  Shimura varieties and period maps of polarized variations of
  $\Z$-Hodge structures. We also give an effective Pila-Wilkie theorem
  for $\R_{\LN,\exp}$-definable sets. Thus $\R_{\LN,\exp}$ can be used
  as an effective variant of $\R_{\an,\exp}$ in the various
  applications of o-minimality to arithmetic geometry and Hodge
  theory.
\end{abstract}
%% -Title
\maketitle
\date{\today}

\setcounter{tocdepth}{1}
{\small \tableofcontents}

\section{Introduction}

\subsection{Synopsis}

We introduce the class of \emph{Log-Noetherian} (LN) functions, which
are holomorphic solutions for systems of algebraic differential
equations with logarithmic singularities. This class strictly contains
the \emph{Noetherian} functions considered by Khovanskii (which is the
non-singular case).

We show that the structure $\R_\LN$ generated by LN-functions is an
effectively o-minimal reduct of $\R_\an$. This means that one can
associate a \emph{format} $\sF$ to each formula $\psi$ in the
first-order language of $\R_\LN$, and replace the general finiteness
results of o-minimality by an effective bound depending on $\sF$. For
instance this implies a Bezout-style theorem bounding the number of
isolated points in the set defined by $\psi$ in terms of its
format. The format $\sF$ generally encodes the degrees and magnitude
of the coefficients of the algebraic differential equations defining
the LN-functions appearing in $\psi$. We show moreover that the
structure generated by graphs of unrestricted Pfaffian functions over
$\R_\LN$, denoted $\R_{\LN,\PF}$ is effectively o-minimal as well. In
particular $\R_{\LN,\exp}$ is an effectively o-minimal reduct of
$\R_{\an,\exp}$.

The key to constructing these new effectively o-minimal structures is
a solution to Khovanskii's conjecture from the early eighties on
effective bounds for the number of solutions of systems of Noetherian
equations (and its generalization to the LN class). The proof of this
result is based on a version of the cellular parametrization theorem
of \cite{me:c-cells} for the LN-category. It is crucial for the
inductive proof of this result to consider log-Noetherian functions
and cells even if one is originally only interested in Khovanskii's
conjecture for non-singular Noetherian systems.

We show that horizontal sections of regular flat connections with
quasiunipotent monodromy over algebraic varieties (with appropriate
branch cuts) are definable in $\R_{\LN,\exp}$. In particular we show
that the universal covering map of the Siegel modular variety $\cA_g$
and the universal abelian scheme $\A_g\to\cA_g$, and the period maps
for polarized variations of $\Z$-Hodge structures (PVHS), are all
definable in $\R_{\LN,\exp}$. Thus all applications of o-minimality to
these areas can in principle be made effective using the theory of
this structure. In particular we give an effective Pila-Wilkie theorem
(a straightforward generalization of \cite{me:effective-pfaff-pw} from
the Pfaffian context) and a bound for the degree of the Hodge locus (a
straightforward consequence of the approach of \cite{bkt:tame}).

\subsection{Noetherian functions and Khovanskii's conjecture}

Let $\cP\subset\C^n$ be a product of discs in $\C^n$ or intervals in
$\R^n$ and $f_1,\ldots,f_N:\cP\to\C$ a collection of analytic
functions satisfying a system of polynomial differential equations
\begin{equation}\label{eq:intro-chain}
  \pd{F_i}{\vz_j} = G_{i,j}(F_1,\ldots,F_N).
\end{equation}
In \cite{khovanskii:fewnomials} Khovanskii studied the situation where
$\cP\subset\R^n$, and the equations~\eqref{eq:intro-chain} are
\emph{triangular} in the sense that $G_{i,j}$ depends only on
$F_1,\ldots,F_i$. Functions satisfying this condition are called a
\emph{Pfaffian chain}, and the maximum among the degrees of $G_{i,j}$
is called the degree of the chain. A polynomial combination
$G(F_1,\ldots,F_N)$ of these functions is called a \emph{Pfaffian
  function} of degree $\deg G$. Khovanskii proved that the number of
connected components of a set defined by a system of Pfaffian
equations can be bounded in terms of $n,N$ and the degrees of the
chain and the functions. This holds even if some of the intervals
defining $\cP$ are unbounded.

Without the triangularity condition on~\eqref{eq:intro-chain} the
chain $F_1,\ldots,F_N$ is called a \emph{Noetherian chain} and their
polynomial combinations \emph{Noetherian functions}. This terminology
is due to Tougeron \cite{tougeron:noetherian}. Khovanskii's global
bounds on the number of solutions certainly do not generalize to the
Noetherian context: for instance, $\sin(x)$ is clearly Noetherian on
$\R$ and admits infinitely many isolated roots. Nevertheless Khovanskii
conjectured in the early eighties that non-global bounds --
restricting either to local germs or to compact domains, should still
hold in this class. One instance of this conjecture appeared in
\cite{GabKho}.

\begin{Conj}\label{conj:gab-kho}
  Let $P_1(\vz,\e),\ldots,P_n(\vz,\e):(\C^{n+1},0)\to\C$ be
  holomorphic functions depending analytically on $\e$, such that for
  each fixed $\e$ they are Noetherian of degree at most $D$ over a
  chain $F_1,\ldots,F_N$ of degree $D'$. Then there
  exists $\delta>0$ such that for any sufficiently small $\e$, the
  number of isolated points in
  \begin{equation}
    \{ \vz\in\C^n : \norm{\vz}\le\delta, \quad P_1(\vz,\e)=\cdots=P_n(\vz,\e)=0\}
  \end{equation}
  is bounded by some explicit function of $N,n$ and $D+D'$.
\end{Conj}

In the \emph{complex} Pfaffian case Conjecture~\ref{conj:gab-kho} was
proved by Gabrielov \cite{gabrielov:pfaff-mult}, giving a local
complex analog of Khovanskii's (real) Pfaffian theory
\cite{khovanskii:fewnomials}. In the case $n=1$ the conjecture is
proved in \cite{gabrielov:mult-morse}. In \cite{GabKho} the general
conjecture was proven under the additional assumption that the
intersection for $\e=0$ is proper (in which case the number of common
roots in any deformation as above is bounded by the local intersection
number at the origin). Some result on improper intersections are given
in \cite{me:noetherian-dim2,me:deflicity} under addition assumptions
on the deformation, but Conjecture~\ref{conj:gab-kho} has remained
open in general.

A stronger form of Conjecture~\ref{conj:gab-kho} is to require an
explicit bound on the number of isolated zeros in the domain $\cP$. In
this case it is clear that the bounds must also involve the magnitude
of the coefficients of the equations $G_{i,j}$ in the Noetherian
chain~\eqref{eq:intro-chain}, as evidenced by the Noetherian function
$\sin (Mx)$ with $M\gg1$. Accordingly, define the
\emph{format} $\sF$ of a Noetherian chain $(F_1,\ldots,F_N)$ and
Noetherian function $F:=G(F_1,\ldots,F_N)$ by
\begin{equation}\label{eq:intro-format}
  \begin{aligned}
    \sF(F_1,\ldots,F_N) &:= n+N+\sum_{i,j} \deg G_{i,j} + \norm{G_{i,j}} + \max_{\substack{i=1,\ldots,N\\\vz\in\bar\cP}} |F_i(\vz)| \\
    \sF(F) &:= \sF(F_1,\ldots,F_N) + \deg G + \norm{G}
  \end{aligned}
\end{equation}
where the norm of a polynomial is given by the sum of the absolute
values of its coefficients. With this definition, Khovanskii's
conjecture is as follows.

\begin{Conj}\label{conj:gab-kho-nonlocal}
  Let $P_1,\ldots,P_n:\cP\to\C$ be Noetherian of format $\sF$. Then
  the number of isolated points in
  \begin{equation}
    \{ \vz\in\cP : P_1(\vz)=\cdots=P_n(\vz)=0\}
  \end{equation}
  is bounded by some explicit function of $\sF$.
\end{Conj}

To see that Conjecture~\ref{conj:gab-kho-nonlocal} implies
Conjecture~\ref{conj:gab-kho}, note that by rescaling the time
variable $\tilde\vz=R\vz$ and the chain functions $\tilde F_i=F/r$
with $R\gg r\gg0$ one can arrange that $\norm{G_{i,j}},|F_i(\vz)|$ and
$\norm{G}$ are all bounded by $1$ when considered on the unit polydisc
$\tilde\cP$. One thus obtains an effective bound depending only on
$n,N,\deg G_{i,j}$ and $\deg G$ on the isolated points in a ball whose
radius $\delta$ (in the original $\vz$-coordinates) depends
effectively on $\sF$. Conjecture~\ref{conj:gab-kho} by comparison
asserts the same bound without control on $\delta$.

Conjecture~\ref{conj:gab-kho-nonlocal} has been resolved in the case
$n=1$ by Novikov and Yakovenko \cite{ny:chains}. Their approach is
based on a combination of commutative algebraic techniques and results
from the theory of analytic ODEs (in one variable). We make extensive
use of their results in our treatment of the higher dimensional
case. To our knowledge no higher dimensional case of
Conjecture~\ref{conj:gab-kho-nonlocal} has been known prior to this
work. Some partial results are given in
\cite{me:mult-ops,me:qfol-geometry} but with bounds degenerating to
infinity near the locus of non-proper intersection. It has generally
been clear that
Conjectures~\ref{conj:gab-kho}~and~\ref{conj:gab-kho-nonlocal} are the
main obstacle toward developing a theory parallel to Khovanskii's
Pfaffian theory in the Noetherian context, as the results of the
present paper demonstrate.

\subsection{Log-Noetherian functions (a simplified version)}

In his ICM talk Khovanskii listed as an open problem
\cite[Section~14,~Problem~4]{khovanskii:icm} the goal of extending the
Pfaffian class while retaining the effective bounds, and more
specifically \cite[Section~14,~Problem~5]{khovanskii:icm} the goal of
extending the effective bounds to the period integrals of algebraic
forms over algebraic families. Period integrals are known to satisfy
systems of algebraic differential equations (classically known as
Picard-Fuchs system, and in greater generality Gauss-Manin
connections) which do not seem to be Pfaffian. This is one of the main
motivations for considering the more general class of Noetherian
functions. It should be noted however that Gauss-Manin connections
generally exhibit (regular) singularities, and the Noetherian class
only describes their behavior away from the singular locus. Toward
this end we enlarge the class of Noetherian functions to allow certain
singularities.

Let $\cP^\circ\subset\C^n$ denote the product of $n$ punctured discs
of radius $1$. An LN-chain over $\cP^\circ$ is a system analogous
to~\eqref{eq:intro-chain} but allowing for logarithmic-type
singularities along the divisors $\vz_j=0$, namely
\begin{equation}\label{eq:intro-LN-chain}
  \vz_j \pd{F_i}{\vz_j} = G_{i,j}(F_1,\ldots,F_N).
\end{equation}
We still assume that the functions $F_j$ are holomorphic and bounded
(and thus extend holomorphically over the punctures). Under these
conditions we call the sequence $F_1,\ldots,F_N$ an LN-chain on
$\cP^\circ$. More generally we define LN-chains on more complicated
cellular domains, but we postpone this definition
to~\secref{sec:basic-def}. We resolve Khovanskii's conjecture in the
stronger sense of Conjecture~\ref{conj:gab-kho-nonlocal} for this
wider class of functions. For domains $\cP^\circ$ as above the result
can be stated as follows.

\begin{Thm*}
  Let $P_1,\ldots,P_n:\cP^\circ\to\C$ be LN-functions of format
  $\sF$. Then the number of isolated points in
  \begin{equation}
    \{ \vz\in\cP^\circ : P_1(\vz)=\cdots=P_n(\vz)=0\}
  \end{equation}
  is bounded by some explicit function of $\sF$.
\end{Thm*}

A much more general form of this result is described
in~\secref{sec:intro-RLN}. The proof of this theorem is based on a
cellular parametrization theorem (CPT) in the sense of
\cite{me:c-cells} for the LN-category. This can be seen as a type of
local desingularization result for sets defined using LN
functions. Since the precise statement of the CPT requires some
technical setup we postpone it to~\secref{sec:CPT}.

\subsection{Notation for effectivity}

When speaking of \emph{effective bounds} in this paper, we mean bounds
that are given by primitive-recursive functions that can be fully
written out explicitly in principle. We use the notation
$\eff(\cdots)$ to denote such a primitive recursive function, which
may be different for each occurrence of this notation.

We define the \emph{format} of various objects such as LN-functions,
LN-cells, definable sets in $\R_{\LN}$, etc. To make the notations
less cumbersome, when $X$ is one of these objects we sometimes write
$\eff(X)$ to mean $\eff(\sF)$ where $\sF\in\N$ is the format
associated to $X$.

\subsection{Effectively o-minimal structures}
\label{sec:intro-effective-omin}

Let $\cS$ be an o-minimal structure, viewed as a collection of
definable subsets of $\R^n$ for $n\in\N$. An \emph{effectively
  o-minimal structure} on $\cS$ consists of a collection
$\{\Omega_\sF\subset\cS\}_{\sF\in\N}$ called the \emph{format
  filtration} and a primitive-recursive function $\cE:\N\to\N$ such
that
\begin{enumerate}
\item The filtration is increasing $\Omega_\sF\subset\Omega_{\sF+1}$
  and exhaustive $\cup_\sF \Omega_\sF = \cS$.
\item For every $A,B\subset\R^n$ with $A,B\in\Omega_\sF$, we have
  \begin{equation}
    A\cup B,A\cap B,\R^n\setminus A,A\times B,\pi^n_k(A) \in \Omega_{\sF+1}
  \end{equation}
  where $\pi^n_k:\R^n\to\R^k$ denotes the projection to the first $k$
  coordinates.
\item If $A\subset\R$ and $A\in\Omega_\sF$ then $A$ has at most
  $\cE(\sF)$ connected components.
\end{enumerate}

When we say that a structure is effectively o-minimal in this paper we
always mean that the function $\cE$ can be explicitly presented in
principle, although for brevity we do not compute it exactly. We will
usually treat $\cE$ as implicit and speak simply of an effectively
o-minimal structure $(\cS,\Omega)$.

When a structure is effectively o-minimal one has effective variants
of all the standard theorems of o-minimality such as
cell-decomposition, existence of stratification, definable
triangulation, etc. We state the effective cell decomposition theorem
as a representative example. The proof follows from the standard proof
of cell decomposition verbatim.

\begin{Thm}[Effective cell decomposition]
  Let $X_1,\ldots,X_k\subset\R^n$ be definable subsets. Then there
  exists a cylindrical decomposition $\R^n=\cup_\alpha C_\alpha$
  compatible with $X_1,\ldots,X_k$ such that
  \begin{align}
    \#\{C_\alpha\} &< \eff(X_1,\ldots,X_k), & \forall\alpha: \sF(C_\alpha) &< \eff(X_1,\ldots,X_k).
  \end{align}
\end{Thm}

The notion of effective o-minimality is defined analogously to
\emph{sharp o-minimality} (\so-minimality) introduced recently in
\cite{me:icm2022}, see also
\cite{me:sharp-os-cells,me:pfaff-wilkie}. However unlike in the case
of \so-minimality, here we do not require polynomial growth with
respect to a separate \emph{degree} parameter and are content with
asserting some explicit bound depending on the format. If
$\Omega_{\sF,D}$ is a \so-minimal filtration then $\Omega_{\sF,\sF}$
is essentially an effectively o-minimal filtration (up to some minor
changing of the indices). Berarducci and Servi have considered in
\cite{bs:effective-omin} a similar notion of effective
o-minimality. Our formalism is slightly different but up to some
technical differences it amounts to roughly the same concept.

\begin{Rem}
  In \so-minimality, since the filtration $\Omega_{\sF,D}$
  involves two indices, there is no canonical way of associating a
  single format and degree to a given definable set. In effective
  o-minimality on the other hand one can define the format of a
  definable set $A\in\cS$ to be $\min \{\sF:A\in\Omega_\sF\}$. 
\end{Rem}

If some collection of sets $\{A_\alpha\in\cS\}$ generates the
structure $\cS$ and one assigns a format to each $A_\alpha$, then one
can always define a filtration $\Omega$ generated by these sets by
taking transitive closure under axiom 2 above. The statement that
the resulting structure is effectively o-minimal then means that axiom
3 holds for some suitable choice of $\cE$. 

\subsection{The structures $\R_\LN$ and $\R_{\LN,\exp}$}
\label{sec:intro-RLN}

Turning back to LN-function, we define the structure $\R_\LN$
generated by the graphs of LN-functions (for the precise definition
see~\secref{sec:model-theory}). We associate to the graph of each
LN-function the format of the function as defined above and generate a
format-filtration.

\begin{Thm}\label{thm:RLN-theory}
  The structure $\R_\LN$ is effectively o-minimal and effectively
  model-complete.
\end{Thm}

Theorem~\ref{thm:RLN-theory} is proved
in~\secref{sec:model-theory}. Model completeness is the statement that
every definable set in the structure can be expressed using a purely
existential formula. Together with o-minimality it is one of the
hallmarks of tame geometry, established for $\R_\an$ by Gabrielov
\cite{gab:projections} and for $\R_{\exp}$ by Wilkie
\cite{wilkie:Rexp}. Effective model completeness here means that the
complexity of the existential formula can be effectively controlled in
terms of the complexity of the original formula. We also prove an
effective \Lojas. inequality for $\R_\LN$.

With Theorem~\ref{thm:RLN-theory} in hand, we proceed to show that
Khovanskii's theory of Pfaffian functions can be carried out over the
structure $\R_\LN$. That is, one can consider Pfaffian systems of
equations with coefficients in $\R_\LN$, and add the graphs of the
resulting functions to obtain a larger structure $\R_{\LN,\PF}$. As a
consequence we prove the following.

\begin{Thm}\label{thm:R-LN-PF}
  The structure $\R_{\LN,\PF}$ is effectively o-minimal.
\end{Thm}

Theorem~\ref{thm:R-LN-PF} is proved in~\secref{sec:effective-PF}. As a
special case of particular interest for the applications, it follows
that the structure $\R_{\LN,\exp}$ is effectively o-minimal. We note
that this type of ``Pfaffian closure'' is well known in o-minimality
(e.g. by Wilkie \cite{wilkie:new-omin} and Speissegger
\cite{speissegger:pfaff-closure}), and the effective nature of the
construction has also been studied by Berarducci-Servi
\cite{bs:effective-omin} and Gabrielov-Vorobjov
\cite{gv:complexity-computations}. We give the construction using
their different approaches in~\secref{sec:effective-PF}.

\subsection{Abelian integrals and the Varchenko-Khovanskii theorem}

In the specific case of abelian integrals
\begin{equation}
  I(t) := \oint_{\delta(t)} \omega, \qquad \text{where }
  \begin{cases}
    \omega=P\d x+Q\d y &\\
    H\in\R[x,y]&\\
    \delta(t)\in H_1(\{H=t\},\Z) &
  \end{cases}
\end{equation}
the problem of finding an explicit upper bound for the number of zeros
of $I(t)$ depending only on $\deg H,\deg\omega$ is known as the
\emph{infinitesimal Hilbert sixteenth problem} because of its relation
to the limit cycles of a perturbed Hamiltonian differential equation
\begin{equation}\label{eq:perturbed-H}
  \d H+\e\w=0, \qquad \e\ll1.
\end{equation}
It can be shown that when the perturbation is non-conservative, the
zeros of $I(t)$ correspond to the limit cycles of the perturbed
system~\eqref{eq:perturbed-H}.

The infinitesimal Hilbert problem has been resolved in
\cite{me:inf16}, and we refer the reader to this paper for a review of
the long history of this problem. Prior to this effective solution, a
uniform (but ineffective) bound depending only on the
$\deg H,\deg\omega$ has been obtained by Varchenko and Khovanskii in
the papers \cite{varchenko:finiteness,asik:finiteness}. Their proof
predates modern o-minimality, but can be described as showing that the
period integrals of algebraic families lie in the structure
$\R_{\an,\PF}$ generated by Pfaffian functions over the globally
subanalytic structure.

Going beyond the original infinitesimal Hilbert problem, Varchenko
\cite[Section~4]{varchenko:finiteness} also proves a
``Bezout'' theorem for period integrals, showing that systems of
equations in several variables involving period integrals over
algebraic families of arbitrary dimension admit finite and uniform
bounds for the number of isolated solutions. No effective extension of
\cite{me:inf16} to this multivariable context has been known.

\subsection{Periods and other examples of LN functions}

In~\secref{sec:regular-conn} we show that horizontal sections of
regular flat connections with quasiunipotent monodromy over algebraic
varieties lie in $\R_{\LN,\exp}$ (after making appropriate branch
cuts). Applying this result to the Gauss-Manin connection of an
algebraic family we obtain an effective form of the
Varchenko-Khovanskii theorem from the effective o-minimality of
$\R_{\LN,\exp}$. This also implies that covering maps of Shimura
varieties and period maps for PVHS are definable in $\R_{\LN,\exp}$ in
the sense of \cite{bkt:tame}.

More specifically, suppose that $S$ is a quasi-projective variety and
$\Phi:S\to D/\Gamma$ is a period map associated to a PVHS
(see~\secref{sec:period-maps} for details). If $Y\subset D/\Gamma$ is
a special subvariety then a theorem by Cattani, Deligne and Kaplan
\cite{cdk:hodge-locus} states that $S_Y:=\Phi^{-1}(Y)$ is
algebraic. An alternative proof of this theorem has been given in
\cite{bkt:tame}. Recall that \cite[Theorem~1.1]{bkt:tame} introduces
an $\R_\alg$-structure on $D/\Gamma$ such that every special
subvariety $Y\subset D/\Gamma$ is $\R_\alg$-definable. Below we
understand $\sF(Y)$ to be taken with respect to this structure with
the standard effectively o-minimal structure on $\R_\alg$. To deduce
the algebraicity of $S_Y$, \cite{bkt:tame} proves that $\Phi$ is
definable in $\R_{\an,\exp}$ and appeals to the definable Chow theorem
\cite{ps:complex-omin}. We show that $\Phi$ is actually definable in
$\R_{\LN,\exp}$ (although see Remark~\ref{rem:period-map-format} for a
discussion of issues around the computation of the format of
$\Phi$). Modulo effective computation of the format $\sF(\Phi)$ we
deduce the following more effective form of the algebraicity of the
Hodge locus.

\begin{Thm}\label{thm:hodge-locus-bound}
  Let $Y\subset D/\Gamma$ be a special subvariety. Then $S_Y:=\Phi^{-1}(Y)$
  is an algebraic subvariety of $S$ and
  \begin{equation}
    \deg S_Y = \eff(\Phi,Y)
  \end{equation}
  where $\deg$ denotes the sum of the degrees of the pure-dimensional
  parts of $S_Y$.
\end{Thm}

We also show in~\secref{sec:universal-abelian} that the universal
covering map for the universal abelian scheme over the Siegel modular
variety $\A_g\to\cA_g$ is definable in $\R_{\LN,\exp}$ when restricted
to an appropriate fundamental domain. Combined with the effective
Pila-Wilkie statements in~\secref{sec:effective-pw} this allows one in
principle to effectivize (at least the point counting aspect of) the
many applications of the Pila-Wilkie theorem to unlikely intersection
problems such as (relative) Manin-Mumford, Andr\'e-Oort, and
Zilber-Pink. We refer the reader to \cite{pila:book,zannier:book} for
an introduction to the vast literature around this topic.

In a related direction, \cite[Section~6]{gsv:quantum-complexity} gives
numerous examples of functions arising in quantum systems that are
shown to be Noetherian, at least when considered away from their
singularities. The question is posed there whether these functions in
fact live in an \so-minimal structure, which would allow one to
quantify their complexity. While we do not establish \so-minimality of
$\R_{\LN,\exp}$ in this paper, it seems that effective o-minimality
could provide a replacement suitable for the purposes of
\cite{gsv:quantum-complexity} and containing many of the examples
considered there.

Finally we note that while the examples above mainly come from linear
systems, the Noetherian class contains the solutions of essentially
arbitrary non-linear systems of ODEs in one or several variables, at
least away from their singularities. The theory of $\R_\LN$ can
therefore be used for the quantitative study of geometric complexity
in a wide range of examples arising for instance in classical
mechanics.

\subsection{Conjecture on \so-minimality}

In \cite[Conjecture~29]{me:icm2022} we conjectured that the structure
generated by restricted Noetherian functions is \so-minimal. We also
conjectured \cite[Conjecture~33]{me:icm2022} that the structure
generated by Q-functions, which are horizontal sections of regular
flat connections (defined over $\bar\Q$) with quasiunipotent
monodromy, is \so-minimal. Since $\R_{\LN,\exp}$ contains both of
these structures, the results of the present paper make a step toward
these conjectures by establishing effective o-minimality. It seems
natural to extend the conjecture to the larger structures considered
in this paper as follows.

\begin{Conj}\label{conj:big-sharp-omin}
  The structure $\R_{\LN,\PF}$ is \so-minimal.
\end{Conj}

It seems likely in light of the material developed
in~\secref{sec:gab-effectivity} and \cite{bv:rest-pfaff} that
Conjecture~\ref{conj:big-sharp-omin} would follow from the
\so-minimality of $\R_\LN$, though the technical details of this
reduction are not yet fully verified.

\section{Log-Noetherian cells and functions}

\subsection{Basic definition}
\label{sec:basic-def}

We introduce the notion of a \emph{Log-Noetherian (LN) cells}
$\cC\subset\C^\ell$ and \emph{Log-Noetherian (LN) functions} on $\cC$
by induction on $\ell$. Denote the set of LN functions on $\cC$ by
$\cO_\LN(\cC)$.

In the case $\ell=0$, we consider $\C^0$ to be a singleton and the
only cell $\cC\subset\C^0$ is $\C^0$ itself. An LN function is any
function $F:\cC\to\C$.

For $r\in\C$ (resp. $r_1,r_2\in\C$) with $|r|>0$
(resp. $|r_2|>|r_1|>0$) we denote
\begin{equation}\label{eq:basic-fibers}
  \begin{aligned}
    D(r)&:=\{|z|<|r|\} & D_\circ(r)&:=\{0<|z|<|r|\} \\
    A(r_1,r_2)&:=\{|r_1|<|z|<|r_2|\} & *&:=\{0\}.
  \end{aligned}
\end{equation}
We also set $S(r):=\partial D(r)$.

Recall the following notation from \cite{me:c-cells}. Let $\cX,\cY$ be
sets and $\cF:\cX\to2^\cY$ be a map taking points of $\cX$ to subsets
of $\cY$. Then we denote
\begin{equation}\label{eq:odot-def}
  \cX\odot\cF := \{(x,y) : x\in\cX, y\in\cF(x)\}.
\end{equation}
In this paper $\cX$ will be a subset of $\C^n$ and $\cY$ will be
$\C$. If $r:\cX\to\C\setminus\{0\}$ then for the purpose of this
notation we understand $D(r)$ to denote the map assigning to each
$x\in\cX$ the disc $D(r(x))$, and similarly for $D_\circ,A$.

We now introduce the notions of LN-cells and LN-functions. We denote
the ring of LN-functions on a cell $\cC$ by $\cO_\LN(\cC)$. Note that
these two definitions are by mutual induction on the length of the
cell.

\begin{Def}
  A cell of length $\ell=0$ is the singleton $\cC=\C^0$. An LN cell
  $\cC\subset\C^{\ell+1}$ of length $\ell+1$ is a set of the form
  \begin{equation}
    \cC:=\cC_{1..\ell}\odot\cF
  \end{equation}
  where $\cC_{1..\ell}\subset\C^\ell$ is an LN cell of length $\ell$,
  the \emph{fiber} $\cF$ is one of
  \begin{equation}
    \cF := *,D(r),D_\circ(r),A(r_1,r_2)
  \end{equation}
  and the radii are as follows:
  \begin{itemize}
  \item If $\cF=D(r),D_\circ(r)$ then $r\in\cO_\LN(\cC_{1..\ell})$ with
    $r:\cC_{1..\ell}\to\C^*$.
  \item If $\cF=A(r_1,r_2)$ then $r_1,r_2\in\cO_\LN(\cC_{1..\ell})$
    with $r_1,r_2:\cC_{1..\ell}\to\C^*$ and $|r_1|<|r_2|$ pointwise on
    $\cC_{1..\ell}$
  \end{itemize}
  Parting with \cite{me:c-cells}, we require that $r$ in the case of
  $D(r)$ be a \emph{constant functions}.
\end{Def}

We now define the \emph{standard derivatives} $\partial^\cC_j$ on
$\cC$. Define $\partial^\cC_j$ to be $\partial^{\cC_{1..\ell}}_j$ if
$j=1,\ldots,\ell$ and
\begin{equation}
  \partial^\cC_{\ell+1} :=
  \begin{cases}
    r\pd{}{\vz_{\ell+1}} & \text{if } \cF=D(r) \\
    \vz_{\ell+1} \pd{}{\vz_{\ell+1}} & \cF=D_\circ(r),A(r_1,r_2). \\
    0 & \cF=*.
  \end{cases}
\end{equation}
\begin{Rem}
  If we had allowed $r$ in the fibers $D(r)$ to be non-constant then
  the fields $\partial^\cC_j$ would be non-commuting. For the correct
  general definition, note that every cell is biholomorphic to a cell
  where all fibers of type $D(r)$ are taken to be $D(1)$ by a linear
  rescaling map $z\to z/r$. We could define $\partial^\cC_j$ for
  the general case by pulling back from this standard model. However
  this makes the fields harder to visualize, so we prefer to stick
  to the standard model in our definition of LN cells.
\end{Rem}

An \emph{LN chain} on $\cC$ is a collection of bounded
holomorphic functions $F_1,\ldots,F_N:\cC\to\C$ such that the ring
$\C[F_1,\ldots,F_N]$ is closed under the standard derivatives; that
is, such that
\begin{equation}\label{eq:LN-chain}
  \partial^\cC_j(F_i) = G_{i,j}(F_1,\ldots,F_N)
\end{equation}
for suitable polynomials $G_{i,j}$ over $\C$. An \emph{LN function}
$F:\cC\to\C$ is a function of the form $G(F_1,\ldots,F_N)$ for some
polynomial $G$ and some LN chain $F_1,\ldots,F_N$ on $\cC$. This
finishes the recursive definition of LN-cells and LN-functions.
Whenever we speak of cells below we will implicitly mean LN cells,
unless explicitly stated otherwise.

\begin{Rem}
  One may be tempted to use a simpler definition for LN (or
  Noetherian) functions, namely to work with functions $F:\cC\to\C$
  such that the ring generated from $\C[F]$ by closing under the
  $\partial_j^\cC$ derivatives is finitely generated. However this
  does not lead to a good notion, as illustrated
  in~\secref{sec:LN-func-not-fin-gen}.
\end{Rem}

\subsection{Format}
\label{sec:LN-format}

We introduce a measure for the complexity an LN-cell $\cC$ and an
LN-function $F$. If $\ell=0$ then the format of $\sF(\cC)$ is $1$, and
the format $\sF(f)$ of any (constant) function on $f:\cC\to\C$ is the
least integer upper bound for $|f|$. In the notations
of~\secref{sec:basic-def}, for general $\cC\subset\C^{\ell+1}$,
\begin{equation}
  \begin{gathered}
    \sF(\cC_{1..\ell}\odot *),1+\sF(\cC_{1..\ell}). \\
    \sF(\cC_{1..\ell}\odot D_\circ(r)),\sF(\cC_{1..\ell}\odot D(r))=1+\sF(\cC_{1..\ell})+\sF(r) \\
    \sF(\cC_{1..\ell}\odot A(r_1,r_2))=1+\sF(\cC_{1..\ell})+\sF(r_1)+\sF(r_2).
  \end{gathered}
\end{equation}
We define the norm of a polynomial $\norm{P}$ to be the sum of the
absolute values of its coefficients. If $F_1,\ldots,F_N$ is an
LN-chain as in~\eqref{eq:LN-chain} and $F=G(F_1,\ldots,F_N)$ an
LN-function we define their formats by
\begin{equation}\label{eq:format}
  \begin{aligned}
    \sF(F_1,\ldots,F_N) &:= \sF(\cC)+n+N+\sum_{i,j} \deg G_{i,j} + \norm{G_{i,j}} + \sup_{\substack{i=1,\ldots,N\\\vz\in\cC}} |F_i(\vz)| \\
    \sF(F) &:= \sF(F_1,\ldots,F_N) + \deg G + \norm{G}.
  \end{aligned}
\end{equation}
To be perfectly accurate, we take the least integer upper bound for
the numbers above, and the definition above defines a filtration -- so
the actual format is defined by the minimum $\sF$ over all possible
representations of the LN-chain or LN-function.

\subsection{Extensions of cells}

We now introduce the notion of $\delta$ extensions of cells. We
largely follow \cite{me:c-cells} here, but reproduce the definitions
for the convenience of the reader.

For any $0<\delta<1$ we define
the $\delta$-extensions, denoted by superscript $\delta$, by
\begin{equation}\label{eq:fiber-delta-ext}
  \begin{aligned}
    D^\delta(r)&:=D(\delta^{-1}r) & D^\delta_\circ(r)&:=D_\circ(\delta^{-1}r) \\
    A^\delta(r_1,r_2)&:=A(\delta r_1,\delta^{-1}r_2) & *^\delta&:=*.
  \end{aligned}
\end{equation}
We also set $S^\delta(r)=A(\delta r,\delta^{-1}r)$. 

Next, we define the notion of a $\delta$-extension of a cell of length
$\ell$ where $\delta\in(0,1)$.

\begin{Def}\label{def:cell-ext}
  The cell of length zero is defined to be its own
  $\delta$-extension. A cell $\cC$ of length $\ell+1$ admits a
  $\delta$-extension
  \begin{equation}
    \cC^\delta:=\cC_{1..\ell}^{\delta}\odot\cF^{\delta}
  \end{equation}
  if $\cC_{1..\ell}$ admits a $\delta$-extension, and if the function
  $r$ (resp. $r_1,r_2$) involved in $\cF$ admits holomorphic
  continuation to an LN function on $\cC_{1..\ell}^{\delta}$ and
  satisfies $|r(\vz_{1..\ell})|>0$
  (resp. $0<|r_1(\vz_{1..\ell})|<|r_2(\vz_{1..\ell})|$) in this larger
  domain.
\end{Def}

As a shorthand, when say that $\cC^\delta$ is an LN cell we mean that
$\cC$ is an LN cell admitting a $\delta$-extension.

The following is a simple exercise in the definitions.

\begin{Lem}
  If $\cC^\delta$ is an LN cell then $\sF(\cC)\le\sF(\cC^\delta)$.
\end{Lem}

\subsection{Maps and cellular maps}

A map $f:\cC\to\hat\cC$ between two LN cells is called an LN map if
each coordinate is an LN function. The format $\sF(f)$ is defined to
be the sum of the formats of $\cC,\hat\cC$ and the coordinate
functions.

An LN map $f:\cC\to\hat\cC$ between two cells of the same length
$\ell$ is called \emph{cellular} if it has the form
$\vw_j=\phi_j(\vz_{1..j})$ where each $\phi_j$ is an LN function and
$\pd{\phi_j}{\vz_j}$ is non-vanishing for $j=1,\ldots,\ell$.

\begin{Rem}
  Note that in \cite{me:c-cells} cellular maps were defined with the
  additional assumption that $\phi_j$ is a monic polynomial in $\vz_j$.
\end{Rem}

We show in Theorem~\ref{thm:pullback} that the pullback $f^*F$ of an
LN function $F$ under an LN map $f$ is again LN and and
$\sF(f^*F)<\eff(f,F)$, but note that this is not a trivial statement
and in involves a slight shrinking of the domain.

\subsection{The real setting}

We inductively define the notion of a \emph{real} LN cell and a
\emph{real} LN function. A cell $\cC\odot\cF$ is real if $\cC$ is real
and the radii involved in the definition $\cF$ are real LN
functions. The \emph{real part} of the cell denoted $\R_+\cC$ is
defined to be $\R_+\cC\odot\R_+\cF$, where
\begin{equation}
  \begin{aligned}
    \R_+*&:=* &  & & \R_+ D(r) &:= D(r)\cap\R  \\
    \R_+ D_\circ(r) &:= D_\circ(r)\cap \R_{>0} & & & \R_+A(r_1,r_2)&:=A(r_1,r_2)\cap \R_{>0}.
  \end{aligned}
\end{equation}
Note that we took $\R_+D(r)$ to be the whole interval, parting with
the convention from \cite{me:c-cells}, as this seems somewhat more
natural. However the difference is merely technical.

An LN function $f:\cC\to\C$ is called \emph{real} if $\cC$ is real and
$f$ is real on $\R_+\cC$. A map $f:\cC\to\hat\cC$ is called real if
its coordinate functions are real and $f(\R_+\cC)\subset\R_+\hat\cC$.

\subsection{Cellular covers}
We will often be interested in covering a cell by cellular
images of other cells. Toward this end we introduce the following
definition.

\begin{Def}\label{def:cell-cover}
  Let $\cC^\delta$ be an LN cell and
  $\{f_j:\cC_j^{\delta'}\to\cC^\delta\}$ be a finite collection of LN
  cellular maps. We say that this collection is an
  \emph{LN $(\delta',\delta)$-cellular cover} of $\cC$ if
  \begin{equation}
    \cC\subset\cup_j(f_j(\cC_j)).
  \end{equation}
  Similarly, when $f,\cC_j^{\delta'},\cC^\delta$ are all real this is
  called a real cover if
  \begin{equation}
    \R_+\cC\subset\cup_j(f_j(\R_+\cC_j)).
  \end{equation}
  When $(\delta',\delta)$ are clear from the context we will speak
  simply of cellular covers.
\end{Def}

The number of maps $f_j$ in a cellular cover is called the \emph{size}
of the cover.

\begin{Rem}\label{rem:cell-cover-composition}
  We remark that if $\{f_j:\cC_j^{\delta'}\to\cC^{\delta}\}$ is an LN
  (real) cellular cover of $\cC$ and
  $\{f_{jk}:\cC_{jk}^{\delta''/2}\to\cC_j^{\delta'}\}$ is an LN (real)
  cellular cover of $\cC_j$ then $\{f_j\circ f_{jk}\}$ is an LN (real)
  $(\delta'',\delta)$-cover of $\cC$. Note that we need a slightly
  larger extension $\delta''/2$ in the $f_{jk}$ cover because of the
  slight shrinking of domains in Theorem~\ref{thm:pullback}.
\end{Rem}

\subsection{The cellular parametrization theorem}
\label{sec:CPT}

We say that $f\in\cO_\LN(\cC)$ is compatible with $\cC$ if $F$ is
either identically vanishing or nowhere vanishing on $\cC$. If
$f:\cC\to\hat\cC$ is a cellular map and $F\in\cO_\LN(\hat\cC)$ we say
that $f$ is compatible with $F$ if the pullback $f^*F$ is compatible
with $\cC$.

Our main technical tool in this paper is the following cellular
parametrization theorem (CPT) in the LN category. In the analytic and
algebraic categories, this theorem is the main result of
\cite{me:c-cells}.

\begin{Thm}[Cellular Parametrization Theorem -- LN category]
  Let $\cC^\delta$ be an LN cell and
  $F_1,\ldots,F_M\in\cO_\LN(\cC^\delta)$. Then there exists a cellular
  cover $\{f_j:\cC_j^\delta\to\cC^\delta\}$ such that each $f_j$ is
  compatible with each $F_k$. Moreover
  \begin{align}
    \#\{f_j\} &< \eff(F_1,\ldots,F_M,1/(1-\delta)) & &\text{and} & \sF(f_j) &< \eff(F_1,\ldots,F_M).
  \end{align}
  If $\cC,F_1,\ldots,F_M$ are real then the cover can be chosen to be
  real.
\end{Thm}

The CPT is proven in~\secref{sec:cpt-proof}.

\section{Background from commutative algebra and differential
  equations}

In this section we recall some facts about polynomial rings with
derivations and about oscillation of solutions of complex ODEs.

\subsection{Ascending chains of ideals}

Let $R=\C[x_1,\ldots,x_n]$ or $R:=\R[x_1,\ldots,x_n]$ be a polynomial
ring and $D:R\to R$ a derivation of $R$. Explicitly,
\begin{equation}
  D = \sum p_i \pd{}{x_i} \qquad \text{where } p_i\in R.
\end{equation}
For $P\in R$ we write $\sF(P)=\deg P+\norm{P}$ with $\norm{P}$ defined
in~\secref{sec:LN-format} and
\begin{equation}
  \sF(D) := \sum_i \sF(p_i).
\end{equation}
Define an ascending chain of ideals $I_P^k$ by
\begin{equation}
  I_P^0 := (P), \qquad I_P^{k+1} := \left< I_P^k, D I_P^k \right>.
\end{equation}
The following result of Novikov and Yakovenko plays a key role in our
approach.

\begin{Thm}[\protect{\cite{ny:chains}}]\label{thm:ascending-chains}
  The chain $I_P^k$ stabilizes, i.e. $I_P^\ell=I_P^{\ell+1}$ with
  $\ell\le\eff(D,P)$. Moreover,
  \begin{equation}
    D^{\ell+1} P = \sum_{j=0}^\ell c_j D^j P, \qquad c_j\in R
  \end{equation}
  with $\sum \sF(c_j) < \eff(D,P)$.
\end{Thm}

In fact \cite{ny:chains} gives more specific bounds, exponential in
the degrees and triply exponential in the dimension $n$. This result
is used in loc. cit. to bound the number of intersections between a
trajectory of $D$ and the hypersurface defined by $P$ using certain
non-oscillation results for solutions of linear ODEs. We use it for
similar purposes but require slightly different complex estimates
that we recall below.

\subsection{Total variation of argument  for solutions of linear ODEs}

Let $U\subset\C$ be a bounded domain and $f:\bar U\to\C$ be a
holomorphic function. The \emph{total variation of argument}
\cite{ky:rolle}, also called \emph{Voorhoeve index}, of a holomorphic
function $f:U\to\C$ along a piecewise-smooth curve $\gamma\subset U$
is defined by
\begin{equation}
  V_{\gamma}(f):= \frac1{2\pi} \int_{\gamma}|\d\Arg f(z)|.
\end{equation}
The following theorem, a complex analog of the classical result of de
la Vall\'e Poussin \cite{poussin:oscillation}, shows that the total
variation of argument of a solution of a scalar ODE with analytic
coefficients can be explicitly bounded in terms of the order of the
equation and the upper bounds for its coefficients.

\begin{Thm}[\protect{\cite[Corollary~2.7]{yakovenko:functions-and-curves}}]\label{thm:ode-variation-bound}
  Let $\gamma\subset\C$ be either an interval or a circular arc of
  length $\ell$. Suppose $w(t)$ is analytic on $\gamma$ and satisfies
  the equation
  \begin{equation}
    w^{(n)}(t)+c_1(t)w^{(n-1)}(t)+\cdots+c_n(t)w(t) = 0
  \end{equation}
  where $c_j(t)$ are all analytic on $I$ and bounded in absolute value
  by $M>1$. Then
  \begin{equation}
    V_I(w) \le C\cdot \ell\cdot n\cdot M
  \end{equation}
  where $C$ is some absolute constant.
\end{Thm}

We remark that in loc. cit. this is stated for straight lines, but the
proof is easily seen to apply also for circular arcs. Alternatively
one can change variables to turn a circular arc into an interval and
this only affects $M$ by some constant factor.

\subsection{Total variation of argument for LN functions}

We show that the total variation of argument of an LN function along
suitable curves in the domain can be bounded in terms of $\sF(f)$.

Supposes $\cC:=\cC_{1..\ell}\odot\cF$. Fix a point $\vz\in\cC$ and let
$\gamma$ be a circular arc in the fiber $\{\vz\}\odot\cF(\vz)$. Define
the standard length of $\gamma$, denoted $\length(\gamma)$, to be the
length of $\gamma$ with respect to the $\partial^\cC_{\ell+1}$
chart. In other words,
\begin{equation}
  \length(\gamma) =
  \begin{cases}
    \frac{1}r \times (\text{Euclidean length of }\gamma) & \cF=D(r) \\
    \text{Euclidean length of } \log(\gamma) & \cF=D_\circ(r),A(r_1,r_2).
  \end{cases}
\end{equation}
In this situation we have the following.

\begin{Prop}\label{prop:basic-oscillation-bound}
  Let $\cC$ be an LN cell and $f\in\cO_\LN(\cC)$ and
  $\gamma\subset\cC$ as above. Then
  \begin{equation}
    V_\gamma(f) \le \eff(f)\cdot\length(\gamma).
  \end{equation}
\end{Prop}
\begin{proof}
  According to Theorem~\ref{thm:ascending-chains} applied to $f$ and
  the derivation $D:=\partial^\cC_j$ we have a differential
  equation
  \begin{equation}
    D^{N+1} f = \sum_{j=0}^N c_j D^j f, \qquad c_j\in R
  \end{equation}
  where $R$ is the ring generated by our LN chain with $N$, the
  degrees of the $c_j$, and their norms, bounded by $\eff(f)$. Along
  with the bounds on the format of the LN chain, this implies that all
  coefficients $c_j$ are bounded in absolute value by $\eff(f)$
  everywhere in $\cC$. Therefore Theorem~\ref{thm:ode-variation-bound}
  implies that the total variation of argument of $f$ along $\gamma$
  is bounded as claimed, noting that the time parametrization with
  respect to $D$ agrees with our notion of standard length.
\end{proof}

\subsection{Taylor and Laurent domination}

\begin{Def}[Laurent domination]\label{def:domination}
  A holomorphic function $f:A(r_1,r_2)\to\C$ is said to possess the
  $(p,q,M)$ Laurent domination property where $p\le q$ are integers and
  $M\in\R_+$ if its Laurent expansion $f(z)=\sum a_k(z-z_0)^k$
  satisfies the estimates
  \begin{equation}
    \begin{aligned}
      |a_k| r_2^k &< M \max_{j=p,\ldots,q} |a_j| r_2^j, \qquad
      k=q+1,q+2,\ldots \\
      |a_k| r_1^k &< M \max_{j=p,\ldots,q} |a_j| r_1^j, \qquad
      k=p-1,p-2,\ldots
    \end{aligned}  
  \end{equation}
  If $f:D(r)\to\C$ we formally put $r_1=0$ above, and refer to the
  $(0,q,M)$ Laurent domination property as the $(q,M)$ Taylor
  domination property.
\end{Def}

In the class of LN functions we have the following effective form of
Taylor and Laurent domination.

\begin{Prop}\label{prop:noetherian-domination}
  Let $\cC:=\cC_{1..\ell}\odot\cF$ be an LN cell and
  $f\in\cO_\LN(\cC^\delta)$. Then $f(\vz_{1..\ell},\cdot)$ has the
  $(p,q,M)$ Laurent domination property for every
  $\vz_{1..\ell}\in\cC_{1..\ell}$ with $|p|+|q|+M<\eff(f)$.
\end{Prop}
\begin{proof}
  This is essentially already contained in
  \cite[Corollary~67]{me:c-cells}. The constants there depend on the
  constant in \cite[Lemma~66]{me:c-cells}, which is derived directly
  from the total variation of argument of $f$ along the boundary of a
  certain disc $D$, which is easily seen to have effectively bounded
  standard length. The claim thus follows from
  Proposition~\ref{prop:basic-oscillation-bound}. Also note that since
  LN functions are always bounded, if $\cF=D_\circ(r)$ then $f$
  automatically extends to a holomorphic function over
  $\cC_{1..\ell}\odot D(r)$ (although not as an LN function) so we
  have in fact Taylor domination in this case.
\end{proof}

Recall the following simple consequence of the domination property.

\begin{Prop}[\protect{\cite[Proposition~68]{me:c-cells}}]\label{prop:domination-residue-bound}
  Let $f:C^\delta\to\C$ where $C=D(r),D_\circ(r)$ or $A(r_1,r_2)$ and assume
  that it has the $(p,q,M)$ domination property. Then for the Taylor or
  Laurent expansion
  \begin{equation}
    f(z) = \sum_{j=p}^q a_j z^j + R(z)
  \end{equation}
  we have for every $z\in C$ a bound
  \begin{equation}
    |R(z)| \le \frac{2\delta M}{1-\delta} \max_{j=p,\ldots,q} |a_j z^j|.
  \end{equation}
\end{Prop}

We also record two basic propositions on Taylor domination that will
be useful later.

\begin{Prop}\label{prop:domination-derivative}
  Suppose
  \begin{equation}
    f:D(1)\to\C, \qquad f(z)=\sum a_j z^j
  \end{equation}
  has the $(q,M)$ Taylor domination property
  with $q>0$, and
  \begin{equation}
    |a_q|\ge \max_{j=0,q-1} |a_j|.
  \end{equation}
  Then $f'(z)$ has the $(q-1,M)$ domination property on $D(1/2)$.
\end{Prop}
\begin{proof}
  For every $k>q$ we have $|a_k|<M|a_q|$ by assumption. Then
  \begin{equation}
    f'(z) = \sum k a_k z^{k-1}.
  \end{equation}
  One easily checks that for $k>q$ we have $k/q\le 2^{k-q}$ so
  \begin{equation}
    k |a_k| (1/2)^k \le M\cdot q |a_q| (1/2)^q
  \end{equation}
  which proves the claim.
\end{proof}

\begin{Prop}\label{prop:domination-no-zeros}
  Suppose $f:D(1)\to\C^*$ has the $(q,M)$ Taylor domination property
  and $\delta>0$. Then there exists $\e=\e(q,M,\delta)$ with
  $1/\e=\eff(q,M,1/\delta)$ such that $f\rest{D(\e)}$ has the
  $(0,\delta)$-domination property.
\end{Prop}
\begin{proof}
  For each $j,k=0,\ldots,q$ consider the radius
  \begin{equation}
    r_{j,k} := \sqrt[k-j]{\frac{a_j}{a_k}}
  \end{equation}
  where $|a_jz^j|=|a_k z^k|$. Outside the annulus
  $A_{j,k}=A(r_{j,k}/4,4r_{j,k})$ the ratio $|a_jz^j|/|a_k z^k|$ is
  either smaller than $1/4$ or larger than $4$. Thus for $z$ outside
  all of these annuli, it is easy to see that for the maximal among these
  terms $a_m z^m$ we have
  \begin{equation}
    |a_m z^m| > 2 \sum_{m\neq j=0,\ldots,q} |a_j z^j|.
  \end{equation}
  By Proposition~\ref{prop:domination-residue-bound} we also have that
  if $|z|<1/(2+8M)$ then
  \begin{equation}
    |a_m z^m| > 2 |R(z)|.
  \end{equation}
  Choose $z$ satisfying $|z|<1/(2+8M)$ and $z\not\in A_{j,k}$ for
  every $j,k$. Then $a_m z^m$ dominates all other terms in the Taylor
  expansion on the circle of radius $|z|$ and it follows by Rouch\'e's
  theorem that $f$ has $m$ zeros in the disc $D(|z|)$. Hence by
  assumption $m=0$, and $f$ has the $(0,M)$ domination property in
  $D(|z|)$. It is clear that $|z|$ can be chosen with absolute value
  $r$ satisfying $1/r<\eff(q,M)$. Setting now $\e=\delta r/M$ we get
  the $(0,\delta)$-domination property on $D(\e)$.
\end{proof}

\section{Basic theory of LN functions}

In this section we establish some theory concerning LN functions on
LN cells.

\subsection{Fundamental lemmas from hyperbolic geometry}

For any hyperbolic Riemann surface $X$ we denote by
$\dist(\cdot,\cdot;X)$ the hyperbolic distance on $X$. We use the same
notation when $X=\C$ and $X=\R$ to denote the usual Euclidean
distance, and when $X=\C P^1$ to denote the Fubini-Study metric
normalized to have diameter $1$. For $x\in X$ and $r>0$ we denote by
$B(x,r;X)$ the open $r$-ball centered $x$ in $X$. For $A\subset X$ we
denote by $B(A,r;X)$ the union of $r$-balls centered at all points of
$A$.

The notion of $\delta$-extension is naturally associated with the
Euclidean geometry of the complex plane. However, in many cases it is
more convenient to use a different normalization associated with the
hyperbolic geometry of our domains. For any $0<\rho<\infty$ we define
the $\hrho$-extension $\cF^\hrho$ of $\cF$ to be $\cF^\delta$ where
$\delta$ satisfies the equations
\begin{equation}\label{eq:rho-ext-def}
  \begin{aligned}
    \rho &= \frac{2\pi\delta}{1-\delta^2} && \text{for $\cF$ of type
      $D$,} \\
    \rho &= \frac{\pi^2}{2|\log\delta|} && \text{for $\cF$ of type
      $D_\circ,A$}.
  \end{aligned}
\end{equation}

The motivation for this notation comes from the following fact,
describing the hyperbolic-metric properties of a domain $\cF$ within
its $\hrho$-extension.

\begin{Fact}[\protect{\cite[Fact~6]{me:c-cells}}]\label{fact:boundary-length}
  Let $\cF$ be a domain of type $A,D,D_\circ$ and let $S$ be a component
  of the boundary of $\cF$ in $\cF^\hrho$. Then the length of $S$ in
  $\cF^\hrho$ is at most $\rho$.
\end{Fact}

We define the $\hrho$-extension $\cC^\hrho$ by analogy with the
$\delta$-extension,
$(\cC\odot\cF)^\hrho:=\cC^\hrho\odot\cF^\hrho$. The lemmas below are
more natural to state in terms of the $\hrho$-extension. However,
since in the present paper we only care about effectivity of the
constants and make no attempt to optimize the asymptotic dependence on
parameters, it will make essentially no difference which type of
extension we use. We have mostly kept to $\delta$-extensions to
simplify the presentation.

Recall the three fundamental lemmas from \cite{me:c-cells}.

\begin{Lem}[Fundamental Lemma for $\D$]\label{lem:fund-D}
  Let $\cC^\hrho$ be a complex cell. Let $f:\cC^\hrho\to\D$ be
  holomorphic. Then
  \begin{equation}
    \diam(f(\cC);\D)=O_\ell(\rho).
  \end{equation}
\end{Lem}

\begin{Lem}[Fundamental Lemma for $\D\setminus\{0\}$]\label{lem:fund-Dcirc}
  Let $\cC^\hrho$ be a complex cell and $0<\rho<1$. Let
  $f:\cC^\hrho\to\D\setminus\{0\}$ be holomorphic. Then one of the
  following holds:
  \begin{align}
    f(\cC)&\subset B(0,e^{-\Omega_\ell(1/\rho)};\C) & &\text{or} & \diam(f(\cC);\D\setminus\{0\})&=O_\ell(\rho).
  \end{align}
  In particular, one of the following holds:
  \begin{align}
    \log |f(\cC)| &\subset (-\infty,-\Omega_\ell(1/\rho)) & &\text{or} & \diam(\log|\log|f(\cC)||;\R)&=O_\ell(\rho).
  \end{align}
\end{Lem}

\begin{Lem}[Fundamental Lemma for $\C\setminus\{0,1\}$]\label{lem:fund-C01}
  Let $\cC^\hrho$ be a complex cell and let
  $f:\cC^\hrho\to\C\setminus\{0,1\}$ be holomorphic. Then one of the
  following holds:
  \begin{align}\label{eq:fund-C01}
    f(\cC)&\subset B(\{0,1,\infty\},e^{-\Omega_\ell(1/\rho)};\C P^1)& &\text{or} & \diam(f(\cC);\C\setminus\{0,1\})&=O_\ell(\rho).
  \end{align}
\end{Lem}

\subsection{Bounds for standard derivatives}

\begin{Lem}\label{lem:partial-j-bound}
  Let $f\in\cO_\LN(\cC^\delta)$ and $k\in\N$. Then
  \begin{equation}\label{eq:partial-j-bound}
    \norm{(\partial^\cC_i)^k f}  \le \rho^k \norm{f}_{\cC^\delta}, \qquad \rho=O\big(\frac1{1-\delta}\big)
  \end{equation}
  for $j=1,\ldots,\ell$.
\end{Lem}
\begin{proof}
  First, we can assume without loss of generality that there are no
  $D_\circ$ fibers. Indeed, we can replace each $D_\circ(r)$ fiber by
  $A(\e r,r)$ for $\e\ll1$. Since~\eqref{eq:partial-j-bound} is
  independent of $\e$ the claim for the original cell will follow from
  the claim for these modified cells.

  Recall \cite[Definition~39]{me:c-cells} that the skeleton of $\cC$
  is defined as the set of points in $\cC$ where each coordinate is in
  the boundary of its corresponding fiber. By the maximum principle it
  is enough to check the maximum of $\partial^\cC_i$ on the skeleton
  of $\cC$. Suppose it is obtained at a point $p$. Moreover, since
  $\partial^\cC$ is invariant by rescaling of each fiber and by
  inversion $z\to r/z$ in the case of annuli fibers, there is also no
  harm in assuming that $p$ is in the component of the skeleton given
  by $S(1)$ in $D,A$ fibers and by $\{0\}$ in $*$ fibers. Then
  $B_\rho(p)\subset\cC^\delta$ where $\rho$ is as given in the
  statement.

  Let $\phi:(\C,0)\to \cC^\delta$ be the flow chart of
  $\partial^\cC_j$ with $\phi(0)=p$. Then we are interested in
  evaluating $|(f\circ\phi)^{(k)}(0)|$. It is easy to see that $\phi$
  extends to a disc of radius $O(\rho)$, and the claim follows from
  the Cauchy estimates.
\end{proof}

\subsection{Four basic theorems}

In this section we prove four basic theorems for LN functions:
monomialization, restricted division, logarithmic derivation, and
stability under pullback. We will establish all of these by concurrent
induction, so we assume throughout that $\cC^\delta$ is an LN cell of
length $\ell+1$ and that all statements have been established for
cells of length $\ell$.

\subsubsection{Monomialization}

We state an analog of the monomialization lemma of \cite{me:c-cells},
with effective control over the constants in terms of $\sF(f)$. We
begin by recalling some notation.

A cell $\cC$ is homotopy equivalent to a product of points (for fibers
$*,D$) and circles (for fibers $D_\circ,A$) by the map
$\vz_i\to\Arg\vz_i$. Thus $\pi_1(\cC)\simeq\prod G_i$ where $G_i$ is
trivial for $*,D$ and $\Z$ for $D_\circ,A$. Let $\gamma_i$ denote the
generator of $G_i$ chosen with positive complex orientation for
$G_i=\Z$ and $\gamma_i=e$ otherwise.

\begin{Def}\label{def:assoc-monom}
  Let $f:\cC\to\C^*$ be continuous. Define the monomial associated
  with $f$ to be $\vz^{\valpha(f)}$ where
  \begin{equation}
    \valpha_i(f) = f_*\gamma_i \in \Z\simeq \pi_1(\C^*).
  \end{equation}
\end{Def}

It is easy to verify that $f\mapsto z^{\valpha(f)}$ is a group homomorphism
from the multiplicative group of continuous maps
$f:\cC\to\C^*$ to the multiplicative group of monomials,
which sends each monomial to itself.

\begin{Thm}\label{thm:monomialization}
  Let $f:\cC^\delta\to\C^*$ with $f\in\cO_\LN(\cC^\delta)$. Then we
  have a decomposition $f=m_f(\vz) U(\vz)$ where
  \begin{equation}
    m_f(\vz) = \lambda \vz^{\valpha(f)}, \qquad \lambda\in\R_{>0},
  \end{equation}
  and the branches of $\log U:\cC^\delta\to\C$ are univalued. We have
  \begin{equation}
    \norm{\valpha(f)}+\sF(U\rest{\cC}) \le \eff(f),
  \end{equation}
  and for all $\vz\in\cC$,
  \begin{equation}
    \frac{1}{\eff(f)} \le |U(\vz)|\le 1.
  \end{equation}
  If $f$ is real then $U$ is real.
\end{Thm}

\begin{proof}
  We assume that the fiber is of type $A(r_1,r_2)$, the other cases
  being similar. We begin by bounding $\valpha(f)_{\ell+1}$ by
  $\eff(f)$. Let $\gamma$ denote the curve given by a circle of radius
  $r_2$ in the $\vz_{\ell+1}$-coordinate and some point
  $p\in\cC_{1..\ell}$ in the $\vz_{1..\ell}$ coordinates. Then
  $\gamma\subset\cC^\delta$, and clearly $\valpha(f)_{\ell+1}$ is
  bounded by the total variation of argument of $f$ along
  $\gamma$. This circle has standard length $2\pi$, and the bound
  therefore follows by Theorem~\ref{prop:basic-oscillation-bound}.

  Now consider the cell $\hat\cC:=\cC_{1..\ell}\odot*$ and the
  cellular map $\phi:\hat\cC^{\delta}\to\cC^\delta$ given by
  \begin{equation}
    \phi(\vz_{1..\ell}) = (\vz_{1..\ell},r_2(\vz_{1..\ell})).
  \end{equation}
  By pullback stability (Theorem~\ref{thm:pullback}), the function
  $\hat f:=\phi^* f$ is LN on the cell $\hat\cC^{\sqrt\delta}$ with
  effectively bounded format. By induction we have an effective bound
  on $\norm{\alpha(\hat f)}$. By a similar reasoning we have an
  effective bound on $\norm{\alpha(r_2)}$. We remark that even though
  $\hat\cC$ is formally of length $\ell+1$, in the proof of pullback
  stability one can identify $\hat\cC$ with $\cC_{1..\ell}$ so our
  induction is in fact well-founded.

  In the notation of Definition~\ref{def:assoc-monom}, the map $\phi$
  takes the fundamental loop $\gamma_i\in\pi_1(\hat\cC)$ to
  $\gamma_i+\alpha(r)_i\cdot\gamma_{\ell+1}$ in $\pi_1(\cC)$. Thus
  \begin{equation}
    \alpha(\hat f)_i = \alpha(f)_i+\alpha(r)_i\cdot \alpha(f)_{\ell+1}. 
  \end{equation}
  Solving for $\alpha(f)_i$ we see that it is also effectively
  bounded.

  We proceed with the bounds for the unit $U$. Fix a branch of
  $\log U$. We will show that the diameter of $\log U\subset\C$ is
  effectively bounded. One can then choose the (positive real)
  constant $\lambda$ to normalize its maximum to $1$. The bound on
  $\sF(U)$ follows from the restricted division theorem
  (Theorem~\ref{thm:restricted-div}) to be proved later.

  First apply the inductive hypothesis to $\hat U:=\phi^*U$. Note that
  $\alpha(\hat U)=0$, so this indeed plays the role of a unit in the
  induction. Thus $\log \hat U$ has diameter bounded by $\eff(f)$. We
  will thus finish the proof by showing that for any
  $\vz_{1..\ell}\in\cC$ we have a bound for the diameter of $\log U$
  in the fiber over $\vz_{1..\ell}$. Set $r_1=r_1(\vz_{1..\ell})$ and
  $r_2=r_2(\vz_{1..\ell})$ and
  \begin{equation}
    U':A(r_1,r_2)^\delta\to\C^*, \qquad U'(t) = U(\vz_{1..\ell},t).
  \end{equation}
  By Theorem~\ref{prop:basic-oscillation-bound} we have an effective
  bound for the total variation of argument of $U'$ along the two
  boundary components of $A(r_1,r_2)^{\sqrt\delta}$.

  Let $L:=\log U'$ and consider $X:=L(A(r_1,r_2)^{\sqrt\delta})$. By the open
  mapping theorem
  \begin{equation}
    \partial X \subset L(\{|t|=\sqrt\delta r_1\}) \cup L(\{|t|=r_2/\sqrt\delta\}).
  \end{equation}
  In particular,
  \begin{equation}
    \diam \Im X \le \diam \Im L(\{|t|=\sqrt\delta r_1\}) + \diam \Im L(\{|t|=r_2/\sqrt\delta\})
  \end{equation}
  where $\Im$ denotes the imaginary part. The two summands on the
  right hand side are bounded by the two total variations of argument
  above, so $\diam \Im X$ is effectively bounded, say by some constant
  $V$.

  Now we can finish as in \cite{me:c-cells}. Set
  $\tilde U:=(U')^{1/(4V)}$ and assume further, by multiplying by a
  scalar, that the image of $\tilde U$ contains $i$. Then the total
  variation of argument of $\tilde U$ is bounded by $\pi/4$, and it
  follows that $\tilde U:A(r_1,r_2)^{\sqrt\delta}\to\H$. Then the
  fundamental lemma for $\D$ (applied in this case to the upper half
  space model) implies that the diameter of $\tilde U(\cC)$ is bounded
  by a constant depending only on $\sqrt\delta$. Recovering $U'$ as
  $\tilde U^{4V}$ and recalling that $V$ was effectively bounded
  finishes the proof.
\end{proof}

\subsubsection{Restricted division}

The following theorem plays a fundamental role in our approach as it
allows us to form an analog of the strict transform used in resolution
of singularities.

\begin{Thm}\label{thm:restricted-div}
  Let $\cC^\delta$ be an LN cell and
  $f,g\in\cO_\LN(\cC^\delta)$. Suppose further that
  \begin{equation}
    g:\cC^\delta\to\C^*
  \end{equation}
  and that $f/g$ is bounded in
  $\cC^\delta$. Then
  \begin{align}
    \frac f g &\in\cO_\LN(\cC) & &\text{and} & \sF(\big(\frac f g\big)\rest{\cC}) &< \eff(f,g,\norm{f/g}_{\cC^\delta}).
  \end{align}
  If $f,g$ are real then $f/g$ is real.
\end{Thm}

\begin{proof}
  We begin with the case where $g=\lambda \vz^\valpha$. Write $\partial_j:=\partial^{\cC^\delta}_j$. By
  Theorem~\ref{thm:ascending-chains} there exists $N=\eff(f)$ such
  that for each $j=1,\ldots,\ell$ there is an equation
  \begin{equation}
    \partial^{N+1}_j f = \sum_{i=0}^N c_{j,i} \partial_j^i f
  \end{equation}
  where $c_{j,i}\in\cO_\LN(\cC^\delta)$ have format $\eff(f)$. We
  claim that adding the functions
  \begin{equation}
    (\partial^\vbeta f)/g \text{ for every
    } \vbeta\in\{0,\ldots,N-1\}^{\ell+1}
  \end{equation}
  to the LN chain defining $f,g$ and $c_{j,i}$ gives an
  LN chain of effectively bounded format on $\cC$.

  We start by showing that these new functions are effectively bounded
  on $\cC$. For this, note that
  \begin{equation}\label{eq:division-chain-step}
    \partial_j (\partial^\vbeta f/g) = (\partial^{\beta+e_j} f)/g+\partial^\vbeta f \partial_j(1/g)=
    (\partial^{\vbeta+e_i} f)/g-\valpha_j (\partial^\vbeta f/g)
  \end{equation}
  where $e_j$ is the $j$-th standard basis vector. Here we use the
  fact that when the $j$-th fiber is of type $D,*$ we always have
  $\valpha_j=0$. From this it easily follows by induction that
  \begin{equation}\label{eq:division-chain}
    \partial^\vbeta(f/g) = (\partial^\vbeta f)/g+\sum_{\vbeta'<\vbeta} a_{\beta'} (\partial^{\vbeta'} f)/g
  \end{equation}
  where $\vbeta'<\vbeta$ means $\vbeta_i\le\vbeta_i$ for every $i$
  and $\vbeta\neq\vbeta'$. Here $a_\vbeta$ are effectively bounded
  integer coefficients computed in terms of $\valpha$.
  
  By Lemma~\ref{lem:partial-j-bound}, the norm of the right hand side
  of \eqref{eq:division-chain} is effectively bounded in terms of the
  norm $\norm{f/g}_{\cC^\delta}$, and by induction on $\vbeta$ with
  respect to the ordering above we then obtain effective bounds
  for $\norm{(\partial^\vbeta f)/g}_\cC$.
  
  We now show how to write an LN chain for our
  functions. From~\eqref{eq:division-chain-step} we already have an
  expression for $\partial_j(\partial^\vbeta f/g)$ in terms of our
  LN chain with effectively bounded coefficients, unless
  $\vbeta_j=N$. In this case write $\vbeta=Ne_j+\hat\vbeta$ and compute
  \begin{multline}
    \partial_j(\partial^\vbeta f/g)=(\partial^{\hat\vbeta}\partial_j^{N+1}f)/g-\valpha_j (\partial^\vbeta f/g) =\\
    \big( \partial^{\hat\vbeta} \sum_{i=0}^N c_{j,i} \partial_j^i f \big)/g -\valpha_j (\partial^\vbeta f/g) =\\
    -\valpha_j (\partial^\vbeta f/g) + \sum_{i=0}^N c_{j,i}  \partial^{\hat\vbeta+ie_j} f/g.
  \end{multline}
  The expression in the right hand side is a polynomial in our LN
  chain with effectively bounded coefficients, thus finishing the
  proof in the case $g=\lambda\vz^\valpha$. Note that the proof in
  this case does not rely on the monomialization lemma.

  In the general case, write $g=m_g U$ using the monomialization
  lemma. Since $U$ on $\cC^{\sqrt\delta}$ is effectively bounded above
  and below in terms of $\sF(g)$, we have an effective bound
  \begin{equation}
    \norm{f/m_g}_{\cC^{\sqrt\delta}} < \eff(f,g,\norm{f/g}_{\cC^\delta}).
  \end{equation}
  By the preceding case, $f/m_g$ is LN with the desired format
  bound. The proof will be finished if we show that $1/U$ is also
  LN with effectively bounded format on $\cC$. For this we add
  $1/U$ to our LN chain, and note that
  \begin{equation}
    \partial_j(1/U) = -\partial_j(U) (1/U)^2
  \end{equation}
  and this gives an LN chain of bounded format because we
  indeed have an effective upper bound for $\norm{1/U}_\cC$.

  When $f,g$ are real then clearly $f/g$ is real on $\R_+\cC$, and the
  chain constructed above is readily seen to be real.
\end{proof}

\begin{Rem}
  Theorem~\ref{thm:restricted-div} does not hold for Noetherian
  functions. For example, even though $e^z-1$ is Noetherian in $D(1)$
  and the division $(e^z-1)/z$ is restricted, it is not possible to
  write a Noetherian chain for it in $D(1)$. This is not trivial to
  check, and is proved
  in~\secref{sec:no-noetherian-restricted-div}. This explains why the
  generalization to LN-functions is crucial for our approach even if
  one is initially interested only in classical Noetherian function
\end{Rem}

\subsubsection{Logarithmic derivation}

\begin{Thm}\label{thm:log-derivative}
  Let $\cC^\delta$ be an LN cell and
  $f\in\cO_\LN(\cC^\delta)$. Suppose further that
  \begin{equation}
    f:\cC^\delta\to\C^*.
  \end{equation}
  Then for every $j=1,\ldots,\ell+1$ we have
  \begin{align}
    \partial_j f/f &\in\cO_\LN(\cC) & &\text{and} & \sF((\partial_j f/f)\rest{\cC}) &< \eff(f).
  \end{align}
  If $f$ is real then $\partial_jf/f$ is real.
\end{Thm}
\begin{proof}
  Using the monomialization theorem write $f=m_f(\vz) U(\vz)$ on the
  cell $\cC^{\sqrt\delta}$. Then
  \begin{equation}
    \partial_jf/f = (\valpha(f)_j f+m_f\partial_j U)/f = \valpha(f)_j+\partial_j U/U.
  \end{equation}
  Now $\norm{\partial_j U}_\cC<\eff(f)$ by
  Lemma~\ref{lem:partial-j-bound} and $1/U<\eff(f)$ by monomialization,
  so $\norm{\partial_jf/f}_\cC<\eff(f)$. The claim now follows by
  restricted division. The real case follows easily.
\end{proof}

\subsubsection{Stability under pullback}

\begin{Thm}\label{thm:pullback}
  Let $\phi:\cC^\delta\to\hat\cC$ be an LN map between LN cells. Let
  $F\in\cO_\LN(\hat\cC)$. Then
  \begin{align}
    \phi^*F &\in\cO_\LN(\cC) & &\text{and} & \sF(\phi^*F\rest{\cC}) &< \eff(\phi,F).
  \end{align}
  If $\phi,F$ are real then $\phi^*F$ is real.
\end{Thm}
\begin{proof}
  Let $F_1,\ldots,F_N$ be the LN chain for $F$. We claim that
  adding the pullbacks $\phi^*F_i$ to the LN chain for $\phi$
  gives an LN chain for $\phi^*F$ on $\cC$ with the effective
  format bound.

  Let $\vz_{1..\ell}$ denote the coordinates on $\cC$ and
  $\vw_{1..\hat\ell}$ the coordinates on $\hat\cC$, so that $\phi$
  takes the form $\vw_j=\phi_j(\vz)$. Write $\Sigma$ for the set of
  indices $j$ where the $j$-th fiber of $\hat\cC$ is of type
  $D_\circ,A$ and $\Sigma'$ for the disc fibers. When $j\in\Sigma'$
  write $r_j$ for the (constant) radius of the $j$-th disc fiber.
  Compute
  \begin{multline}
    \partial^{\cC^\delta}_k (F_i\circ\phi) = \sum_{j=1,\ldots,\hat\ell} \pd{F_i}{w_j}\circ\phi\cdot(\partial^{\cC^\delta}_k\phi_j) = \\
    \sum_{j\in\Sigma'} \frac{\partial^{\cC^\delta}_k\phi_j}{r_j}\cdot (\partial^{\hat\cC}_j F_i)\circ\phi\ + \\
    \sum_{j\in\Sigma} \frac{\partial^{\cC^\delta}_k\phi_j}{\phi_j}\cdot \big( \partial^{\hat\cC}_j F_i\big)\circ\phi.   
  \end{multline}
  Now note that for $j\in\Sigma'$ the map $\phi_j$ maps into $D(r)$,
  so $\partial^{\cC^\delta}_k\phi_j / r$ is bounded effectively by
  Lemma~\ref{lem:partial-j-bound} on $\cC^{\sqrt\delta}$, and is
  therefore LN of effectively bounded format on $\cC$ by restricted
  division. For $j\in\Sigma$ the map $\phi_j$ is non-vanishing since
  it maps into a fiber of type $D_\circ,A$, and the logarithmic
  derivative $\partial^{\cC^\delta}_k\phi_j/\phi_j$ is thus LN with
  effectively bounded format on $\cC$ by logarithmic
  derivation. Putting these together finishes the proof. Note that it
  is important here that as a real cellular map,
  $f(\R_+\cC^{\delta})\subset\R_+\hat\cC$ so the pullback of $F$ is
  indeed real valued on $\R_+\cC^{\delta}$.
\end{proof}

Stability under pullback implies for example that if
$f:D^{1/2}\to\cC^{1/2}$ is an LN-map then the diameter of $f(D)$ with
respect to the $\partial^\cC$-parametrization is bounded in terms of
$\sF(f)$. Indeed, for the coordinates of type $D,*$ this is obvious. For
$D_\circ,A$ coordinates the pullback is a non-vanishing LN-function
and therefore, by monomialization, a unit of bounded logarithmic
variation.

\begin{Rem}
  We remark that this metric restriction is not hyperbolic, i.e. it is
  not a consequence of Schwarz-Pick. For example, the map
  \begin{equation}\label{eq:unbounded-width-example}
    f:D(1)^{1/2}\to A(1,\e)^{1/2}, \qquad f(z):= e^{-\tfrac12 \log\e +\tfrac14 z \log\e}
  \end{equation}
  has image of logarithmic width proportional to $|\log\e|$, which is
  not uniformly bounded independently of the cell.

  Stability under pullback is rather a consequence of metric
  restrictions related to valency of holomorphic maps. Suppose for
  example that $f:D^{1/2}\to A(1,\e)^{1/2}$ is a univalent map from a
  disc to an annulus. Then it can be lifted under the exponential
  cover of $A(1,\e)^{1/2}$ to a map
  \begin{equation}
    \tilde f: D^{1/2} \to \{ \log\e-\log 2 <\Re z< \log 2 \}.
  \end{equation}
  The derivative of $\tilde f$ at every point of $D$ is bounded by a
  constant, because otherwise by Kobe's $\tfrac14$-theorem theorem it would
  contain a disc of diameter greater than $2\pi$, in which case
  $f=e^{\tilde f}$ would not be univalent. In general the valency of
  LN functions is always bounded in terms of the format, and this
  prevents counterexamples like~\eqref{eq:unbounded-width-example}
  where the valency is indeed proportional to $|\log\e|$.
\end{Rem}

\subsection{Root extraction}

\begin{Prop}\label{prop:root-extract}
  Let $\cC^\delta$ be an LN cell and
  $f\in\cO_\LN(\cC^\delta)$. Suppose further that
  \begin{equation}
    f:\cC^\delta\to\C^*
  \end{equation}
  and that $f^{1/N}$ has a univalued branch for
  $N\in\N$. Then
  \begin{align}
    f^{1/N}&\in\cO_\LN(\cC) & &\text{and} & \sF(f^{1/N}) &< \eff(f).
  \end{align}
  If $f$ is real and positive on $\R_+\cC^\delta$ then $f^{1/N}$ can
  be chosen to be real.
\end{Prop}
\begin{proof}
  By monomialization write $f=m_f U$ on $\cC^{\sqrt\delta}$. The root
  $f^{1/N}$ is univalued if and only if $\valpha(f)=\boldsymbol0$
  modulo $N$. In particular note that this implies $N<\eff(f)$. We
  will show that both $m_f^{1/N}$ and $U^{1/N}$ are LN with
  effectively bounded format on $\cC$.

  Starting with $m_f^{1/N}$, by assumption it is a monomial effectively
  bounded on $\cC^{\sqrt\delta}$. Any such monomial is LN with
  bounded format. Indeed, it satisfies
  \begin{equation}
    \partial^{\cC}_j(m_f^{1/N}) = \frac{\valpha_j}N m_f^{1/N}.
  \end{equation}
  As for $U$, we write
  \begin{equation}
    \partial^\cC_j(U^{1/N}) = \frac{U^{1/N} \partial^\cC_j U}{N U}.
  \end{equation}
  Since $U$ is effectively bounded from above and below and has a
  well-defined root $U^{1/N}$, the function $1/U$ is LN with bounded
  format by restricted division and the equation above gives an LN
  chain for $U^{1/N}$.

  In the real case it is clear that $f^{1/N}$ defined above is real if
  one choses $U^{1/N}$ to be the (unique) positive branch.
\end{proof}

\subsection{Removable singularities}
The following is an LN version of the Riemann removable singularity
theorem.

\begin{Prop}\label{prop:removable-sing}
  Let $F\in\cO_\LN(\cC\odot D_\circ(r))$ for some LN cell
  $\cC$ of length $\ell$. Then $F_0:=F(\vz,0)\in\cO_\LN(\cC)$ and
  $\sF(F_0)\le\sF(F)$. If $F$ is real then $F_0$ is real.
\end{Prop}
\begin{proof}
  Suppose the LN chain of $F$ is given by functions
  $F_1,\ldots,F_N$. By assumption these functions are all bounded on
  $\cC$, and by the Riemann removable singularity theorem they all
  extend to well-defined functions $F_{0,i}(\vz,0)$. In particular the
  derivation rules
  \begin{equation}
    \partial^\cC_j(F_i) = G_{i,j}(F_1,\ldots,F_N), \qquad j=1,\ldots,\ell
  \end{equation}
  continue holomorphically to the same identities on $F_{0,i}$. The
  statement follows.

  When $F$ is real on $\cC\odot D_\circ(r)$ then its limit $F_0$ is
  real on $\cC$, and the chain above is real by definition.
\end{proof}

\begin{Rem}\label{rem:removable-sing-sin}
  We cannot claim in Proposition~\ref{prop:removable-sing} that $F$
  extends to an LN function over $\cC\odot D(r)$ in general, because
  there is no clear way to extend the derivation rules in the
  $\partial^\cC_{\ell+1}$ direction. For example $\sin(z)/z$ is LN in
  $D_\circ(1)$, as one may easily verify (or by restricted
  division). But we do not know whether it is LN in $D(1)$.
\end{Rem}

\begin{Cor}\label{cor:removable-sing-taylor}
  Let $F\in\cO_\LN(\cC^\delta\odot D_\circ(r))$ for some LN cell $\cC$
  of length $\ell$. Write a Taylor expansion
  \begin{equation}
    F(\vz) = \sum_{k=0}^\infty a_j(\vz_{1..\ell})\vz_{\ell+1}^j.
  \end{equation}
  Then $a_j\in\cO_\LN(\cC)$ and $\sF(a_j\rest\cC)\le\eff(F,j)$.
\end{Cor}
\begin{proof}
  First note that we have a Taylor instead of Laurent expansion
  because of the removable singularity theorem. For $a_0$ the claim is
  Proposition~\ref{prop:removable-sing}. Now write
  \begin{equation}
    \pd{F}{\vz_{\ell+1}} = \frac{\partial^{\cC^\delta\odot D_\circ(1)}_{\ell+1} F}{\vz_{\ell+1}}
  \end{equation}
  and note that this is a restricted division on
  $\cC^\delta\odot D_\circ(1)$. Thus
  $\pd{F}{\vz_{\ell+1}}\in\cO_\LN(\cC\odot D_\circ(r))$ with an
  effective bound on the format. The claim for $a_1$ thus follows from
  Proposition~\ref{prop:removable-sing} again, and repeating this
  argument we get a similar bound for $a_j$.
\end{proof}

\subsection{The $\nu$-cover of a cell}
\label{sec:nu-cover}

Recall the notion of a $\nu$-cover of a cell from
\cite[Section~2.6]{me:c-cells} making small adaptations for the LN
setting.

\begin{Def}[The $\vnu$-cover of a cell]\label{def:nu-cover}
  Let $\cC$ be an LN cell of length $1$. For $\cC=D_\circ,A$ and
  $\nu\in\Z$ we define the $\nu$-cover $\cC_{\times\nu}$ by
  \begin{align}
    D_\circ(r)_{\times\nu}&:=D_\circ(r^{1/\nu}) & A(r_1,r_2)_{\times\nu}&:=A(r_1^{1/\nu},r_2^{1/\nu})
  \end{align}
  For $\cC=D(r),*$ the cover $\cC_{\times\nu}$ is defined only for
  $\nu=1$. In all cases we define $R_\nu:\cC_{\times\nu}\to\cC$ by
  $R_\nu(z)=z^\nu$.
  
  Let $\cC$ be a cell of length $\ell$ and let
  $\vnu=(\vnu_1,\ldots,\vnu_\ell)\in\pi_1(\cC)$ be such that
  $\vnu_j\vert\vnu_k$ whenever $j>k$ and $G_j=G_k=\Z$. Define the
  \emph{$\vnu$-cover} $\cC_{\times\vnu}$ of $\cC$ and the associated
  \emph{cellular map} $R_\vnu:\cC_{\times\vnu}\to\cC$ by induction on
  $\ell$. For $\cC=\cC_{1..\ell-1}\odot\cF$ we let
  \begin{equation}
    \cC_{\times\vnu}:=(\cC_{1..\ell-1})_{\times\vnu_{1..\ell-1}}
    \odot (R_{\vnu_{1..\ell-1}}^*\cF_{\times\vnu_{\ell}})
  \end{equation}
  We define $R_\vnu(\vz_{1..\ell}):=\vz^\vnu$.
\end{Def}

As explained in \cite[Section~2.6]{me:c-cells}, the divisibility
conditions on $\vnu$ ensure that the radii of $\cF_{\times\vnu_\ell}$,
which are a-priori multivalued roots, are in fact univalued.  We will
usually consider the $\nu$-cover with $\nu\in\N$, meaning that we take
$\vnu$ with $\vnu_i=\nu$ when $G_i=\Z$ and $\vnu_i=1$ otherwise.

The pullback to a $\vnu$-cover will be used in our treatment to
resolve the ramification of multivalued cellular maps. We record a
simple proposition concerning the interaction between extensions and
$\nu$-covers. 

\begin{Prop}\label{prop:ext-v-cover}
  Let $\cC^\delta$ be an LN cell and $\nu\in\N$. Then
  $\cC_{\times\nu}$ is LN and admits a
  $2\delta^{1/\nu}$-extension, and the covering map $R_\nu$ extends to
  an LN map
  $R_\nu:(\cC_{\times\nu})^{2\delta^{1/\nu}}\to\cC^\delta$. Moreover
  \begin{equation}
    \sF(R_\nu\rest{(\cC_{\times\nu})^{2\delta^{1/\nu}}}) < \eff(\cC^\delta,\nu).
  \end{equation}
  When $\cC^\delta$ is real the cover $\cC_{\times\nu}$ can also be
  chosen to be real.
\end{Prop}

\begin{proof}
  The main difference compared with the usual complex cellular case is
  that we get a slightly weaker extension $2\delta^{1/\nu}$ instead of
  $\delta^{1/\nu}$. The proposition is proved by using stability under
  pullbacks and Proposition~\ref{prop:root-extract} for extracting
  univalued roots; since each of these steps involves passing to some
  (arbitrarily small) extension we only get the LN-ness of $R_\nu$ on
  a slightly smaller extension than in the complex case. In fact $2$
  is an arbitrary choice and any constant larger than $1$ would have
  sufficed.

  For the real case, note that whenever we have a real cell
  $\cC\odot D_\circ(r)$ (or similarly with other fiber types) the
  radius $r$ must have a constant sign on $\R_+\cC$, since it is
  non-vanishing and real there. Up to changing $r$ with $-r$, which
  does not affect the cell itself, we can assume that the radius is
  positive and can therefore choose $r^{1/\nu}$ in the definition of
  $(D_\circ(r))_{\times\nu}$ to be real as well.
\end{proof}

\subsection{Refinement}

The following theorem shows that one can always cover a cell with a a
given extension by cells with a larger extension. Following
\cite{me:c-cells} we state this with $\hrho$-extensions because in
this natural choice of the parameter the asymptotic bounds is more
precise. However in the present paper we will not make use of these
precise bounds.

\begin{Thm}[Refinement theorem]\label{thm:cell-refinement}
  Let $\cC^\hrho$ be a (real) cell and $0<\sigma<\rho$. Then there
  exists a (real) cellular cover $\{f_j:\cC_j^\hsigma\to\cC^\hrho\}$
  of size $\poly_\ell(\rho,1/\sigma)$. Moreover
  $\sF(f_j)<\eff(\cC^\hrho)$.
\end{Thm}

The proof of this theorem is the same as in
\cite[Theorem~9]{me:c-cells}, using stability under pullback of LN
maps. We leave the straightforward verification of the details to the
reader.

\section{Proof of the CPT for the LN category}
\label{sec:cpt-proof}

We will prove the CPT for a single function $F$. The general case
follows from this in a straightforward way by first finding cells
compatible with $F_1$, then covering them by cells compatible with
$F_2$ and so on. Note that the proof of \cite{me:c-cells} does not
seem to extend to the LN category as it requires working with various
holomorphic functions that are not LN.

We proceed by induction on the length of the cell $\cC$. So suppose
the theorem is established for cells of length $\ell$, and we will
prove the statement for cells of the form $\cC\odot\cF$ where $\cC$ is
of length $\ell$. To simplify the notation we will write
$\vz=\vz_{1..\ell}$ for the coordinates on $\cC$ and $w$ for the
coordinate on $\cF$.

In the case $\cF=*$ the map $F$ pulls back to a map
$F'\in\cO_\LN(\cC)$ and the claim follows from the CPT for $\cC$. For
the remaining cases, we will first reduce the cases
$\cF=D_\circ,A$ to the case $\cF=D(r)$ and then prove the disc case at
the end.

By applying the refinement theorem we can assume that $\cC\odot\cF$
admits $\delta$-extension with some small $\delta$ to be chosen
later. Note that in this reduction the format of
$\cC^\delta\odot\cF^\delta$ remains independent of $\delta$.

\subsection{Reducing $\cF=A(r_1,r_2)$ to $\cF=D(r)$}
\label{sec:A-to-D-reduction}

Write a Laurent expansion for $F$,
\begin{equation}
  F(\vz,w) = \sum_{j=-\infty}^\infty a_j(\vz) w^j.
\end{equation}
Using Proposition~\ref{prop:noetherian-domination} we have integers
$p,q$ with $|p|+|q|$ effectively bounded, and some effectively bounded
$M>0$, such that $F(\vz,\cdot)$ has the $(p,q,M)$ Laurent domination
property for every $\vz\in\cC^\delta$. Write
\begin{equation}
  F(\vz,w) = \sum_{j=p}^q a_j(\vz) w^j + R(\vz,w).
\end{equation}
Unfortunately, even though $a_j(\vz)$ are bounded holomorphic on
$\cC$, we do not know if they are LN. We will approximate them by a
slightly more complicated argument.

Write $\partial_w:=\partial^{\cC\odot\cF}_{\ell+1}$. Then
\begin{equation}
  \partial_w F(\vz,w) = \sum_{j=p}^q j a_j(\vz) w^j + \partial_w R(\vz,w).
\end{equation}
Define the constant-coefficient differential operators $D_k$ by
\begin{equation}
  D_k = \prod_{j=p,\ldots,k-1,k+2,\ldots,q} \frac{\partial_w-j\cdot\I}{k-j}
\end{equation}
so that
\begin{equation}
  A_k := D_k F(\vz,w) = a_k(\vz)w^k + D_k R(\vz,w).
\end{equation}
\begin{Claim}\label{claim:DkR-bound}
  We have
  \begin{equation}
    D_k R(\vz,w) \le M' \delta \max_{j=p,\ldots,q} |a_j(\vz) w^j|
  \end{equation}
  for every $(\vz,w)\in\cC^\delta\odot\cF$, where $M'<\eff(F)$.
\end{Claim}
\begin{proof}
  Since $|p|,|q|$ are bounded effectively in $F$ it is enough to prove
  the claim with $\partial_w^l$ instead of $D_k$, with the bound
  depending effectively on $l$. The case $l=0$ is
  Proposition~\ref{prop:domination-residue-bound}, and the case of
  general $l$ is proved in essentially the same way.
\end{proof}
In particular, let $\e>0$ be some constant to be chosen later. Since
$M$ depends only on the format of $F$ which is independent of our
choice of $\delta$, we can choose $\delta<\e/(2M)$ and in this case
\begin{equation}\label{eq:A-vs-a}
  1-\e < \frac{\max_{j=p,\ldots,q} |a_j(\vz)w^j|}{\max_{j=p,\ldots,q} |A_j(\vz,w)|} < 1+\e
\end{equation}
holds uniformly for $(\vz,w)\in\cC^\delta\odot\cF$.

Note that $A_k\in\cO_\LN(\cC^\delta\odot\cF^\delta)$. Consider two
pullbacks
\begin{align}
  A_k^1(\vz) &:= A_k(\vz,r_1(\vz)) & & & A_k^2(\vz) &:= A_k(\vz,r_2(\vz))
\end{align}
which are both LN functions of effectively bounded format. Apply the
CPT inductively to $\cC^\delta$ and the collection of functions
\begin{equation}\label{eq:A_k-pair-collection}
  \{ A_k^b, A_k^b-A_j^b : j,k=p,\ldots,q \text{ and } b=1,2 \}.
\end{equation}
For each of the resulting maps $f_j:\cC_j^\delta\to\cC^\delta$ we pull
back $\cF$ along $f_j$ to give a cell
$\cC_j^\delta\odot(f_j^*\cF^\delta)$. In order to get a cellular cover
for the original cell $\cC\odot\cF$ it will suffice to prove the CPT
for each of these cells and the pullback of $F$ separately. In other
words, we may replace $\cC$ by each of the cells $\cC_j$, and simply
assume without loss of generality below that $\cC^\delta$ is already
compatible with the functions~\eqref{eq:A_k-pair-collection}.

Assume for simplicity of the notation that none of the $A_k^b$ are
identically vanishing on $\cC$. If they are then they should simply be
removed from consideration below, which would only make the notation
slightly more cumbersome. By~\eqref{eq:A_k-pair-collection} we have
\begin{equation}
  A_k^1/A_j^1:\cC^\delta\to\C\setminus\{0,1\}.
\end{equation}
The fundamental lemma then implies that these ratios ``do not move
much'' on $\cC^{\sqrt\delta}$.

We claim first that we may assume $A_p^1/A_j^1>N$ for some large $N$
to be chosen later, for every $j>p$, uniformly over
$\cC^{\sqrt\delta}$. Indeed, suppose this fails for some $j$ at some
point in $\cC^{\sqrt\delta}$. Then by the fundamental lemma with a
suitable choice of $\delta$ we will have $A_p^1/A_j^1<N+1$ uniformly
over $\cC^{\sqrt\delta}$. Then~\eqref{eq:A-vs-a} also implies that 
\begin{equation}
  \frac {\max_{j=p+1,\ldots,q} |a_j(\vz)r_1(\vz)^j|}{a_p(\vz)r_1(\vz)^p} < 2N+2
\end{equation}
Then in fact $F$ has the $(p+1,q,\tilde M)$ Laurent domination
property on $\cC^{\sqrt\delta}\odot \cF^\delta$ for some slightly
larger but still effective $\tilde M$. Indeed, whenever the maximum in the
domination property is achieved at index $p$ the inequality above
implies that the same maximum is achieved at $k>p$ with
$\tilde M=(2N+2)M$. In this case we can finish by our induction on
$q-p$.

Applying a similar reasoning to $A_k^2/A_j^2$ we conclude that we may
assume $A_q^1/A_j^1>N$ for every $j<q$, uniformly over
$\cC^{\delta}$. Under these conditions Claim~\ref{claim:DkR-bound}
also implies, with suitable choice of $\delta$, that
\begin{equation}\label{eq:A-ratio-bound}
  \begin{aligned}
    1-\frac 1N < \frac{A^1_p(\vz)}{a_p(\vz)r_1(\vz)^p} < 1+\frac 1N\\
    1-\frac 1N < \frac{A^2_q(\vz)}{a_q(\vz)r_2(\vz)^q} < 1+\frac 1N.
  \end{aligned}  
\end{equation}
Consider the quotient
\begin{equation}
  \tilde s_{p,q} := \frac{A^1_p(\vz) r_2(\vz)^q}{A^2_q(\vz)r_1(\vz)^p}.
\end{equation}
We claim that $\tilde s_{p,q}\in\cO_\LN(\cC^{\sqrt\delta})$ has effectively
bounded format. It is enough to show that it is a restricted division
with a bound on the norm. And indeed, by~\eqref{eq:A-ratio-bound} we
have
\begin{equation}\label{eq:s-vs-r}
  (1-1/N)^2 < \tilde s_{p,q}(\vz)/r_{p,q}(\vz)^{q-p} < (1+1/N)^2
\end{equation}
where
\begin{equation}
  r_{p,q} = \sqrt[q-p]{\frac{a_p(\vz)}{a_q(\vz)}}.
\end{equation}
To see that this latter $r_{p,q}$ is effectively bounded, note that it
is the radius $|w|=r_{p,q}$ where $|a_p w^p|=|a_q w^q|$. Since by our
assumption $a_p w^p$ is dominant on $|w|=r_1$ and $a_q w^q$ is
dominant on $|w|=r_2$, we must have $r_1<r_{p,q}<r_2$.

Finally, we wish to extract the $(q-p)$-th root of
$\tilde s_{p,q}$. For this purpose consider the covering cell
$\hat\cC:=\cC_{\times(q-p)}$ and pull back our fiber $\cF$ to
$\hat\cC$. This is again an LN cell of effectively bounded
format, and as before it will suffice to prove the claim for this new
cell (note that we get a slightly smaller extension for $\hat\cC$ but
it is easy to compensate, e.g. by refining $\hat\cC$). Crucially on
$\hat\cC$ the LN function $\tilde s_{p,q}$ admits a $(q-p)$-th
root. Without loss of generality we replace $\cC$ by $\hat\cC$ and
simply assume below that $\tilde s_{p,q}$ admits a $(q-p)$-th root
$\tilde r_{p,q}\in\cO_\LN(\cC^{\sqrt\delta})$.

Now cover our original fiber $A(r_1,r_2)$ as
\begin{equation}
  A(r_1,r_2) \subset A(r_1,c\tilde r_{p,q}) \cup A(c^{-1} \tilde r_{p,q},r_2)
\end{equation}
where $c>1$ is arbitrarily close to $1$, taken simply to cover the
circle of radius $\tilde r_{p,q}$. It will suffice to prove the claim
for the pullback of $F$ to
\begin{equation}
  \cC^{\sqrt\delta}\odot A(r_1,c\tilde r_{p,q}) \text{ and } \cC^{\sqrt\delta}\odot A(c^{-1} \tilde r_{p,q},r_2)
\end{equation}
We proceed with the first of these, the other being analogous. On the
circle $|w|=c \tilde r_{p,q}$ we already have for an appropriate
choice of $N$ the estimate
\begin{equation}
  \frac{|a_q w^q|}{|a_p w^p|} < 2.
\end{equation}
Indeed, on the circle of radius $r_{p,q}$ this ratio is $1$, and
by~\eqref{eq:s-vs-r} the ratio $\tilde r_{p,q}/r_{p,q}$ is arbitrarily
close to $1$ for suitable $N$ and $c$ is arbitrarily close to $1$. It
follows on $\cC^{\sqrt\delta}\odot A(r_1,c\tilde r_{p,q})$ we have the
$(p,q-1,2M)$ domination property and we can finish as before by
induction on $q-p$. This finishes the proof.

In the real case, the functions $A_j^b$ defined above are also
real. Then $\tilde s_{p,q}$ are real, and the roots $\tilde r_{p,q}$
can be chosen to be real, so the new annuli we produce are again
real.

\subsection{Reducing $\cF=D_\circ(r)$ to $\cF=D(r)$}
\label{sec:Dcirc-to-D-reduction}

Without loss of generality by pulling back along a rescaling map we
may assume $\cF=D_\circ(1)$.  Write a Taylor expansion for $F$,
\begin{equation}
  F(\vz,w) = \sum_{j=0}^\infty a_j(\vz) w^j.
\end{equation}
where according to Corollary~\ref{cor:removable-sing-taylor} we have
$a_j\in\cO(\cC^{\delta})$ with $\cF(a_j)\le\eff(F,j)$. (Formally we
start with a slightly larger extension than $\delta$ to get the format
bounded on $\cC^\delta$.) The proof is now quite similar to the one given in
\cite[Section~7.2.2]{me:c-cells} so we outline the argument briefly
only to show that the process is indeed effective.

Using Proposition~\ref{prop:noetherian-domination} we have an integer
$q$, effectively bounded, and some effectively bounded $M>0$, such
that $F(\vz,\cdot)$ has the $(q,M)$ Taylor domination property for
every $\vz\in\cC^\delta$. First apply the CPT to $\cC^\delta$ with the
functions $a_j,a_j-a_k$ for $j,k=0,\ldots,q$. As before, pulling back
the fiber to each of the resulting cells we may assume without loss of
generality that these functions are all non-vanishing (or some are
identically vanishing, which does not affect the arguments
below). Then
\begin{equation}
  a_j/a_k:\cC^\delta\to\C\setminus\{0,1\}
\end{equation}
and the fundamental lemma then implies that these ratios ``do not move
much'' on $\cC^{\sqrt\delta}$. In particular, we can divide the pairs
$(j,k)\in\{0,\ldots,q\}^2$ with $j>k$ into a set $\Sigma$ where
$a_j/a_k$ is uniformly bounded by $2$, and its complement where
$a_j/a_k$ is uniformly bounded below by $1$. For $(j,k)\in\Sigma$ we
define the radii
\begin{equation}
  r_{j,k}(\vz):=\sqrt[j-k]{\frac{a_k(\vz)}{a_j(\vz)}}.
\end{equation}
Note that the divisions here are restricted by our definition of
$\Sigma$. As in~\secref{sec:A-to-D-reduction} we may, after pulling
back to a $q!$-cover of $\cC$ and refining, assume that
$r_{j,k}\in\cO_\LN(\cC^\delta)$ with effectively bounded format.

By definition $r_{j,k}$ is the circle where $|a_j w^j|=|a_k w^k|$. It
is easy to see that if $|w|$ has distance at least $\log 3$ to each of
the radii $r_{j,k}$ then $F$ is non-vanishing, because one term
$a_t w^t$ would dominate all other terms up to order $q$ as well as
the Taylor residue. Since the radii $r_{j,k}$ are ``almost constant''
on $\cC$ one can group them into at most $q^2$ disjoint ``special
annuli'' of effectively bounded logarithmic width, such that every
radius along with the annulus of logarithmic width $\log 3$ around it
is contained in one of these annuli. The annuli that remain between
adjacent pairs of special annuli cannot contain zeros, so they are
already compatible with $F$. We cover each of the special annuli by
discs (with their number effectively bounded in terms of the
logarithmic width), and thus reduce the problem to the case
$\cF=D(r)$. For more details on this clustering construction see
\cite[Section~7.2.2]{me:c-cells}.

In the real case the functions $a_j$ are real, and the radii $r_{j,k}$
can be chosen to be real. Therefore the ``special annuli'' above are
real, and we can cover their real part by real discs centered over the
positive real line. Thus we obtain a real cover.

\subsection{The case $\cF=D(r)$}

We start similarly to the case of $D_\circ(r)$. Without loss of
generality by pulling back along a rescaling map we may assume
$\cF=D(1)$. It will also suffice to prove that there exists some
$\e>0$ with $1/\e<\eff(F)$ and a covering for $\cC\odot D(\e)$
compatible with $F$. This simply amounts to choosing our cells (by
refinement) to have $\e\delta$-extension instead of a
$\delta$-extension.

Write a Taylor expansion for $F$,
\begin{equation}
  F(\vz,w) = \sum_{j=0}^\infty a_j(\vz) w^j.
\end{equation}
We see again that $F(\vz,\cdot)$ has the $(q,M)$ Taylor domination
property with effective $q,M$ in $D(1)$ for every
$\vz\in\cC^\delta$. After suitable application of the CPT we have
\begin{equation}
  a_j/a_k:\cC^\delta\to\C\setminus\{0,1\}
\end{equation}
and the fundamental lemma then implies that these ratios do not move
much on $\cC^{\sqrt\delta}$. We proceed by induction on $q$, so
suppose for functions with $(q',M')$ domination and $q'<q$ the CPT is
already proved.

For the base case $q=0$, the free term $a_0(\vz)$ dominates the Taylor
residue on a ball of radius $\e$ with $1/\e=\eff(F)$ by
Proposition~\ref{prop:domination-residue-bound} so $F$ is already
non-vanishing in $\cC\odot D(\e)$ and the claim is proved.

Proceeding now with general $q>0$, we may as well assume that at some
point $\vz\in\cC^\delta$ we have $|a_q(\vz)|\ge|a_j(\vz)|$ for all
$j$, because otherwise we have the $(q-1,M)$ domination property and
are done by induction. Since the ratios are nearly constant, we have
$|a_q(\vz)|\ge |a_j(\vz)|/2$ uniformly in $\cC^{\sqrt\delta}$. It
follows using Proposition~\ref{prop:domination-derivative} that
$\partial_w F(\vz,w)$ has the $(q-1,M)$ domination property on the
slightly smaller fiber $D(1/2)$.

Apply the inductive case to $\partial_w F$. We get a covering of
$\cC\odot\cF$ and as usual it will suffice to now prove the CPT for
the pullback of $F$ to each of the resulting cells. Crucially, since
all the maps in the CPT are affine in the final variable (see
Remark~\ref{rem:CPT-affine}), these pullbacks still the derivative in
the $\partial_w$ direction either non-vanishing or identically
zero. In other words we are reduced to proving the CPT under the
additional assumption that $\partial_w F$ is already compatible with
$\cC^\delta\odot D(1)$.

If $\partial_w F\equiv0$ then we can find a covering of $\cC^\delta$
compatible with $F(\vz,0)$ by induction on $\ell$ and then multiply
each cell by a constant $D(1)$. So assume $\partial_w F$ is nowhere
vanishing. By Proposition~\ref{prop:domination-no-zeros} there is an
$\e>0$ with $1/\e=\eff(F)$ such that $\partial_w F(\vz,\cdot)$ has the
$(0,1/4)$-domination property in $D(\e)$ for every
$\vz\in\cC^\delta$. To simplify the notation by rescaling $D(\e)$ we
may as well assume that $\partial_w F$ has the $(0,1/4)$-domination
property in the fiber $D(1)$. Then $F$ has the $(1,1/4)$-domination
property. Write
\begin{equation}
  F(\vz,w) = a_0(\vz)+a_1(\vz)w+R(\vz,w).
\end{equation}
Perform an inductive CPT for $\cC^\delta$ and the functions
$a_0(\vz),a_1(\vz),a_0(\vz)-a_1(\vz)$. By the same reduction used
before, after pulling back to the resulting cells we may as well
assume that 
\begin{equation}
  a_1/a_0:\cC^\delta\to\C\setminus\{0,1\}.
\end{equation}
By the fundamental lemma, on $\cC^{\sqrt\delta}$ we uniformly have
either
\begin{align}
  |a_1/a_0| &> 100 & &\text{ or } & |a_1/a_0| < 101.
\end{align}
In the latter case, the $(1,1/4)$-domination property implies that in
the disc of radius $D(1/200)$ the term $a_0(\vz)$ already dominates
$a_1(\vz)+R(\vz,w)$, so $F$ has no zeros. In this case we can just use
a cover $\cC^{\sqrt\delta}\odot D(1/10)$. So assume we are in the case
$|a_1/a_0|>100$.

In this case by a similar reasoning on $\cC^{\sqrt\delta}\odot D(1/2)$
the term $a_1(\vz)$ dominates $a_0(\vz)$ as well as $R(\vz,w)$. Taking
a restricted division
\begin{equation}
  \tilde F := \frac{F}{a_1(\vz)} = a_0(\vz)+w+\tilde R(\vz,w)
\end{equation}
we have an effective bound for the format of $\tilde F$, and it will
be enough to find a covering compatible with $\tilde F$ since its
zeros agree with those of $F$. Below we simply replace $F$ by
$\tilde F$ and assume
\begin{align}\label{eq:simplified-F}
  a_1(\vz)&\equiv 1 & |a_0(\vz)| &<\tfrac1{100} & |R(z,w)|&<\tfrac12
\end{align}
uniformly on $\cC^{\sqrt\delta}\odot D(1/2)$. Under these conditions
it is clear that $F(\vz,w)$ has a single zero
\begin{equation}
  w=w(\vz)\in D(1/2)
\end{equation}
for each $\vz\in\cC^{\sqrt\delta}$ because the term $w$ in the Taylor
expansion is dominant over the circle of radius $1/2$, and in fact the
zeros lies in $D(1/10)$ because $w$ is also dominant there. Moreover
$w(\vz)$ is holomorphic on $\cC^{\sqrt\delta}$. If we show that it is
LN with format $\eff(F)$ then we can finish the proof. Indeed, a
covering for $\cC\odot D(1/2)$ can then be obtained using the cell
$\cC\odot *$ with the map $(\vz,*)\to(\vz,w(\vz))$ to cover the zeros
of $F$ and the cell $\cC\odot D_\circ(1/5)$ with the map
$(\vz,w)\to(\vz,w+w(\vz))$ to cover $\cC\odot D(1/10)$. Note that this
map is compatible with $F$ because $F$ has no zeros other than
$w(\vz)$ in $D(1/2)$.

Note that $|\partial_w F|$ is bounded below by $1/2$ on
$\cC^{\delta}\odot D(1/2)$. To see this write
\begin{equation}
  \partial_w F = 1+\partial_w R(\vz,w)
\end{equation}
and use the $(0,1/4)$-domination property to bound
$\partial_w R(\vz,w)$ on $D(2/3)$. Therefore the functions
\begin{equation}
  D_j\in\cO(\cC^{\sqrt\delta}\odot D(1/2)), \qquad D_j := -\frac{\partial_j F}{\partial_w F}
\end{equation}
are given by restricted division and hence have an effectively bounded
format. Moreover,
\begin{equation}
  0 = \partial_j F(\vz,w(\vz))= (\partial_j F)(\vz,w(\vz))+(\partial_w F)(\vz,w(\vz)) \cdot \partial_j w(\vz)
\end{equation}
where we crucially used the fact that the fiber is of type $D$, so
$\partial_w=\pd{}w$. Concluding,
\begin{equation}
  \partial_jw(\vz) = D_j(\vz,w(\vz))
\end{equation}
for $\vz\in\cC^{\sqrt\delta}$.

We are finally in position to construct an LN chain for the function
$w(\vz)$. Let $F_1,\ldots,F_N$ be a LN chain containing the functions
$D_i$ and the coordinate $F_1\equiv w$. Now consider the system of
equations
\begin{equation}
  \partial_j F_i(\vz,w(\vz)) = (\partial_j F_i + \partial_w F_i \cdot D_j)(\vz,w(\vz)).
\end{equation}
Since $\partial_j F_i,\partial_w F_i$ and $D_j$ are polynomials in the
$F_i$, this is an LN chain of format $\eff(F)$ for $F_i(\vz,w(\vz))$ on
$\cC^{\sqrt\delta}$. In particular we got chain for
$F_1(\vz,w(\vz))\equiv w(\vz)$, finishing the proof.

In the real case, the reduction up to~\eqref{eq:simplified-F}
preserves the realness of $F$. (If we are in the other case where $F$
has no zeros then the cell we use is clearly real.) At this point the
real function $F$ admits a unique root, so it follows by symmetry that
the root $w(\vz)$ must be real. It then easily follows that the two
cells used above and corresponding maps are indeed real.

\subsection{Simple cellular maps}
\label{sec:simple-maps}

It is sometimes useful to know that the maps produced by the CPT are
of a somewhat special form. By inspection of the proof of the CPT we
may assume in the CPT that all maps are simple in the following sense.

Say that a cellular map $f:\cC\to\hat\cC$ is \emph{basic simple} if
each each coordinate $\vw_j=f_j(\vz_{1..j})$ is either
\begin{enumerate}
\item An affine map $\vz_j\to\rho \vz_j+\phi_j(\vz_{1..j-1})$ with
  $\rho\in\R_+$, or
\item A covering maps $\vz_j\to\vz_j^k$ for $k\in\N$, only in the case
  that the $j$-th fiber is of type $A,D_\circ$.
\end{enumerate}
A simple map is a composition of basic simple maps. Note that
simple maps are always bijective from $\R_+\cC$ to $f(\R_+\cC)$, as
they are monotone on each coordinate.

\begin{Rem}\label{rem:CPT-affine}
  With respect to the final variable $\vz_\ell$ we only get affine
  maps regardless of the fiber type.
\end{Rem}

\section{Effective model completeness and o-minimality of $\R_\LN$}
\label{sec:model-theory}

In this section we establish the effective model completeness and
o-minimality of the structure $\R_\LN$, proving
Theorem~\ref{thm:RLN-theory}.

For each real LN-cell $\cC^\delta$ and real LN-function
$F\in\cO_\LN(\cC^\delta)$ let $\R_+F$ denote the restriction
$\R_+F:\R_+\cC^\delta\to\R$. Define the structure $\R_\LN$ to be the
structure generated by the relations $=,<$ and the graphs
$\gr \R_+F\rest\cC$ whenever $F\in\cO_\LN(\cC^\delta)$ is real. We
declare the $\gr \R_+F\rest\cC$ to have format $\sF(F)+1/(1-\delta)$.
Consider the language $\cL_\LN$ having relation symbols $=,<$ and
relation symbols for the graphs above and for the graphs of
$+,*$. Note that while the restrictions of $+,*$ are LN, here we add
the full graphs on $\R^2\to\R$. We declare he formats of $=,<,+,*$ to
be $1$. We also include constants for all real numbers in out language
and define the format of the singletons to be $1$. Finally let
$\Omega_\sF$ be the filtration generated by all these sets with their
respective formats.

\begin{Rem}
  Note that even though we define $\R_\LN$ to be the structure
  generated by the graphs of real LN-functions restricted to the real
  part of the cell, we later prove in
  Proposition~\ref{prop:R_LN-complex-cells} that $\R_\LN$ also
  contains the graphs of complex LN-functions on complex LN-cells.
\end{Rem}

\subsection{Boundary equations}

Write $\cC_F:=\cC\times D(\norm{F}_\cC)$ so
\begin{align}
  \gr F&\subset\cC_F &&\text{ and } & \gr\R_+F&\subset\R_+\cC_F.
\end{align}
We introduce a set of equations for the walls of the cell $\cC$ and
the graph of $F$.

\begin{Def}
  Let $\cC^\delta$ be a real LN-cell. Define the \emph{boundary
    equations}
  \begin{equation}
    B(\cC)\subset\cO_\LN(\cC^\delta)
  \end{equation}
  inductively as follows. If $\cC$ is of length zero $B(\cC)=0$. If
  $\cC=\cC_{1..\ell}\odot\cF$ we define
  $B(\cC):=B(\cC_{1..\ell})\cup B(\cF)$ where
  \begin{equation}
    \begin{aligned}
      B(*) &:= \{\vz_{\ell+1}\} \\
      B(D(r)) &:= \{\vz_{\ell+1}-r,\vz_{\ell+1}+r\} \\
      B(D_\circ(r)) &:=\{\vz_{\ell+1},\vz_{\ell+1}-r, \vz_{\ell+1}+r\} \\
      B(A(r_1,r_2)) &:=\{\vz_{\ell+1}-r_1,\vz_{\ell+1}+r_1, \vz_{\ell+1}-r_2, \vz_{\ell+1}+r_1\}
    \end{aligned}
  \end{equation}
  If $F:\cC^\delta\to\C$ is a real LN-function on a cell of length
  $\ell$ we define $B(F\rest\cC)\subset\cO_\LN(\cC_F^\delta)$ to be
  \begin{equation}
    B(F\rest\cC) := B(\cC) \cup \{\vz_{\ell+1}-F(\vz_{1..\ell})\}.
  \end{equation}
\end{Def}

\begin{Lem}\label{lem:boundary-compatability}
  Let $\cC^\delta$ be a real LN-cell and
  $\{f_j:\cC^\delta_j\to\cC^\delta\}$ a real cover compatible with
  $B(\cC)$. Then every image $f_j(\R_+\cC^\delta_j)$ is either contained in
  or disjoint from $\R_+\cC$.
  
  Similarly if $F:\cC^\delta\to\C$ is a real LN-function and
  $\{f_j:\cC^\delta_j\to\cC_F\}$ a real cover compatible with
  $B(F\rest\cC)$ then every image $f_j(\R_+\cC_j^\delta)$ is either contained
  or disjoint from $\gr\R_+ F\rest\cC$.
\end{Lem}
\begin{proof}
  The first statement follows since $\R_+\cC$ is given by some sign
  conditions on the boundary equations $B(\cC)$ and these signs remain
  constant on each $f_j(\R_+\cC^\delta_j)$ by compatibility. The
  second statement follows similarly.
\end{proof}

\subsection{Effective model completeness}
\label{sec:proof-R_LN-MC}

We say that a formula $\psi(\vx)$ is \emph{existential} if it is of
the form $\exists y:\psi_0(\vx,\vy)$ where $\psi_0$ is
quantifier-free. Our goal in this section is to prove the effective
model completeness part of Theorem~\ref{thm:RLN-theory}. More
explicitly we prove the following.

\begin{Claim}\label{claim:R_LN-MC}
  The structure $\R_\LN$ is effectively model complete. That is,
  every $\cL_\LN$-formula $\psi$ is equivalent to an existential
  formula $\psi_e$ and $\sF(\psi_e)<\eff(\psi)$.
\end{Claim}

We make several reductions. First, we can assume that all relations
corresponding to real LN-functions $F:\cC^\delta\to\C$ appearing in
$\psi$ have $\gr\R_+F\subset I^n$ where $I:=[-1,1]$ (with different
$n$s). Indeed one can always choose some $M=\eff(F)$ so that
$\tilde F:=F(M\vz)/M$ satisfies $\gr\R_+\tilde F\subset I^n$. One can
then replace each relation $\vx\in \gr\R_+F$ in $\psi$ by
\begin{equation}
   \exists \vy : (M \vy=\vx) \land \vy\in\gr\R_+\tilde F.
\end{equation}
Repeating this for each LN-graph in $\psi$ finishes the reduction.

Next, it is enough to prove the case when
$\psi(\R^n)\subset I^n$. Let us demonstrate this reduction for the
$n=1$ case. For general $\psi$ we can write
\begin{equation}
  \begin{aligned}
  \psi(x) &\iff (-1\le x\le 1 \land \psi(x)) \lor
  \big((x<-1 \lor x>1) \land \psi(x)\big) \\
          & =\psi_1(x)\lor\psi_2(x)
  \end{aligned}
\end{equation}
Then $\psi_1(x)$ defines a subset of $I$. For $\psi_2$ write the
formula for the ``inverse'' of $\psi_2(\R^n)$, namely
\begin{equation}
  \psi_2^i(y) := (-1\le y \le 1) \land \exists x : xy=1 \land \psi_2(x).
\end{equation}
Then $\psi_2^i$ defines a subset of $I$, so by our assumption is
equivalent to an existential formula $\psi_{2e}^i(y)$. Then finally
\begin{equation}
  \psi_2(x) \iff \exists y : (xy=1) \land \psi_{2e}^i(y).
\end{equation}
The case of general $n$ is the same, except one should consider $2^n$
cases for the potential inversions of each coordinate.

Next, it is well known that by a simple induction over the
quantification depth, it is enough to prove the claim for the negation
of an existential formula $\psi$. We can also reduce to the case where
all quantifiers in $\psi$ are of the form $\exists y\in I$. Indeed,
suppose $\psi(\vx)$ is of the form $\exists y : \psi_1(\vx,y)$. Then
it is equivalent to
\begin{equation}
  \big( \exists y\in I : \psi_1(\vx,y) \big) \lor \big(\exists y'\in I : \psi'(\vx,y')\big)
\end{equation}
where we will define $\psi'(\vx,y')$ satisfying
\begin{equation}
  \forall y'\in I : \big(\psi'(\vx,y') \iff \psi_1(\vx,1/y')\big).
\end{equation}
We do this as follows. First, if $\psi$ contains a predicate
$\vz\in\gr\R_+F$ and $y$ is one of the components of $\vz$ then we
just replace this by $0=1$ in $\psi'$, because by our assumption all
graphs are contained in $I^n$. Next, if $y$ appears in an equality
$y=T$ where $T$ is a variable or constant we replace it by
$y'T=1$. Similarly we replace $y>T$ by
\begin{equation}
  (y'>0 \land y'T<1) \lor (y'<0 \land y'T>1).
\end{equation}
Finally, it remains to treat the case where $y$ appears in the
relations $+,*$. Since these are semialgebraic we can also express
them using polynomial inequalities in $y'=1/y$, for instance
\begin{equation}
  \begin{aligned}
    x_1+x_2=y &\iff y'(x_1+x_2)=1 &&& x_1+y=x_2 &\iff y'(x_2-x_1)=1 \\
    x_1x_2 = y &\iff y'x_1x_1=1 &&& x_1 y=x_2 &\iff x_1=x_2y'
  \end{aligned}
\end{equation}
Repeating this for each quantifier in $\psi$ finishes the reduction.

% Then we claim that
% $\phi(I^n)$ is a union of images $f_(\R_+\cC_j)$ where $f_j$ is a
% simple cellular map in the sense of~\secref{sec:simple-maps}, and the
% number maps $f_j$ and their formats are $\eff(\psi)$. To see this,

Finally after these reductions, write
$\psi(\vx,\vy):=\exists y\in I^m \psi'(\vx,\vy)$. Consider a cell
\begin{equation}
  \hat\cC = D(2)^\ell \times \prod_{F\in\psi'} \cC_F 
\end{equation}
where the product ranges over all function graphs appearing as
relation symbols in $\psi'$, once for each appearance. Denote the
variables of $\cC_F$ by $\vz^F_{1..\ell(F)}$. Apply the CPT to the
following set of equations. First we include all boundary equations
$B(F)$ for $F\in\psi'$, and for each relation symbol $\vx\in\gr\R_+ F$
we include $\vx_i=\vz^\cC_i$ for $i=1,\ldots,\ell(F)$; for each $+$
relation $x+y=z$ we include $x+y-z$ and similarly with $*$; and for
each $x=y$ or $x<y$ we include $x-y$. Note for that all of these
relations, even if one of $x,y,z$ is a large constant rather than a
variable this does not affect the format of our equations (as
LN-functions) because we may just divide by some large constant to
make sure all coefficients are bounded by $1$. Thus the set of all
equations we use have format $\eff(\psi)$.

Let $\{f_j:\cC_j^\delta\to\hat\cC^\delta\}$ be the real cover obtained
from the CPT. Consider only the maps $f_j$ for which the relations
$\vx_i=\vz^\cC_i$ vanish identically -- denote these
$j\in\Sigma$. Then for each of the relations $R(\vx)$ in $\psi'$, it
is either uniformly true on $f_j(\cC_j^\delta)$ or uniformly false
when $j\in\Sigma$. Indeed, for the LN-relations this is exactly
Lemma~\ref{lem:boundary-compatability} and for the remaining relations
it true for the same reason.

Since the $f_j(\R_+\cC_j)$ cover $\R_+\hat\cC$, we conclude that the
set $\psi'(I^{n+m})$ is given by a union of $f_j(\R_+\cC_j)$ for $j$
in a subset $\Sigma'\subset\Sigma$ given by those $j\in\Sigma$ over
which $\psi'$ holds true. The cellular structure of $f_j$ then implies
that $\psi(I^n)$, given by the projection of $\psi'(I^{n+m})$ to
$I^n$, is given by the images $(f_j)_{1..\ell}((\cC_j)_{1..\ell})$
with $j\in\Sigma'$. Claim~\ref{claim:R_LN-MC} now follows from the
following lemma, showing that the complement of each of these sets
(and therefore also the complement of their union) is given by an
existential formula of effectively bounded format.

\begin{Lem}
  Let $\cC^\delta$ be a real LN-cell and $f:\cC^\delta\to D(2)^\ell$ a
  simple cellular map as described in~\secref{sec:simple-maps}. Then
  $f(\R_+\cC)$ and its complement are given by an existential
  $\cL_\LN$-formula of format $\eff(f)$.
\end{Lem}
\begin{proof}
  By assumption $f$ is a composition of affine and covering maps. If
  we write $\vw_{1..\ell} = f(\vz_{1..\ell})$ then we have equations
  \begin{equation}\label{eq:w-to-z}
    \begin{aligned}
      \vw_1 &= P_1(\vz_1), & &&  \qquad P_1&\in\R[\vz_1] \\
      \vw_2 &= P_2(\vz_2), & &&  \qquad P_2&\in\cO_\LN(\cC_1^\delta)[\vz_2] \\
            & & \vdots & & & \\
      \vw_\ell &= P_\ell(\vz_\ell), & &&  \qquad P_\ell&\in\cO_\LN(\cC_{1..\ell-1}^\delta)[\vz_\ell]
    \end{aligned}
  \end{equation}
  where for each fixed $\vz_{1..j-1}$ the polynomials
  $P_j(\vz_{1..j-1},\cdot)$ are injective on $\R$ for fibers of type
  $D$ and on $\R_+$ for fibers of type $D_\circ,A$. Write
  $\cC=\cF_1\odot\cdots\odot\cF_\ell$ and set
  \begin{equation}
    \R^\cC := \prod_{j=1}^\ell R_j, \qquad R_j :=
    \begin{cases}
      \{0\} & \cF_j=* \\
      \R & \cF_j=D \\
      \R_+ & \cF_j= D_\circ,A
    \end{cases}
  \end{equation}
  Then we can express $f(\R_+\cC)$ with the formula
  \begin{multline}
    \psi(\vw) := \exists \vz\in\R^\cC :\eqref{eq:w-to-z} \text{ holds } \land \\
    \vz_1\in\cF_1 \land \vz_2\in\cF_2(\vz_1)\land \cdots \land \vz_\ell\in\cF_\ell(\vz_{1..\ell-1})
  \end{multline}
  where the conditions $\vz_j\in\cF_j(\vz_{1..j-1})$ are clearly
  expressible in $\cL_\LN$, for instance if $\cF_j=A(r_1,r_2)$ then we
  write this as
  \begin{equation}
    r_1(\vz_{1..j-1}) < \vz_j < r_2(\vz_{1..j-1})
  \end{equation}
  where $r_1,r_2$ are by definition real LN-functions on
  $\cC_{1..j-1}^\delta$.

  Now consider the complement of $f(\R_+\cC)$ in $\R^\ell$. This can
  be expressed as follows. We first show that
  \begin{equation}\label{eq:z-v-w-negation}
    \lnot \exists \vz\in\R^\cC : \eqref{eq:w-to-z} \text{ holds}
  \end{equation}
  is existential. This is equivalent to the disjunction of the
  following:
  \begin{equation}
    \begin{gathered}
      \lnot\exists \vz_1 : \vw_1=P_1(\vz_1) \\
      \exists \vz_1: \vw_1=P_1(\vz_1), \lnot\exists\vz_2 : \vw_2=P_2(\vz_1,\vz_2) \\
      \vdots \\
      \exists \vz_{1..\ell-1} : w_{1..\ell-1}=P_{1..\ell-1}(\vz_{1..\ell-1}), \lnot\exists \vz_\ell : \vw_\ell=P_\ell(\vz_{1..\ell-1},\vz_\ell).
    \end{gathered}
  \end{equation}
  where above each $\exists \vz_j$ is shorthand for
  $\exists \vz_j\in\R^\cC_j$.  We claim that the negated formulas
  \begin{equation}
    \lnot\exists \vz_j\in\R^\cC_j : \vw_j=P_j(\vz_{1..j-1},\vz_j)
  \end{equation}
  can be rewritten as existential formulas. Indeed, for $\cF_j=*$ this
  just means $\vw_j\neq P_j(\vz_{1..j-1},0)$. For $\cF_j=D(r)$ this
  condition is empty because $P_j(\vz_{1..j-1},\cdot)$ is onto $\R$,
  being an affine translate. And for $\cF_j=D_\circ,A$ this condition
  is equivalent to $\vw_j<P(\vz_{1..j-1},0)$ because
  $P_j(\vz_{1..j-1},\vz_j)$ is monotone and tends to infinity as
  $\vz_j\to\infty$.

  Having established that~\eqref{eq:z-v-w-negation} is existential, we
  can finish writing the negation of $\psi(\vx)$ as
  \begin{multline}
    \psi_n(\vw) := \big(\lnot\exists \vz\in\R^\cC : \eqref{eq:w-to-z} \text{
      holds}\big) \lor \exists \vz\in\R^\cC :\eqref{eq:w-to-z} \text{ holds } \land \\
    \big( \vz_1\not\in\cF_1 \lor \vz_2\not\in\cF_2(\vz_1)\lor \cdots \lor \vz_\ell\not\in\cF_\ell(\vz_{1..\ell-1}) \big)
  \end{multline}
  where this is now an existential formula.
\end{proof}

\subsection{Effective o-minimality}

We now finish the proof of Theorem~\ref{thm:RLN-theory} be showing the
effective o-minimality of $\R_\LN$. More specifically we prove an
$\eff(\psi)$ bound for the number of connected components for a set
$A\subset\R^n$ defined by a formula $\psi$. By effective model
completeness we may assume that $\psi$ is existential. Making the same
reductions as in~\secref{sec:proof-R_LN-MC} we may reduce to bounding
the case
\begin{equation}
  \psi(\vx) := \exists \vy\in I^m : \psi'(\vx,\vy)
\end{equation}
where $\psi'$ is quantifier free. This is certainly bounded by the
number of connected components of $\psi'(I^{n+m})$. But we have
already seen at the end of~\secref{sec:proof-R_LN-MC} that this set is
given by the union of images $f_j(\R_+\cC_j)$ where
$\#\{f_j\}<\eff(\psi)$. Since each image is connected this gives the
required bound.

\subsection{Effective polynomial boundedness}

Another key feature of $\R_\an$ that is often used in application is
polynomial boundedness. Say that an effectively o-minimal structure is
\emph{effectively polynomially bounded} if for every definable map
$f:\R\to\R$ there exists $N=\eff(f)$ such that $f(t) < t^N$ for all
$t\gg1$.

\begin{Prop}
  $\R_\LN$ is effectively polynomially bounded.
\end{Prop}
\begin{proof}
  It is equivalent to prove that for every definable $f:(0,1)\to(0,1)$
  we have $f(t)>t^N$ with $N=\eff(f)$ and all sufficiently small $t$.
  Consider the graph $G$ of $f$. As in~\secref{sec:proof-R_LN-MC} we
  can cover $G$ by images $f_j(\R_+\cC_j)$. In particular, the
  projection to $x$ of one of these images must contain an interval
  $(0,\e)$ for some $\e>0$. Denote this cell $\cC$ and the
  corresponding map $f$. Since $f(\cC^\delta)\subset G$ one easily
  sees that the base of $\cC$ must be a punctured disc $D_\circ(r)$, and the
  map $f$ therefore takes the form
  \begin{equation}
    f(\vz_1,*) = (\vz_1^k,f(\vz_1))
  \end{equation}
  where $f\in\cO_\LN(D_\circ(r))$. Both $k$ and the order of zero of
  $f$ at $z_1=0$ are bounded by $\eff(f)$ for instance by effective
  o-minimality (or directly by bounds on the variation of argument),
  and the claim follows.
\end{proof}

As a consequence of polynomial boundedness we can deduce an effective
\Lojas. inequality. Below definability can be taken with respect to any
effectively o-minimal structure which is effectively polynomially
bounded.

\begin{Thm}[Effective \Lojas. inequality]
  Let $X\subset\R^n$ be closed and bounded and $f,g:X\to\R$ two
  definable continuous functions with
  $f^{-1}(0)\subset g^{-1}(0)$. Then there exists $N=\eff(f,g)$ and
  $C>0$ such that $|g(x)|^N<C|f(x)|$ for all $x\in X$.
\end{Thm}
\begin{proof}
  Set
  \begin{equation}
    \lambda(\e) := \min \{ |f(x)| : x\in X \text{ and } |g(x)|=\e \}.
  \end{equation}
  Note that the set on the left hand side is compact and does not
  contain $0$ by assumption, so $\lambda(\e)>0$ for all $\e>0$. The
  claim now easily follows from polynomial boundedness, i.e. $\lambda(\e)\gg\e^N$.
\end{proof}

\section{Effective o-minimality of $\R_{\LN,\PF}$}
\label{sec:effective-PF}

In this section we prove Theorem~\ref{thm:R-LN-PF}: that the extension
of $\R_\LN$ by (unrestricted) Pfaffian functions is effectively
o-minimal. It has long been known
\cite{wilkie:new-omin,speissegger:pfaff-closure} that adding Pfaffian
functions (or even general Rolle leafs) to an o-minimal structure
preserves o-minimality. It has also been shown in
\cite{bs:effective-omin} that the approach of \cite{wilkie:new-omin}
to this problem can be made effective, and we explain how this can be
used to deduce the effective o-minimality of $\R_{\LN,\PF}$
in~\secref{sec:bs-effectivity}. In another direction, Gabrielov and
Vorobjov have developed in a series of papers a more concrete approach
to this effectivity for Pfaffian functions over the algebraic
structure. We show in~\secref{sec:gab-effectivity} how their approach
can be carried over to include Pfaffian functions over $\R_\LN$. While
this requires somewhat more work than the approach of
\cite{bs:effective-omin}, it seems more likely to play a role in an
approach toward Conjecture~\ref{conj:big-sharp-omin}.

\subsection{Extension by Pfaffian functions}
\label{sec:pfaff-extension}

Let $(\cS,\Omega)$ be an effective o-minimal structure. We will write
$\sF^\cS(\cdot)$ for the format of definable sets in $\cS$ to avoid
confusion (as we will be defining another type of format).

Let $G\subset\R^n$ be an open cell in $\cS$,
\begin{equation}\label{eq:pfaff-domain}
  G = I_1\odot\cdots\odot I_n, \qquad \text{where }I_k:= (a_k(\vx_{1..k-1}),b_k(\vx_{1..k-1}))
\end{equation}
such that the walls $a_k,b_k$ are real analytic on $G_{1..k-1}$ and
satisfy $a_k<b_k$ everywhere. We also allow $a_k=\infty$ and
$b_k=\infty$.

Let $f_1,\ldots,f_\ell:G\to\R$ be an $\ell$-tuple of real analytic
functions. Write $X\subset\R^\ell$ for the image of $G$ under
$(f_1,\ldots,f_\ell)$. Say that these functions form a
\emph{restricted Pfaffian chain} over $\cS$ if they satisfy a system
of differential equations
\begin{equation}\label{eq:pfaff-system}
  \pd{f_j}{x_k} = P_{jk}(f_1,\ldots,f_j)
\end{equation}
where the functions $P_{jk}$ are real analytic on $X$. Define the
format of the Pfaffian chain to be
\begin{equation}
  \sF(f_1,\ldots,f_\ell) := \ell+\sF^\cS(G)+\sum_{j,k}\sF^\cS(P_{jk}\rest U).
\end{equation}
A \emph{Pfaffian function} is a function of the form
$f:=P(f_1,\ldots,f_\ell)$ where $P$ is real analytic on $X$. Define
the format of $f$ to be
\begin{equation}
  \sF(f) := \sF(f_1,\ldots,f_\ell)+\sF^\cS(P).
\end{equation}
Denote by $(\cS_\PF,\Omega_\PF)$ the structure generated by all
restricted Pfaffian functions over $\cS$ with the filtration generated
by the formats defined above. Denote by $\cL_{\cS,\PF}$ the language
containing constants for all real numbers and a function symbol for
each Pfaffian function over $\cS$.

\subsection{Khovanskii's bound over $\cS$}

Let $(\cS,\Omega)$ be an effective o-minimal structure.

\begin{Thm}[Khovanskii's bound]\label{thm:khovnanskii-bound}
  Let $F_1,\ldots,F_k$ be Pfaffian over $\cS$, all defined on a common
  domain $G$. Then the number of connected components of the set
  \begin{equation}\label{eq:khovanskii-bound}
    \{ \vx\in G : F_1(\vx)=\cdots=F_k(\vx)=0 \}
  \end{equation}
  is bounded by $\eff(F_1,\ldots,F_k)$.
\end{Thm}
\begin{proof}
  This is the main result of \cite{khovanskii:fewnomials}. The main
  case is when $k=n$ and one counts isolated solutions of the system
  of equations. In this case Khovanskii proves the theorem by defining
  the $\tilde *$-sequence of $F_1,\ldots,F_n$, which turns out to be a
  sequence of polynomials $P_1,\ldots,P_n$ of bounded degrees. In our
  more general case these will now be functions definable in $\cS$
  with formats bounded in terms of the format of $F_1,\ldots,F_k$.

  Khovanskii shows that the number of isolated solutions
  of~\eqref{eq:khovanskii-bound} is bounded by the ``virtual number of
  zeros'' of the $\tilde *$-sequence. This is bounded
  \cite[Section~3.10,~Corollary~3]{khovanskii:fewnomials} by the
  supremum for the number of isolated points in any fiber of the map
  $(P_1,\ldots,P_n):G\to\R^n$ taken over all fibers. When
  $P_1,\ldots,P_n$ this supremum is bounded by Bezout; when
  $P_1,\ldots,P_n$ are definable in $\cS$ a bound of the form
  $\eff(F_1,\ldots,F_n)$ follows trivially by effective o-minimality.
\end{proof}

\begin{Cor}
  Let $\psi(\vx)$ be a quantifier-free $\cL_{\cS,\PF}$-formula. Then the
  number of connected components of $\psi(\R^n)$ is bounded by $\eff(\psi)$.
\end{Cor}
\begin{proof}
  We can write $\psi$ as a disjunction of basic sets of the form
  \begin{equation}
    \{ \vx : P_1(\vx)=\cdots=P_k(x)=0, Q_1(\vx)>0,\ldots,Q_j(\vx)>0\}.
  \end{equation}
  Each connected component of this corresponds to at least one
  connected component of the following set by projection:
  \begin{multline}
    \{ (\vx,\vy,\vz) : P_1(\vx)=\cdots=P_k(x)=0, Q_1(\vx)=\vy_1^2,\ldots,Q_j(\vx)=\vy_j^2,\\
    \vy_1\vz_1=1,\ldots,\vy_j\vz_j=1\}.    
  \end{multline}
  Since this latter set involves only equalities the bound on the
  number of connected components follows from
  Theorem~\ref{thm:khovnanskii-bound}.
\end{proof}

\subsection{Effective o-minimality following Wilkie and Berarducci-Servi}
\label{sec:bs-effectivity}

The structure $\R_{\LN,\PF}$ is generated by Pfaffian functions over
$\R_\LN$ as defined in~\secref{sec:pfaff-extension}. Indeed, every
complex cell $\cC$ is a domain of the form $G$ (up to a minor
technicality with $*$-fibers which is of no significance), and every
LN function $f\in\cO_\LN(\cC^\delta)$ is Pfaffian by definition,
taking the Pfaffian chain to consist of just the coordinate variables
and $P$ to be $f$.

The effective o-minimality of $\R_{\LN,\PF}$ now essentially follows
from the main result of \cite{bs:effective-omin} modulo some small
remarks. Indeed, \cite[Theorem~2.2]{bs:effective-omin} shows that
given an an effective bound for the number of connected components of
quantifier-free formulas, one can obtain effective bounds for
the number of connected components for arbitrary formulas.

A few remarks are in order. First,
\cite[Theorem~2.2]{bs:effective-omin} is formulated for recursive
bounds rather than primitive-recursive as we have insisted on, but by
the remark following the theorem the primitive-recursive analog holds
with the same proof. Second, unlike us \cite{bs:effective-omin} allows
an extension by only finitely many Pfaffian functions (and no
constants), but this makes no difference as one can always consider
the restriction of $\cL_{\LN,\PF}$ that includes the relevant
functions or constants used in a given formula. Finally, the expansion
considered in \cite{bs:effective-omin} is by functions $f:\R^n\to\R$
rather than our more general cellular domains $G$. This can easily be
circumvented as follows. Consider some map $\phi_{a,b}(x):\R\to(a,b)$,
algebraic and real analytic in $a,b,x$ for all $a<b$, which restricts
to a bijection for every fixed $a,b$. Then, for any Pfaffian function
$f=(f_1,\ldots,f_n):G\to\R$ we can consider a pullback
\begin{equation}
  \tilde f:\R^n\to\R, \qquad \tilde f = (\phi_{a_1,b_1}^*f_1,\ldots,\phi^*_{a_n(\vx_{1..n-1}),b_n(\vx_{1..n-1})}f_n).
\end{equation}
Theorem~\ref{thm:khovnanskii-bound} applies to formulas defined with
these pullbacks as well -- in fact they are also easily seen to be
Pfaffian over $\cS$. Then we can consider the structure defined by
these $\tilde f$ functions, and it is easy to see that the structure
generated by these functions contains the graphs of the original $f$.

\subsection{Effective o-minimality following Gabrielov-Vorobjov}
\label{sec:gab-effectivity}

In this section we indicate how the approach of Gabrielov-Vorobjov can
also be generalized to work over $\R_\LN$. The bounds obtained by
Gabrielov-Vorobjov are generally more explicit, and often polynomial
with respect to degrees, so this approach may offer a more plausible
line of attack toward Conjecture~\ref{conj:big-sharp-omin}.

\subsubsection{Restricted Pfaffian functions}

We start by introducing \emph{restricted} Pfaffian functions. We
assume that the domain $G$ in~\eqref{eq:pfaff-domain} is precompact in
$\R^n$, and also that the cell walls are real analytic on
$\bar G_{1..k-1}$.

Now let $f_1,\ldots,f_\ell:\bar G\to\R$ be an $\ell$-tuple of real
analytic functions. Write $X\subset\R^\ell$ for the image of $\bar G$
under $(f_1,\ldots,f_\ell)$. Say that these functions form a
\emph{restricted} Pfaffian chain over $\cS$ if they satisfy a system
of differential equations~\eqref{eq:pfaff-system} where the functions
$P_{jk}$ are real analytic on $X$, and if there exists some domain
$U\subset\C^n$ containing $X$ such that $P_{jk}$ extends
holomorphically to $U$ and this extension if definable in
$\cS$. Define the format of the Pfaffian chain to be
\begin{equation}
  \sF(f_1,\ldots,f_\ell) := \ell+\sF^\cS(G)+\sum_{j,k}\sF^\cS(P_{jk}\rest U).
\end{equation}
A \emph{restricted Pfaffian function} is a function of the form
$f:=P(f_1,\ldots,f_\ell)$ where $P$ again $P$ extends holomorphically
to some complex domain $U$ containing $X$. Define the format of $f$ to
be
\begin{equation}
  \sF(f) := \sF(f_1,\ldots,f_\ell)+\sF^\cS(P\rest U).
\end{equation}
Denote by $(\cS_\rPF,\Omega_\rPF)$ the structure generated by all
restricted Pfaffian functions over $\cS$ with the filtration generated
by the formats defined above.

\begin{Rem}
  An LN-cell of the form $\cC:=D_\circ(1)\odot A(x,2)$ admits an LN
  function $x/y$ which does not extend real-analytically to
  $\bar\cC\subset\R^4$ as it is not analytic in a neighborhood of the
  origin. This function would not be restricted Pfaffian by our
  definition. We will overcome this difficulty with another
  construction at the end to obtain an honest expansion of $\R_\LN$.
\end{Rem}

\subsubsection{Effective o-minimality of $S_\rPF$}

Our goal in this section is to prove the following theorem.

\begin{Thm}\label{thm:effective-SrPF}
  The structure $(\cS_\rPF,\Omega_\rPF)$ is effectively o-minimal.
\end{Thm}

The proof consists in a routine generalization of various results by
Gabrielov and Vorobjov from the case where the coefficients $P_{ij}$
in~\eqref{eq:pfaff-system} are polynomials to the general case
considered above. We outline the main results.

The key ingredient in Gabrielov-Vorobjov's effective theory of
Pfaffian structures, in addition to
Theorem~\ref{thm:khovnanskii-bound}, is the following local complex
analog of due to Gabrielov.

\begin{Def}
  A \emph{deformation} of format $\sF$ of a restricted Pfaffian
  function $F_0$ is a function germ $F(\vz,\e):(\C^{n+1},0)\to\C$
  analytic at the origin such that $F(\vz,0)=F_0$ and for every
  sufficiently small $\e$ we have $F(\cdot,\e)\in L$ where $L$ is a
  linear space of functions satisfying $L\subset\Omega^\cS_\sF$.
\end{Def}

With this definition we can state Gabrielov's result.

\begin{Thm}[Gabrielov's bound]\label{thm:gabrielov-mult-bound}
  Let $F_1,\ldots,F_n:(\C^{n+1},0)\to\C$ be restricted Pfaffian
  deformations of format at most $\sF$. Then there exists $r>0$ such
  that for all sufficiently small $\e$ the number of isolated points
  in
  \begin{equation}
    \{ \vz\in\C^n : \norm{z}<r \text{ and } F_1(\vz,\e)=\cdots=F_n(\vz,\e)=0\}
  \end{equation}
  is bounded by $\eff(\sF)$.
\end{Thm}
\begin{proof}
  This is quite similar to the proof given above for
  Theorem~\ref{thm:khovnanskii-bound}. Indeed, Gabrielov
  \cite[Theorem~2.1]{gabrielov:pfaff-mult} defines a sequence much
  like Khovanskii's $\tilde *$-sequence, which is the classical case
  is a polynomial sequence and in our case is a sequence of definable
  functions of format $\eff(\sF)$. The bound is then given by the
  number of isolated solutions of this sequence -- which is bounded by
  Bezout in Gabrielov's case, and by effective o-minimality in our
  cases.

  We remark that the main difference compared to Khovanskii's result
  is that Gabrielov needs to involve a derivative in the
  $\e$-direction in the definition of the $\tilde *$-sequence. To
  accommodate this we had to include the assumption that the
  deformations $F(\cdot,\e)$ all belong to a fixed linear space
  independent of $\e$, so that the derivative remains in this space
  and has the same bound on the format. In Gabrielov's case this is
  automatic (because polynomials degree at most $\sF$ form a linear
  space).
\end{proof}

We call the sets generated by quantifier free formulas in the language
of restricted Pfaffian functions over $\cS$ the ``semi-Pfaffian sets''
over $\cS$. We endow this collection of sets with a filtration
$\Omega_\rPF^\qf$, where we associate format to sets as described at
the end of~\secref{sec:intro-effective-omin} excluding the projection
axiom.

\begin{Thm}[\protect{\cite{gabrielov:closure}}]\label{thm:closure}
  Let $X$ be a semi-Pfaffian set over $S$ with quantifier free format
  $\sF$. Then $\bar X$ is semi-Pfaffian over $S$ of quantifier free
  format $\eff(\sF)$.
\end{Thm}
\begin{proof}
  The proof is the same as in the original, using
  Theorem~\ref{thm:gabrielov-mult-bound} to replace the classical
  version over the algebraic structure.
\end{proof}

A \emph{stratification} of a semi-Pfaffian set $X\subset\R^n$ is a
partition of $X$ as a disjoint union of smooth, not necessarily
connected, semi-Pfaffian subsets. The following is a generalization of
a theorem by Gabrielov and Vorobjov.

\begin{Thm}[\protect{\cite{gv:stratifications}}]\label{thm:stratification}
  Let $X$ be a semi-Pfaffian set over $S$ with quantifier free format
  $\sF$. Then there exists a stratification $X=\cup_j X_i$ where the
  quantifier free formats of the $X_i$ strata, and their number, is
  bounded by $\eff(\sF)$.
\end{Thm}
\begin{proof}
  The proof the same is in the original, using
  Theorems~\ref{thm:khovnanskii-bound}
  and~\ref{thm:gabrielov-mult-bound} to replace the classical versions.
\end{proof}

Finally we can state the theorem of the complement for $\cS_\rPF$.

\begin{Thm}[\protect{\cite{gv:cell-decomposition}}]\label{thm:subpfaff-complement}
  Let $X$ be semi-Pfaffian over $\cS$ with quantifier-free format
  $\sF$. Then $\R^n\setminus \pi_n(X)$ is the projection $\pi_n(Y)$ of
  a semi-Pfaffian set $Y\subset\R^N$ and the quantifier-free format of
  $Y$ is bounded by $\eff(\sF)$.
\end{Thm}
\begin{proof}
  This theorem is proved by showing that cell-decomposition is
  possible in the class of sub-Pfaffian sets (i.e., projections of
  semi-Pfaffian sets). The proof is based on
  Theorems~\ref{thm:khovnanskii-bound},~\ref{thm:closure}
  and~\ref{thm:stratification}. See also \cite{bv:rest-pfaff}
  for another proof that isolates the dependence of the argument on
  these three results.
\end{proof}

As a consequence of Theorem~\ref{thm:subpfaff-complement}, we deduce
that a every $\cL_\rPF$-formula $\psi$ is equivalent to an existential
formula $\exists y:\psi'(x,y)$ with $\sF(\psi')<\eff(\psi)$. Since the
number of connected components of a set defined by $\psi'$ is bounded
by Theorem~\ref{thm:khovnanskii-bound}, we obtain
Theorem~\ref{thm:effective-SrPF}.

\subsubsection{The unrestricted case: limit sets}

In \cite{gabrielov:limit-sets} Gabrielov develops an effective approach to
unrestricted Pfaffian functions. We define the analog of this
construction over $\R_\LN$.

In the following definition we fix some domain $G$ as
in~\eqref{eq:pfaff-domain} and a Pfaffian chain over it, and consider
the formats of restricted semi-Pfaffian sets with respect to this
chain. If $X\subset\R^n\times\R$ and $\lambda\in\R$ we denote
\begin{equation}
  X_\lambda := \{ \vx \in \R^n : (\vx,\lambda)\in X\}.
\end{equation}
\begin{Def}[A-Limit sets]
  Let $X\subset\R^n\times\R_{>0}$ be semi-Pfaffian over $\cS$, and
  suppose for every $\lambda\in\R_{>0}$ the set $X_\lambda$ is
  restricted semi-Pfaffian of format at most $\sF$. Then $(\bar X)_0$
  is called an A-limit. The format of $(\bar X)_0$ is defined to be
  $\sF$.
\end{Def}

Denote by $A(\R_{\LN,\rPF})$ the structure generated by the A-limit
sets and by $\Omega_A$ the filtration generated by the formats
above. The main result of \cite{gabrielov:limit-sets} is that
$A(\R_{\LN,\rPF})$ is effectively o-minimal. This is done by showing
that every set in this structure admits a representation as a union of
certain simple sets (\emph{relative closures}), and showing that the
format of this representation remains effectively bounded under all
boolean operations and projection, and that the format effectively
bounds the number of connected components.

\begin{Rem}
  In fact Gabrielov defines the notion ``limit set'' in a slightly
  different manner, but such that it is clear that they generate the
  same structure as the A-limit sets above. We introduced the
  non-standard terminology of A-limits to avoid potential confusion
  with his terminology.
\end{Rem}

The proof in \cite{gabrielov:limit-sets} goes through essentially
unchanged for semi-Pfaffian sets over $\R_\LN$. The only difference is
in \cite[Proposition~2.12]{gabrielov:limit-sets} which establishes the
exponential Lojasiewicz inequality for $\R_\PF$. But since $\R_\LN$ is
polynomially bounded (being a reduct of $\R_\an$) the same inequality
also holds in $\R_{\LN,\PF}$ by \cite[Theorem~3]{lms:diff-poly-bdd}.

We claim that $A(\R_{\LN,\rPF})$ is in fact $\R_{\LN,PF}$. First, we
show that all LN-functions are A-definable. Indeed, if $\cC^\delta$ is
an LN-cell and $f\in\cO_\LN(\cC^\delta)$ then for every $0<\e<1$ we
can define an LN-cell $\cC(\e)$ by replacing every $D_\circ(r)$ factor
by $A(\e r,r)$. Then the $C(\e)$ are relatively compact in $\cC^\delta$
and are thus definable in $\R_{\LN,\rPF}$. Moreover, their format is
independent of $\e$. Then $\cC$ is the A-limit of the family $\cC(\e)$
and is thus contained in $A(\R_{\LN,\rPF})$. A similar reasoning
applies to the graph of $f\rest\cC$.

In the same way one can show that the graphs of unrestricted Pfaffian
functions over $\R_\LN$ are A-limits. One slightly shrinks the domain
$G$ in \eqref{eq:pfaff-domain}, say by taking $I_k$ to be the interval
with the same center and length equal to $(1-\e)$ times the length of
$I_k$ (or some similar construction if the intervals are
infinite). Then the Pfaffian chain restricted to these precompact
domains is a restricted Pfaffian chain with format independent of
$\e$, and the A-limit set recovers the graph of the full unrestricted
Pfaffian function.

In conclusion the A-limit sets contain all functions generating
$\R_{\LN,PF}$ and thus in fact $\R_{\LN,PF}=A(\R_{\LN,\rPF})$ and the
effective o-minimality of this structure follows from Gabrielov's
result for A-limit sets.

\section{Geometric constructions with LN cells}

In this section we collect some useful geometric constructions with LN
cells. Our main motivation, which we establish at the end of the
section, is to show that complex LN cells and the LN functions on them
are definable in $\R_\LN$.

\subsection{Monomialization of cells}

Following \cite{me:c-cells} we define monomial LN cells.  Let $\cC$ be
an LN cell of length $\ell$. An \emph{admissible monomial} on $\cC$ is
a function of the form $c\cdot\vz^\valpha$ where $c\in\C$ and
$\valpha=\valpha(f)$ is the associated monomial of some function
$f\in\cO_\LN(\cC)$.

\begin{Def}[Monomial cell]\label{def:monom-cell}
  An LN cell $\cC$ is \emph{monomial} if $\cC=*$ or if
  $\cC=\cC_{1..\ell}\odot\cF$ where $\cC_{1..\ell}$ is monomial and
  the radii involved in $\cF$ are admissible monomials on
  $\cC_{1..\ell}$.
\end{Def}

We remark that it is always possible to choose all the constants in
the admissible monomials to be real. Therefore, every monomial cell is
a real cell.

The following proposition extends from \cite{me:c-cells} to the LN
category with the same proof. Essentially, by the monomialization
lemma every radius is given by a monomial times a unit. If we choose
the extensions sufficiently large (using the refinement theorem) then
the unit becomes nearly constant and one can just replace it by (say)
its maximum on the cell.

\begin{Prop}\label{prop:monomial-cover}
  Let $\cC^\delta$ be a (real) cell. Then there exists a (real)
  cellular cover $\{f_j:\cC_j^\delta\to\cC^\delta\}$ where each
  $\cC_j$ is a monomial cell. Moreover
  \begin{equation}
    \#\{f_j\},\sF(f_j) < \eff(\cC^\delta,1/(1-\delta)).
  \end{equation}
\end{Prop}

\subsection{Symmetrization}

\begin{Prop}\label{prop:symmetrization}
  Let $\cC$ be a real cell and $F\in\cO_\LN(\cC)$, not necessarily
  real. Then there exist real LN functions $F_R,F_I\in\cO_\LN(\cC)$
  with $\sF(F_R),\sF(F_I)<\eff(F)$ and
  \begin{align}
    F_R \rest{\R_+\cC} &\equiv \Re F\rest{\R_+\cC}, & & & F_I \rest{\R_+\cC} &\equiv \Im F\rest{\R_+\cC}.
  \end{align}
\end{Prop}
\begin{proof}
  Let $F_1,\ldots,F_N$ be the LN chain for $F$. Since $\cC$ is real,
  its radii are all real and it follows that it is symmetric under the
  map $\vz\to\bar\vz$. If we define $F_1^\dag,\ldots,F_N^\dag$ by
  \begin{equation}
    F_j^\dag(\vz) = \overline{F_j(\bar\vz)}
  \end{equation}
  then $F_1^\dag,\ldots,F_N^\dag$ are again bounded holomorphic
  functions on $\cC$. By symmetry we have LN chain equations for these
  function where the polynomials $G_{ij}$ in~\eqref{eq:LN-chain} are
  replaced by $G_{ij}^\dag$ (which just means taking complex
  conjugates of the coefficients).

  If $F=G(F_1,\ldots,F_N)$ then we put
  $F^\dag:=G^\dag(F_1^\dag,\ldots,F_N^\dag)$ and this is again an LN
  function. Then setting
  \begin{align}
    F_R &:= \frac{F+F^\dag}2, &&& F_I &:= \frac{F-F^\dag}2
  \end{align}
  we obtain real LN functions satisfying the required conditions.
\end{proof}

\subsection{Real covers of complex cells}

\begin{Prop}\label{prop:real-complex-cover}
  Let $\cC^\delta$ be an LN cell. Then there exist LN maps
  $\{f_j:\cC_j^\delta\to\cC^\delta\}$ with $\cC_j^\delta$ real and
  $\cC\subset \cup_j f_j(\R_+\cC_j)$. Moreover
  \begin{equation}
    \#\{f_j\},\sF(f_j) < \eff(\cC^\delta,1/(1-\delta)).
  \end{equation}
\end{Prop}
\begin{proof}
  By refinement and Proposition~\ref{prop:monomial-cover} it is enough
  to prove the claim for monomial cells. Write
  $\cC:=\cF_1\odot\cdots\odot\cF_\ell$ and let
  \begin{equation}
    \hat\cC := A(\tfrac12,2)^{\odot \ell} \odot \cC.
  \end{equation}
  Denote the coordinates on $\hat\cC$ by $(\vx,\vz)$. We have an LN map
  $\phi:\hat\cC^{2\delta}\to\cC^\delta$ given by
  \begin{equation}
    \phi(\vx,\vz) = (\vx_1\vz_1,\cdots,\vx_\ell\vz_\ell)
  \end{equation}
  where we use the fact that monomial cells are invariant under
  rotation in each of the coordinates. Then
  \begin{equation}
    \cC \subset \phi\big(A(\tfrac12,2)^{\odot\ell}\odot\R_+\cC\big)
  \end{equation}
  and it will therefore suffice to find maps
  $f_j:\cC_j^\delta\to\hat\cC^{2\delta}$ such that
  \begin{equation}
    A(\tfrac12,2)^{\odot\ell}\odot\R_+\cC \subset \cup_j f_j(\R_+\cC_j).
  \end{equation}
  Reducing further, it will suffice to find
  $g_j:\cC_j'\to A(1/2,2)^{\odot\ell}$ with
  $A(1/2,2)^{\odot\ell}\subset\cup_j g_j(\R_+\cC'_j)$ and then take
  $C_j:=C'_j\odot\cC$ and $f_j:=g_j\times\id$. Finally finding $g_j$
  as above is elementary. We can first cover each $A(1/2,2)$ by discs,
  and then use maps $D(1)^{\odot 2}\to D(1)^{\delta}$ given by
  $(z,w)\to(z+iw)$.
\end{proof}

Finally we have the following.

\begin{Prop}\label{prop:R_LN-complex-cells}
  Let $\cC^\delta$ be an LN-cell and $f\in\cO_\LN(\cC^\delta)$. Then
  $f\rest\cC$ is definable in $\R_\LN$ under the identification
  $\C^\ell\simeq\R^{2\ell}$ and its format is bounded by $\eff(f)$.
\end{Prop}
\begin{proof}
  Let $\{f_j:\cC_j^\delta\to\cC^\delta\}$ be as in
  Proposition~\ref{prop:real-complex-cover}. For each $f_j$ apply
  Proposition~\ref{prop:symmetrization} to the coordinates of $f_j$
  and to $f_j^*f$. Then on $\R_+\cC_j^{\sqrt\delta}$ the real and
  imaginary parts of all of these maps are given by LN functions of
  effectively bounded format. Thus the set
  \begin{equation}
    \cup_j \{ \big((f_j(\vz),f(f_j(\vz))\big) : \vz\in\cC_j^{\sqrt\delta} \}
  \end{equation}
  is definable in $\R_\LN$. This is in fact the graph of $f$ on some
  set $S$ with $\cC\subset S\subset\cC^\delta$. In other words, we
  have proven the proposition (for all $\cC^\delta,f$) with
  $f\rest\cC$ in the conclusion replaced by $f\rest S$.

  We now deduce the original statement from this modified statement by
  induction on $\ell$. Suppose the statement is true for $\ell-1$. In
  particular the cell $\cC_{1..\ell-1}$ is definable in $\R_\LN$. By
  the modified statement, the radii of the fiber over
  $\cC_{1..\ell-1}$ are definable restricted to some set containing
  $\cC_{1..\ell-1}$, and since $\cC_{1..\ell-1}$ is definable as well
  this implies that $\cC$ is definable. By the modified statement $f$
  is definable on some set containing $\cC$, so it is also definable
  restricted to $\cC$. The bound on the format follows easily from
  this.
\end{proof}

\section{Effective Pila-Wilkie in $\R_{\LN,\PF}$}
\label{sec:effective-pw}

For $\vx\in\bar\Q^n$ denote
\begin{equation}
  H(\vx) := \max_j H(\vx_j)
\end{equation}
where $H(\vx_j)$ denotes the absolute multiplicative Weil height. For
$X\subset\R^n$ and $g\in\N$ denote
\begin{equation}
  X(g,H) := \{ \vx\in X\cap\bar\Q^n : [\Q(x):\Q]\le g \text{ and
  } H(\vx)\le H\}.
\end{equation}
We also denote by $X^\alg$ the union of all connected
positive-dimensional $\R^\alg$-definable subsets of $X$, and
$X^\trans:=X\setminus X^\alg$.

The Pila-Wilkie theorem \cite{pw:thm} states that for $X$ definable in
an o-minimal structure and $\e>0$, there exists a constant $C(X,\e)$
such that $\#X^\trans(1,H)<C(X,\e) H^\e$ for every $H\in\N$. This has
been extended by Pila \cite{pila:blocks} to a bound
$\#X^\trans(g,H)<C(X,g,\e) H^\e$. When $X$ is definable in $\R_\PF$,
it was shown by Jones, Thomas, Schmidt and the author
\cite{me:effective-pfaff-pw} that the constant $C(X,g,\e)$ can be
bounded effectively in terms of the format of $X$. The proofs in
loc. cit. extend to general effectively o-minimal structures verbatim,
although we note that much of the work there is related to obtaining
polynomial bounds with respect to degrees in the restricted Pfaffian
context. The proofs would be somewhat simpler if one only wants
effective bounds. The effective analog of the Pila-Wilkie theorem is
as follows.

\begin{Thm}\label{thm:effective-pw}
  Let $X\subset\R^n$ be definable in an effectively o-minimal
  structure. Then for every $\e>0$ and every $g,H\in\N$ we have
  \begin{equation}
    \#X^\trans(g,H) \le C(X,g,\e) H^\e, \qquad \text{where } C(X,g,\e)=\eff(X,g,1/\e).
  \end{equation}
\end{Thm}

We also give a ``blocks'' version following \cite{pila:blocks}. Say
that a definable set $B\subset\R^n$ is a \emph{basic block} if it is a
smooth $k$-dimensional manifold contained in a semialgebraic set of
dimension $k$.

\begin{Thm}\label{thm:effective-pw-blocks}
  Let $X\subset\R^m\times\R^n$ be definable in an effectively
  o-minimal structure. Then for every $\e>0$ and $g\in\N$ there is a
  definable family $Y\subset\R^p\times\R^m\times\R^n$ of format
  $\eff(X,g,1/\e)$ such that for every $(p,y)\in\R^p\times\R^m$ the
  fiber $Y_{p,y}$ is a basic block contained in $X_y$.

  Moreover, for every $y\in\R^m$ and $H\in\N$ there is a set
  $P_y\subset\R^p$ such that
  \begin{equation}
    \# P_y \le C(X,g,\e) H^\e, \qquad \text{where } C(X,g,\e)=\eff(X,g,1/\e)
  \end{equation}
  and
  \begin{equation}
    X_y(g,H) \subset \bigcup_{p\in P_y} Y_{p,y}.
  \end{equation}
\end{Thm}

An effective Pila-Wilkie theorem for semi-Noetherian sets was given in
\cite{me:noetherian-pw}. This result was restricted to sets defined by
quantifier-free formulas using Noetherian functions in compact domains
(in particular, not allowing unrestricted exponentiation). Another
result establishing sharper, polylogarithmic bounds in the Noetherian
category is given in \cite{me:qfol-geometry}. However this result is
again restricted to compact domains, and also has technical conditions
related to the absence of unlikely intersections. It was observed
already in \cite[Section~1.5]{me:qfol-geometry} that this limitation
involving unlikely intersections is related to Khovanskii's
conjecture.

\section{LN-functions and regular flat connections}
\label{sec:regular-conn}

In this section we show that holomorphic horizontal sections of
regular meromorphic connections with LN coefficients are LN-functions,
and that if the connection has quasiunipotent monodromy then all
sections are definable in $\R_{\LN,\exp}$ after making an appropriate
branch cut. As a consequence of Deligne's Riemann-Hilbert
correspondence we deduce that $\R_{\LN,\exp}$ contains every
monodromic tuple of functions with moderate growth and locally
quasiunipotent monodromy. We also show as a consequence of these
constructions that period maps for PVHS, and the universal covering
map for the universal abelian scheme $\A_g\to\cA_g$, are definable in
$\R_{\LN,\exp}$.

\subsection{Connections with log-singularities on LN-cells}
\label{sec:log-connections}

Let $\cC^\delta$ be an LN-cell. Consider a connection on the trivial
$\C^\ell$-bundle on $\cC^\delta$ given by
\begin{equation}
  \nabla := \d-\sum_{j=1}^\ell A_j \partial_j^*, \qquad A_j\in\Mat_{l\times l}(\cO_\LN(\cC^\delta)).
\end{equation}
where $\partial_j^*$ is the one-form dual to $\partial_j$, i.e.
\begin{equation}
  \partial^*_j :=
  \begin{cases}
    \d\vz_j & \cF_j=D \\
    \frac{\d\vz_j}{\vz_j} & \cF_j=D_\circ,A
  \end{cases}
\end{equation}
where $\cF_j$ denotes the $j$-th fiber of $\cC$. We say that the
format of a matrix with LN entries is the sum of the formats of the
entires, and the format of $\nabla$ is the sum of the formats of $A_j$
for $j=1,\ldots,\ell$. In the case where $\cC$ is a product of discs
and punctured discs this agrees with the usual notion of a connection
with log-singularities.

\subsubsection{Holomorphic sections}

We will consider horizontal sections of $\nabla$. We start with the
holomorphic case.

\begin{Prop}\label{prop:holomorphic-section}
  Suppose $X(\vz)$ is holomorphic and bounded in $\cC^{\delta}$ and
  $\nabla X(\vz)=0$, where either $X(\vz)\in\C^l$ or
  $X(\vz)\in\End(\C^n)$. Then $X$ is LN and
  \begin{equation}
    \sF(X) = \eff(\nabla,\norm{X}_{\cP^{1/2}}).
  \end{equation}
\end{Prop}
\begin{proof}
  This is essentially just rewriting the connection equation in terms
  of the standard vector fields,
  \begin{equation}
    \partial_j X = A_j \partial_j^*(\partial_j) X = A_j X
  \end{equation}
  and the entries of $X$ indeed form an LN chain over the LN chain
  defining $\nabla$.
\end{proof}

\begin{Cor}\label{cor:holomorphic-section}
  Suppose $X(\vz)$ is holomorphic and bounded in $\cC^\delta$ and
  $\nabla X(\vz)=0$. Then $X$ is definable in $\R_\LN$ and its format
  is bounded by $\eff(\nabla)$.
\end{Cor}
\begin{proof}
  Since $\nabla X=0$ is linear in $X$ we may as well rescale it and
  assume $\norm{X}_{\cC^\delta}\le 1$. We then have a bound
  $\eff(\nabla)$ on the format of this rescaled $X$, and now scaling
  back gives the result.
\end{proof}

\subsubsection{The quasiunipotent monodromy case}

We say that the \emph{standard branch cut} on $\cC$ is the simply
connected domain $\cC_\SC\subset\cC$ given by removing the negative
real line from every $D_\circ$ and $A$ fiber in $\cC$.

\begin{Thm}\label{thm:unipotent-sections-definable}
  Suppose $\nabla$ has quasiunipotent monodromy. Let
  $X:\cC^\delta_\SC\to\End(\C^l)$ be a horizontal section of $\nabla$,
  i.e. $\nabla X=0$. Then $X$ is definable in $\R_{\LN,\exp}$ and its
  format is bounded by $\eff(\nabla)$.
\end{Thm}
\begin{proof}
  Fix some base point $s\in\cC$ and let
  $M_1,\ldots,M_\ell\in\End(\C^n)$ where $M_j$ denotes the monodromy
  operator of $\nabla$ along a simple loops around the divisor
  $\vz_j=0$ (trivial for $D$ or $*$ fibers). Denote
  \begin{equation}
    t(\vz) := \prod_{j:\cF_j=D_\circ,A} \vz_j.
  \end{equation}
  Note that since $\pi_1(\cC,s)$ is commutative the $M_j$s
  commute. Moreover since $\sF(\nabla)$ bounds the norm of the
  matrices $A_j$, it is easy to see using Gr\"onwall's inequality that
  $\norm{M_j}<\eff(\nabla)$. Also by Gr\"onwall's inequality, for
  every $\vz\in\cC^\delta_\SC$ we have
  \begin{equation}
    \norm{X(\vz)} \le t(\vz)^{-N}
  \end{equation}
  where $N=\eff(\nabla)$. Indeed, working for example with respect to
  $\vz_j$ and assuming it is of type $D_\circ,A$, we have an equation
  \begin{equation}
    \partial_j X = A_j X
  \end{equation}
  where $A_j$ is bounded on $\cC^\delta$, say by some
  $M=\eff(\nabla)$. In the logarithmic chart $w=\log\vz_j$ this
  amounts to $\pd{}w X = A_jX$ and Gr\"onwall then gives
  \begin{equation}
    \norm{X(w)} < e^{M|w|} \norm{X(0)} = |\vz_j|^{-M} \norm{X(0)}.
  \end{equation}
  Repeating this for each coordinate proves the claim.

  By \cite[Lemma~IV.4.5]{borel:d-modules} we may choose commuting
  logarithms $L_j$ such that
  \begin{equation}
    e^{2\pi i L_j} = M_j
  \end{equation}
  and $\norm{L_j}<\eff(\nabla)$. Now consider the matrix function
  \begin{equation}
    Y(\vz) := X(\vz)\cdot \vz_1^{-L_1}\cdots\vz_n^{-L_\ell}.
  \end{equation}
  Since the monodromy of $\vz_j^{-L_j}$ around $\vz_j=0$ is given by
  $M_j^{-1}$ and this matrix commutes with every $\vz_j^{-L_j}$, we
  see that $Y(\vz)$ is univalued on $\cC$. Adding an appropriate
  integer multiple of the identity (depending on $N$ above) to $L_j$
  we can also arrange the $Y(\vz)$ is bounded on $\cC^\delta_\SC$, and
  still $\norm{L_j}<\eff(\nabla)$.

  Now we think of $Y(\vz)$ as a flat holomorphic section of the linear
  connection
  \begin{equation}\label{eq:Y-system}
    \d Y = \d(X(\vz)\cdot \vz_1^{-L_1}\cdots\vz_\ell^{-L_n}) = A Y - Y \big(\sum_{j=1}^\ell L_j \frac{\d\vz_j}{\vz_j}\big)
  \end{equation}
  where we think of $Y$ as a vector in $\C^{l^2}$. In this
  sense~\eqref{eq:Y-system} is a connection equation with logarithmic
  singularities. By Corollary~\ref{cor:holomorphic-section} we
  conclude that $Y(\vz)$ is definable in $\R_\LN$ with format
  bounded by $\eff(\nabla)$.

  Finally the entries of $\vz_j^{L_j}$ are each of the form
  $\vz_j^{\lambda} P(\log\vz_j)$ where $\lambda$ ranges over the
  spectrum of $L_j$ (modulo $2\pi i$) and $P$ is a polynomial of
  degree at most $l$. Since we assume the monodromy is quasiunipotent
  these $\lambda$ can be assumed real, and the functions above are all
  definable in $\R_{\exp}$ with format depending only on $l$. We
  recover $X$ as
  \begin{equation}
    X(\vz) = Y(\vz)\cdot\vz_1^{L_1}\cdots\cdot\vz_\ell^{L_\ell}.
  \end{equation}
  which finishes the proof.  
\end{proof}

\subsection{Riemann-Hilbert correspondence and monodromic functions}
\label{sec:riemann-hilbert}

Let $M$ be a smooth quasi-projective variety, and suppose
$\bar M\setminus M$ is a normal-crossings divisor. Let $U\subset M$ a
semialgebraic simply-connected subset. Let $\cV$ be a local system on
$M$, by which we always mean a local system of finite-dimensional
$\C$-vector spaces.

\begin{Def}
  We say that $\cV$ has locally quasiunipotent monodromy if for any
  map $\phi:D_\circ(1)\to M$ the monodromy of $\phi^*\cV$ is
  quasi-unipotent.
\end{Def}

By Kashiwara's theorem \cite{kashiwara:quasiunipotent} a system $\cV$
is locally quasiunipotent if and only if it has quasiunipotent
monodromy around every smooth point of the boundary
$\bar M\setminus M$.

According to Deligne's Riemann-Hilbert correspondence
\cite{deligne:diff-eqs}, every local system $\cV$ on $M$ arises as the
horizontal system $\cV\simeq V^\nabla$ of an algebraic vector bundle
$V$ with a flat connection $\nabla$ with logarithmic singularities
along the boundary. The results of the previous sections imply the
following.

\begin{Prop}\label{prop:constant-sections}
  The following two statements hold:
  \begin{enumerate}
  \item If $U$ is precompact in $M$ then the constant sections of
    $V^\nabla\rest U$ are definable in $\R_\LN$.
  \item If $\cV$ has locally quasiunipotent monodromy then the
    constant sections of $V^\nabla\rest U$ are definable in
    $\R_{\LN,\exp}$.
  \end{enumerate}
\end{Prop}
\begin{proof}
  We can cover $M$ by (finitely many) neighborhoods given by products
  of discs and punctured discs. By a simple covering argument it is
  enough to consider the intersection of $U$ with each of these
  separately. Then the first claim is just a reformulation of
  Proposition~\ref{prop:holomorphic-section} and the second a
  reformulation of Theorem~\ref{thm:unipotent-sections-definable}.
\end{proof}

We reformulate this as a result about the $\R_{\LN,\exp}$-definability
of multivalued holomorphic functions in terms of growth and
monodromy.

Let $f=(f_1,\ldots,f_n):M\to\C^n$ be a holomorphic map (formally a map
on the universal cover of $M$). We say that $f$ is \emph{monodromic}
if fixed $s_0\in M$ and any $\gamma\in\pi_1(M,s_0)$ we have
\begin{equation}
  \Delta_\gamma (f_1,\ldots,f_n) = (f_1,\ldots,f_n) M_\gamma, \qquad M_\gamma\in\GL_n(\C)
\end{equation}
where $\Delta_\gamma$ is the analytic continuation operator along
$\gamma$. We say that $f$ has \emph{moderate growth} if whenever
$u:D_\circ(1)\to M$ there are $C,N$ such that
\begin{equation}
  \norm{f\circ u(t)} < C t^{-N} \qquad \forall t\in (D_\circ)_\SC.
\end{equation}

\begin{Thm}
  Let $f:M\to\C^n$ be a monodromic tuple with locally quasiunipotent
  monodromy and moderate growth. Then $f\rest U$ is definable in
  $\R_{\LN,\exp}$.
\end{Thm}
\begin{proof}
  Let $\cV$ be the local system corresponding to the monodromy
  representation of $f$. Let $(V,\nabla)$ be the corresponding
  connection as in Proposition~\ref{prop:constant-sections}. It is
  enough to prove the claim with $M$ replaced by an affine cover, so
  we may assume without loss of generality that $V$ is trivial over
  $M$. In a trivializing chart a fundamental solution matrix
  $\nabla Y=0$ satisfies
  \begin{equation}
    \Delta_\gamma Y = Y M_\gamma
  \end{equation}
  so that $f\cdot Y^{-1}$ is univalued. Since it also has moderate
  growth, it is algebraic (and hence definable in $\R_\alg$) by
  GAGA. Since $Y$ is definable in $\R_{\LN,\exp}$ by
  Proposition~\ref{prop:constant-sections} we conclude that $f$ is
  definable in $\R_{\LN,\exp}$.
\end{proof}

\subsection{Period maps are definable in $\R_{\LN,\exp}$}
\label{sec:period-maps}

Let $S$ be a smooth quasi-projective variety over $\C$, $\cV_\Z$ a
local system of free $\Z$-modules, $\V$ the corresponding algebraic
vector bundle, and $\F^\bullet\subset\V$ a filtration forming a
polarized variation of $\Z$-Hodge structures. By resolution of
singularities we may assume that $\bar S\setminus S$ is a normal
crossings divisor. By compactness of $\bar S$ we can cover $\bar S$ by
local charts of the form
\begin{equation}
   \cP := D_\circ(1)^{\odot n}\odot D(1)^{\odot m}.
\end{equation}
with $\bar S\setminus S$ given in $\cP$ by $\vz_1\cdots\vz_n=0$. By
Borel's monodromy theorem \cite[(4.5)]{schmid:variation-hodge} the
connection $\nabla$ corresponding to $\cV_\Z$ has quasiunipotent
monodromy in $\cP$, and by a theorem
\cite[(4.13)]{schmid:variation-hodge} of Griffiths $\nabla$ has
logarithmic singularities on $\bar S\setminus S$.

Pick a basepoint $s_0\in\cP$ and let $b_1,\ldots,b_l\in(\cV_\Z)_{s_0}$ be a
basis for the free $\Z$-module $\cV_\Z$ at $s_0$. Let $v_1,\ldots,v_l$
denote a basis of algebraic sections of $\V$ over $\cP$ compatible
with the filtration $\F^\bullet$, i.e. such that the first vectors
form a basis for the first filtered piece, etc. Then the matrix
\begin{equation}
  X:\cP_\SC\to \GL_l(\C), \qquad X(s) := (v_i(b_j))_{i,j=1,\ldots,l}
\end{equation}
forms a horizontal section of $\nabla$ as we analytically continue
from $s_0$ to $s\in\cP_\SC$. By
Theorem~\ref{thm:unipotent-sections-definable} the map $X(s)$ is
definable in $\R_{\LN,\exp}$.

Denote by $\check D$ the variety corresponding to the Hodge flags of
dimension $\dim\F^\bullet$, and $q:\GL_m\to\check D$ the quotient map
sending the matrix $X$ to the flag with the $k$-th piece given by the
span of the first $\dim \F^k$ columns. Clearly $q$ is definable even
in $\R_\alg$, so the map $q\circ X$ is also definable in
$\R_{\LN,\exp}$.

Since the VHS $(S,\cV_\Z,\F^\bullet)$ is polarized the map above
restricts to a map $q\circ X:\cP_\SC\to D$ where $D\subset\check D$
denotes the open subspace of Hodge flags polarized by the given
polarization on $\V_{s_0}$. The monodromy of $\nabla$ on $S$ induces a
monodromy subgroup $\Gamma\subset\GL(\V_{s_0},\Z)$ and factoring
$q\circ X$ modulo this monodromy we obtain a well-defined map
\begin{equation}
  \Phi : \cP \to D / \Gamma.
\end{equation}
In \cite{bkt:tame} it is shown that this map is definable when one
chooses an appropriate $\R_\alg$-structure on
$D/\Gamma$. Essentially this is the mod-$\Gamma$ quotient of
the standard $\R_\alg$-structure on Siegel domains in
$D$. Accordingly, the statement that $\Phi$ is definable in
$\R_{\LN,\exp}$ follows from the fact that
$q\circ X:\cP_\SC\to D$ is definable in $\R_{\LN,\exp}$, plus
the fact that the image of $q\circ X$ lies in the union of finitely
many Siegel domains. This latter fact is established in
\cite[Theorem~1.5]{bkt:tame}, which thus proves the definability of
$\Phi$ in $\R_{\LN,\exp}$.

\begin{Rem}[The format of $\Phi$]\label{rem:period-map-format}
  The computation of the format of $\Phi$ in $\R_{\LN,\exp}$ involves
  two separate steps. First, one should effectively determine the
  Gauss-Manin connection in order to compute $\sF(X)$. If the PVHS is
  given by the Gauss-Manin connection of a projective smooth family of
  $f:V\to S$ then this computation is carried out in
  \cite{urbanik:bdd-degree}. In order to deduce a bound for
  $\sF(\Phi)$ one should then also produce an effective bound on the
  Siegel domains needed in \cite[Theorem~1.5]{bkt:tame}, which appears
  to be a non-trivial task.

  In the rare case that the base $S$ of the PVHS is projective, there
  are no singularities and $\Phi$ is in fact definable in $\R_\LN$. In
  this case it is also far less difficult to estimate the number of
  Siegel domains that the image of $q\circ X$ meets in terms of the
  format of the Gauss-Manin connection. 
\end{Rem}

We are now in position to finish the proof of
Theorem~\ref{thm:hodge-locus-bound}.

\begin{proof}[Proof of Theorem~\ref{thm:hodge-locus-bound}]
  By \cite[Theorem~1.1]{bkt:tame} the set $Y$ is $\R_\alg$-definable,
  so $S_Y$ above is $\R_{\LN,\exp}$-definable with the format
  depending effectively on $\Phi,Y$. The algebraicity of $S_Y$ then
  follows from the definable Chow theorem of Peterzil-Starchenko
  \cite[Theorem~4.4]{ps:complex-omin} as explained in
  \cite[Theorem~1.6]{bkt:tame}. We can now define the irreducible
  components of $S_Y$ (as an algebraic variety) as the closures of the
  connected components of $S_Y\setminus (S_Y)_{\text{sing}}$. By
  effective o-minimality the number of these irreducible components,
  and the format of each component, are effectively bounded. Finally
  for each irreducible component of dimension $k$ the degree is given
  by the number of intersections with $k$ generic hyperplane sections,
  which is again bounded by effective o-minimality.
\end{proof}

\subsection{The Siegel modular variety}
\label{sec:Ag}

As a special case of the construction of~\secref{sec:period-maps}, the
universal covering map of the Siegel modular variety $\cA_g$ of
principally polarized abelian varieties of genus $g$, restricted to an
appropriate fundamental domain $F\subset\H_g$, is definable in
$\R_{\LN,\exp}$. To see this let $\A_g\to\cA_g$ be the universal
abelian variety over $\cA_g$ and denote by $\nabla$ the corresponding
Gauss-Manin connection. Choose a compactification $\bar\cA_g$ such
that $\bar\cA_g\setminus\cA_g$ is a normal crossings divisor. In a
neighborhood $U$ of every point $p\in\bar\cA_g$ one can choose a basis
of Shimura differentials $\w_1,\ldots,\w_g$ generating a basis for the
holomorphic differentials over $U$, and complete it to a basis of
$H^1(\A_g/U)$ using meromorphic differentials of the 2nd kind
$\w_{g+1},\ldots,\w_{2g}$. If one chooses a symplectic basis
$\delta_1,\ldots,\delta_{2g}$ with respect to the principal
polarization for the local system $H_1(\A_g/U,\Z)$ then the extended
period matrix
\begin{equation}
  X(q) :=
  \begin{pmatrix}
    \oint_{\delta_1}\w_1 & \cdots & \oint_{\delta_{2g}} \w_1 \\
    \vdots & \ddots & \vdots \\
    \oint_{\delta_1}\w_{2g} & \cdots & \oint_{\delta_{2g}} \w_{2g} \\
  \end{pmatrix}
\end{equation}
is a horizontal section of the Gauss-Manin connection and therefore
definable in $\R_{\LN,\exp}$. Writing
\begin{equation}
  X(q) = \begin{pmatrix}
    A & B \\ C & D
  \end{pmatrix}
\end{equation}
where $A,B,C,D$ are $g\times g$ blocks, the inverse of the universal
cover $e:\H_g\to\cA_g$ is given locally by $B^{-1}A:\cA_g\to\H_g$. It
is known, for instance by \cite{ps:theta-definability} that the image
of this map, after making suitable algebraic branch cuts in $\cA_g$,
meets finitely many translates of the standard fundamental domain
$F\subset\H_g$. It follows that the graph of $\pi:F\to\cA_g$ can be
defined in $\R_{\LN,\exp}$ by gluing together finitely many translates
of the (inverted) graph of $B^{-1}A$. We remark on issues related to
effectivity of this construction at the end
of~\secref{sec:universal-abelian}.

In the special case $g=1$ one obtains from the construction above the
definability of the modular $\lambda$-function $\lambda:\H\to\C$
associated to the Legendre family
\begin{equation}
  E_\lambda := \{y^2=x(x-1)(x-\lambda)\}
\end{equation}
as the inverse to the ratio $I_1(\lambda)/I_2(\lambda)$ of the two
elliptic integrals
\begin{equation}
  I_j := \oint_{\delta_j(\lambda)} \frac{\d x}y, \qquad j=1,2
\end{equation}
where $\delta_1,\delta_2$ form a symplectic basis for
$H_1(E_\lambda,\Z)$. In this case there is no difficulty working out
an upper bound for the format of the $\lambda$-function restricted to
its standard fundamental domain by hand. One should examine the
asymptotics of $I_1/I_2(\cdot)$ to determine how many translates of
the fundamental domain meet the image of a ball around
$\lambda=0,1,\infty$ with a branch cut along the negative real axis,
and cover the rest of $\P^1\setminus\{0,1,\infty\}$ with finitely many
discs. Using
\begin{equation}
  j = \frac{256(1-\lambda+\lambda^2)^3}{\lambda^2(1-\lambda)^2}
\end{equation}
it then follows that the Klein modular invariant $j:\H\to\C$ is also
definable in $\R_{\LN,\exp}$ with effectively bounded format. Thus the
structure $\R_j$ generated by the modular invariant is effectively
o-minimal, as a substructure of $\R_{\LN,\exp}$. This answers a
question raised by Pila \cite[Section~13.3]{pila:andre-oort}, who
asked for the effective o-minimality of the structure $\R_j$ and noted
some implications for effective Andre-Oort.

\subsection{The universal abelian variety over $\cA_g$}
\label{sec:universal-abelian}

In this section we show that the universal cover for the universal
abelian variety over $\cA_g$ restricted to an appropriate fundamental
domain is definable in $\R_{\LN,\exp}$. It is possible to compute this
directly similarly to the construction of the previous section, but
for variation we demonstrate this definability using the material
of~\secref{sec:riemann-hilbert}.

Let $\pi:\A_g\to\cA_g$ be the universal abelian variety over
$\cA_g$. Denote by $*:\cA_g\to\A_g$ the identity section. Choose a
basis $\w_1,\ldots,\w_g\in\Omega(\A_g/U)$ over an affine open
$U\subset\cA_g$. Define maps $I_1,\ldots,I_g$ by
\begin{equation}
  I_j:\A_g\to\C^{2g+1}, \qquad I_j(p)=\big(\oint_{\gamma_1(\pi(p))}\w_j,\ldots,\oint_{\gamma_{2g}(\pi(p))}\w_j, \int_*^p \w_j \big)
\end{equation}
where the path of integration in the first integrand is chosen inside
$\pi^{-1}(\pi(p))$ and
\begin{equation}
  \gamma_1,\ldots,\gamma_{2g}\in H_1(\pi^{-1}(\pi(p)),\Z)
\end{equation}
form a basis, chosen over some basepoint and then analytically
continued.

It is classical that the maps $I_j$ have moderate growth. Moreover
they are monodromic with locally quasiunipotent monodromy: the first
$2g$ coordinates realize the monodromy of the Gauss-Manin connection
of $\pi$, and the last coordinate is univalued modulo the previous
coordinates. The results of~\secref{sec:riemann-hilbert} then imply
that the restriction of $I_j$ to a simply-connected semialgebraic
subset $(\A_g)_\SC$ of $\A_g$ is $\R_{\LN,\exp}$ definable. Consider
the universal covering map
\begin{equation}
  e:\H_g\times\C^g \to \A_g.
\end{equation}
As before, we can describe the inverse of $e$ on $(\A_g)_\SC$ in terms
of the coordinates of $I_1,\ldots,I_g$: on $\H_g$ it is given as
in~\secref{sec:Ag}, and on $\C^g$ it is given by the last coordinates of
$I_1,\ldots,I_g$. It follows from \cite{ps:theta-definability} that
$e^{-1}((\A_g)_\SC)$ meets finitely many translates of the standard
fundamental domain $\cF\subset\H_g\times\C^g$. Gluing together
finitely many translates we deduce the following.

\begin{Thm}
  The restriction $e\rest\cF$ is definable in $\R_{\LN,\exp}$.
\end{Thm}

Note that the computation of $\sF(e\rest\cF)$ involves two non-trivial
tasks. First one should show how to explicitly compute (or estimate
from above) the differential equations for $I_j$: this is an additive
extension of the Gauss-Manin connection and should be algebraically
computable in principle. One should then also effectivize the
Peterzil-Starchenko bound on the number of fundamental domains, which
is an interesting problem that we plan to return to.

\appendix
\section{Noetherian functions over number fields}

In this appendix we consider Noetherian functions ``defined over a
number field''. We use this in particular do produce some natural
examples of LN-functions that are \emph{not} Noetherian, thus
demonstrating the necessity of considering the larger LN class.

\subsection{Geometric formulation of Noetherian functions}

For the purposes of this section we reformulate the notion of a
Noetherian chain~\eqref{eq:intro-chain} in a more geometric
language. We will say that a ring $R$ is a ring of Noetherian
functions in $n$ variables if:
\begin{enumerate}
\item $R$ is finitely generated over $\Z$ by elements
  $F_1,\ldots,F_N\in R$;
\item $R$ is an integral domain.
\item $R$ is equipped with $n$ commuting derivations
  $\xi_1,\ldots,\xi_n\in \Der R$.
\end{enumerate}
Suppose $R$ satisfies these conditions and let
$p\in\Spec R\otimes_\Z\C$, thought of as a ring homomorphism
$p:R\to\C$. Then we can form a Noetherian chain~\eqref{eq:intro-chain}
given by $F_1,\ldots,F_N:\cP\to\C$ on some sufficiently small polydisc
$\cP$ with
\begin{equation}
  \pd{F_i}{\vz_j} = \xi_j(F_i)
\end{equation}
such that $F_i(0)=p(F_i)$. This is just the Frobenius integrability
theorem for commuting vector fields. We denote the germs of these
functions by $F_1^p,\ldots,F_N^p:(\C^n,0)\to\C$.

Conversely, let $F1,\ldots,F_N:\cP\to\C$ be a Noetherian
chain~\eqref{eq:intro-chain}. Adding all coefficients of the $G_{ij}$
to the chain (and defining their derivatives to be zero), we may
assume without loss of generality that the coefficients of $G_{ij}$
are in $\Z$. Then the subring $R$ of the ring of holomorphic functions
on $\cP$ generated by $F_1,\ldots,F_N$ is a finitely generated
integral domain equipped with the $n$ commuting derivations
$\xi_j=\pd{}{z_j}$.

\subsection{Noetherian functions over a field}

Let $\K$ be a field of characteristic zero. We are mostly interested
in the case $[\K:\Q]<\infty$ but the results of this subsection do not
require this assumption. If $p\in\Spec R$ is a $\K$-point then we say
that the functions $F_1^p,\ldots,F_N^p$ are defined over $\K$. We say
that $F$ is defined over $\K$ if it is a polynomial combination of
$F_1^p,\ldots,F_N^p$ with integer coefficients.

\begin{Prop}\label{prop:K-noetherian}
  Suppose that $F$ is a Noetherian function and its Taylor
  coefficients at the origin are all in $\K$. Then it is defined over
  $\bar\K$ as a Noetherian function.
\end{Prop}
\begin{proof}
  Let $R$ be a ring of Noetherian functions generated by
  $F_1,\ldots,F_N$ and $F=P(F_1,\ldots,F_N)$. We may assume without
  loss of generality that the coefficients of $P$ are integers, by
  including any non-integer coefficient as additional functions in the
  chain $F_1,\ldots,F_N$. For $q\in\Spec R$ write $F^q$ for
  $P(F_1^q,\ldots,F_N^q)$.
  
  We have $F\equiv F^q$ for some $q\in\Spec R$ by definition. The
  set of all $p\in\Spec R$ such that $F\equiv F^p$ is a
  $\K$-variety: indeed, it is given by the conditions
  \begin{equation}
    \pd{^\valpha}{z^\valpha} F(0) = \pd{^\valpha}{z^\valpha} F^p(0) = p(\vxi^\valpha F), \qquad \forall\valpha\in\N^n.
  \end{equation}
  The expressions $\vxi^\alpha F$ on the right hand side are
  polynomials over $\Z$ in $F_1,\ldots,F_N$, and the expressions on
  the left hand side are in $\K$, so this is a collection of
  polynomial equations over $\K$. Since the solution set is non-empty,
  it must also contain a solution $p$ over $\bar\K$ by the
  Nullstellensatz. This finishes the proof.
\end{proof}

We also record a simple consequence of the definitions.

\begin{Prop}\label{prop:K-noetherian-valuation}
  Let $F$ be a Noetherian function over $\K$, that is
  $F=P(F^p_1,\ldots,F^p_N)$ where $P$ has integer coefficients and
  $p\in\Spec R$ is a $\K$-point. Let $\nu$ be a valuation on $\K$ with
  \begin{equation}
    \nu(F_j^p(0))\le 0, \qquad \text{for }j=1,\ldots,N.
  \end{equation}
  Then
  \begin{equation}
    \nu\big(\pd{^\valpha}{z^\valpha} F(0)\big) \le 0 \qquad \forall\valpha\in\N^n.
  \end{equation}
\end{Prop}
\begin{proof}
  This is simply because all derivatives of $F_j^p$ are given by
  polynomials in $F_1^p,\ldots,F_N^p$ with integer coefficients, and
  at the origin these all belong to the valuation ring of $\nu$.
\end{proof}

\subsection{An example of an LN function that is not Noetherian}
\label{sec:no-noetherian-restricted-div}

Let $\K$ be a number field, and
\begin{equation}
  f(\vz) = \sum \frac{c_\valpha }{\valpha!} \vz^\valpha
\end{equation}
be a Noetherian function with $c_\valpha\in\K$ for every
$\valpha\in\N^n$. Then by Proposition~\ref{prop:K-noetherian}, $f$ is
Noetherian over some finite extension $\K'\supset\K$. By
Proposition~\ref{prop:K-noetherian-valuation} we have
$f=P(F_1^p,\ldots,F_M^p)$ where $P$ has integer coefficients and
$p\in\Spec R$ is a $\K$-point. Denote by $\Sigma'$ the finite set of
places of $\K'$ corresponding to a valuation $\nu$ such that
$\nu(F_j^p(0))>0$ for some $j=1,\ldots,N$. Denote by $\Sigma_f$ the
set of restrictions of these places to $\K$. Then we have the
following corollary.

\begin{Cor}
  For every valuation $\nu$ on $\K$ corresponding to a place outside
  the finite set $\Sigma_f$,
  \begin{equation}
    \nu(c_\alpha)\le 0 \qquad \forall \alpha\in\N^n.
  \end{equation}
\end{Cor}

As a consequence, the function
\begin{equation}
  f(z) = \frac{e^z-1}z = \sum_{j=0}^\infty \frac{1/(j+1)}{j!} z^j
\end{equation}
is not Noetherian: the coefficients $c_j=1/(j+1)$ lie outside the
valuation ring for infinitely many (in fact all) $p$-adic valuations
on $\Q$. In particular, the Noetherian class is not closed under
restricted divisions (or strict transforms). On the other hand $f$ is
certainly LN on $D_\circ(1)$, for instance by closedness under
restricted division, or directly by the Noetherian chain $f,z,e^z$
satisfying
\begin{align}
  \partial_z(f) &= e^z-f, & \partial_z(z)&=z, & \partial_z(e^z)&=z e^z,
\end{align}
with $\partial_z=z\pd{}z$.

\subsection{An LN function that does not generate a finitely-generated
  ring}
\label{sec:LN-func-not-fin-gen}

We close with a simple example showing that an LN function need not
generate, by closing under derivatives, a finitely generated ring. In
other words, to check that a function is not LN it does not suffice to
check that the ring it generates is not finitely generated. Consider
\begin{equation}
  f(z) := e^{z^2}
\end{equation}
on $D(1)$. By straightforward computation
\begin{equation}
  \C[f,\partial_z f,\cdots,\partial_z^k f] = \C[f,zf,\cdots,z^kf].
\end{equation}
Since $z,f$ are algebraically independent over $\C$ the chain of rings
above does not stabilize as $k$ grows, since $z^{k+1}f$ is clearly not
a polynomial in $f,zf,\cdots,z^k f$. 

\bibliographystyle{plain} \bibliography{nrefs}

\begin{thebibliography}{10}

\bibitem{bkt:tame}
B.~Bakker, B.~Klingler, and J.~Tsimerman.
\newblock Tame topology of arithmetic quotients and algebraicity of {H}odge
  loci.
\newblock {\em J. Amer. Math. Soc.}, 33(4):917--939, 2020.

\bibitem{bs:effective-omin}
Alessandro Berarducci and Tamara Servi.
\newblock An effective version of {W}ilkie's theorem of the complement and some
  effective o-minimality results.
\newblock {\em Ann. Pure Appl. Logic}, 125(1-3):43--74, 2004.

\bibitem{me:noetherian-pw}
Gal Binyamini.
\newblock Density of algebraic points on {N}oetherian varieties.
\newblock {\em Geometric and Functional Analysis}, 29(1):72--118, 2019.

\bibitem{me:qfol-geometry}
Gal Binyamini.
\newblock Point counting for foliations over number fields.
\newblock {\em Forum of Mathematics, Pi}, 10:e6, 2022.

\bibitem{me:effective-pfaff-pw}
Gal Binyamini, Gareth~O Jones, Harry Schmidt, and Margaret E~M Thomas.
\newblock An effective {Pila-Wilkie} theorem for sets definable using pfaffian
  functions, with some diophantine applications.
\newblock {\em {P}reprint}, arXiv:2301.09883, January 2023.

\bibitem{me:sharp-os-cells}
Gal Binyamini, Dmitri Novikov, and Benny Zack.
\newblock Sharply o-minimal structures and sharp cellular decomposition.
\newblock {\em {P}reprint}, arXiv:1407.1183, September 2022.

\bibitem{me:noetherian-dim2}
Gal Binyamini and Dmitry Novikov.
\newblock Intersection multiplicities of {N}oetherian functions.
\newblock {\em Adv. Math.}, 231(6):3079--3093, 2012.

\bibitem{me:deflicity}
Gal Binyamini and Dmitry Novikov.
\newblock Multiplicities of {N}oetherian deformations.
\newblock {\em Geom. Funct. Anal.}, 25(5):1413--1439, 2015.

\bibitem{me:mult-ops}
Gal Binyamini and Dmitry Novikov.
\newblock Multiplicity operators.
\newblock {\em Israel Journal of Mathematics}, 210(1):101--124, 2015.

\bibitem{me:c-cells}
Gal Binyamini and Dmitry Novikov.
\newblock Complex cellular structures.
\newblock {\em Ann. of Math. (2)}, 190(1):145--248, 2019.

\bibitem{me:icm2022}
Gal Binyamini and Dmitry Novikov.
\newblock Tameness in geometry and arithmetic: beyond o-minimality.
\newblock In {\em I{CM}---{I}nternational {C}ongress of {M}athematicians.
  {V}ol. {III}. {S}ections 1--4}, pages 1440--1461. EMS Press, Berlin, [2023]
  \copyright 2023.

\bibitem{me:inf16}
Gal Binyamini, Dmitry Novikov, and Sergei Yakovenko.
\newblock On the number of zeros of {A}belian integrals.
\newblock {\em Invent. Math.}, 181(2):227--289, 2010.

\bibitem{me:pfaff-wilkie}
Gal Binyamini, Dmitry Novikov, and Benny Zack.
\newblock Wilkie's conjecture for {P}faffian structures.
\newblock {\em Accepted to appear in Annals of Math.}, arXiv:2202.05305, 2022.

\bibitem{bv:rest-pfaff}
Gal Binyamini and Nicolai Vorobjov.
\newblock Effective cylindrical cell decompositions for restricted
  sub-{P}faffian sets.
\newblock {\em Int. Math. Res. Not. IMRN}, 2022(5):3493--3510, 11 2020.

\bibitem{borel:d-modules}
A.~Borel, P.-P. Grivel, B.~Kaup, A.~Haefliger, B.~Malgrange, and F.~Ehlers.
\newblock {\em Algebraic {$D$}-modules}, volume~2 of {\em Perspectives in
  Mathematics}.
\newblock Academic Press, Inc., Boston, MA, 1987.

\bibitem{cdk:hodge-locus}
Eduardo Cattani, Pierre Deligne, and Aroldo Kaplan.
\newblock On the locus of {H}odge classes.
\newblock {\em J. Amer. Math. Soc.}, 8(2):483--506, 1995.

\bibitem{poussin:oscillation}
C.~de~la Vall\'ee~Poussin.
\newblock Sur l'\'equation diff\'erentielle lin\'eaire du second ordre.
  d\'etermination d'une int\'egrale par deux valeurs assign\'ees. extension aux
  \'equations d'ordre n.
\newblock {\em Journal de Math\'ematiques Pures et Appliqu\'ees}, 8:125--144,
  1929.

\bibitem{deligne:diff-eqs}
Pierre Deligne.
\newblock {\em \'{E}quations diff\'{e}rentielles \`a points singuliers
  r\'{e}guliers}, volume Vol. 163 of {\em Lecture Notes in Mathematics}.
\newblock Springer-Verlag, Berlin-New York, 1970.

\bibitem{gabrielov:pfaff-mult}
A.~Gabri\`elov.
\newblock Multiplicities of {P}faffian intersections, and the \l ojasiewicz
  inequality.
\newblock {\em Selecta Math. (N.S.)}, 1(1):113--127, 1995.

\bibitem{gabrielov:closure}
A.~Gabrielov.
\newblock Frontier and closure of a semi-{P}faffian set.
\newblock {\em Discrete Comput. Geom.}, 19(4):605--617, 1998.

\bibitem{gv:stratifications}
A.~Gabrielov and N.~Vorobjov.
\newblock Complexity of stratifications of semi-{P}faffian sets.
\newblock {\em Discrete Comput. Geom.}, 14(1):71--91, 1995.

\bibitem{gv:cell-decomposition}
A.~Gabrielov and N.~Vorobjov.
\newblock Complexity of cylindrical decompositions of sub-{P}faffian sets.
\newblock {\em J. Pure Appl. Algebra}, 164(1-2):179--197, 2001.

\bibitem{gv:complexity-computations}
A.~Gabrielov and N.~Vorobjov.
\newblock Complexity of computations with {P}faffian and {N}oetherian
  functions.
\newblock In {\em Normal forms, bifurcations and finiteness problems in
  differential equations}, volume 137 of {\em NATO Sci. Ser. II Math. Phys.
  Chem.}, pages 211--250. Kluwer Acad. Publ., Dordrecht, 2004.

\bibitem{gab:projections}
A.~M. Gabri\`elov.
\newblock Projections of semianalytic sets.
\newblock {\em Funkcional. Anal. i Prilo\v zen.}, 2(4):18--30, 1968.

\bibitem{gabrielov:mult-morse}
Andrei Gabrielov.
\newblock Multiplicity of a zero of an analytic function on a trajectory of a
  vector field.
\newblock In {\em The {A}rnoldfest ({T}oronto, {ON}, 1997)}, volume~24 of {\em
  Fields Inst. Commun.}, pages 191--200. Amer. Math. Soc., Providence, RI,
  1999.

\bibitem{gabrielov:limit-sets}
Andrei Gabrielov.
\newblock Relative closure and the complexity of {P}faffian elimination.
\newblock In {\em Discrete and computational geometry}, volume~25 of {\em
  Algorithms Combin.}, pages 441--460. Springer, Berlin, 2003.

\bibitem{GabKho}
Andrei Gabrielov and Askold Khovanskii.
\newblock Multiplicity of a {N}oetherian intersection.
\newblock In {\em Geometry of differential equations}, volume 186 of {\em Amer.
  Math. Soc. Transl. Ser. 2}, pages 119--130. Amer. Math. Soc., Providence, RI,
  1998.

\bibitem{gsv:quantum-complexity}
Thomas~W. Grimm, Lorenz Schlechter, and Mick van Vliet.
\newblock {Complexity in Tame Quantum Theories}.
\newblock 10 2023.

\bibitem{kashiwara:quasiunipotent}
M.~Kashiwara.
\newblock Quasi-unipotent constructible sheaves.
\newblock {\em J. Fac. Sci. Univ. Tokyo Sect. IA Math.}, 28(3):757--773 (1982),
  1981.

\bibitem{asik:finiteness}
A.~Khovanskii.
\newblock Real analytic manifolds with the property of finiteness, and complex
  abelian integrals.
\newblock {\em Funktsional. Anal. i Prilozhen.}, 18(2):40--50, 1984.

\bibitem{ky:rolle}
A.~Khovanskii and S.~Yakovenko.
\newblock Generalized {R}olle theorem in {${\bf R}^n$} and {${\bf C}$}.
\newblock {\em J. Dynam. Control Systems}, 2(1):103--123, 1996.

\bibitem{khovanskii:fewnomials}
A.~G. Khovanski{\u\i}.
\newblock {\em Fewnomials}, volume~88 of {\em Translations of Mathematical
  Monographs}.
\newblock American Mathematical Society, Providence, RI, 1991.
\newblock Translated from the Russian by Smilka Zdravkovska.

\bibitem{khovanskii:icm}
A.~G. Khovanski\u{\i}.
\newblock Fewnomials and {P}faff manifolds.
\newblock In {\em Proceedings of the {I}nternational {C}ongress of
  {M}athematicians, {V}ol. 1, 2 ({W}arsaw, 1983)}, pages 549--564. PWN, Warsaw,
  1984.

\bibitem{lms:diff-poly-bdd}
Jean-Marie Lion, Chris Miller, and Patrick Speissegger.
\newblock Differential equations over polynomially bounded o-minimal
  structures.
\newblock {\em Proc. Amer. Math. Soc.}, 131(1):175--183, 2003.

\bibitem{ny:chains}
D.~Novikov and S.~Yakovenko.
\newblock Trajectories of polynomial vector fields and ascending chains of
  polynomial ideals.
\newblock {\em Ann. Inst. Fourier (Grenoble)}, 49(2):563--609, 1999.

\bibitem{ps:complex-omin}
Ya'acov Peterzil and Sergei Starchenko.
\newblock Complex analytic geometry and analytic-geometric categories.
\newblock {\em J. Reine Angew. Math.}, 626:39--74, 2009.

\bibitem{ps:theta-definability}
Ya'acov Peterzil and Sergei Starchenko.
\newblock Definability of restricted theta functions and families of abelian
  varieties.
\newblock {\em Duke Math. J.}, 162(4):731--765, 2013.

\bibitem{pw:thm}
J.~Pila and A.~J. Wilkie.
\newblock The rational points of a definable set.
\newblock {\em Duke Math. J.}, 133(3):591--616, 2006.

\bibitem{pila:blocks}
Jonathan Pila.
\newblock On the algebraic points of a definable set.
\newblock {\em Selecta Math. (N.S.)}, 15(1):151--170, 2009.

\bibitem{pila:andre-oort}
Jonathan Pila.
\newblock O-minimality and the {A}ndr\'e-{O}ort conjecture for {$\mathbb C^n$}.
\newblock {\em Ann. of Math. (2)}, 173(3):1779--1840, 2011.

\bibitem{pila:book}
Jonathan Pila.
\newblock {\em Point-counting and the {Z}ilber-{P}ink conjecture}, volume 228
  of {\em Cambridge Tracts in Mathematics}.
\newblock Cambridge University Press, Cambridge, 2022.

\bibitem{schmid:variation-hodge}
Wilfried Schmid.
\newblock Variation of {H}odge structure: the singularities of the period
  mapping.
\newblock {\em Invent. Math.}, 22:211--319, 1973.

\bibitem{speissegger:pfaff-closure}
Patrick Speissegger.
\newblock The {P}faffian closure of an o-minimal structure.
\newblock {\em J. Reine Angew. Math.}, 508:189--211, 1999.

\bibitem{tougeron:noetherian}
Jean-Claude Tougeron.
\newblock Alg\`ebres analytiques topologiquement noeth\'{e}riennes.
  {T}h\'{e}orie de {K}hovanski\u{\i}.
\newblock {\em Ann. Inst. Fourier (Grenoble)}, 41(4):823--840, 1991.

\bibitem{urbanik:bdd-degree}
David Urbanik.
\newblock Sets of special subvarieties of bounded degree.
\newblock {\em Compos. Math.}, 159(3):616--657, 2023.

\bibitem{varchenko:finiteness}
A.~N. Varchenko.
\newblock Estimation of the number of zeros of an abelian integral depending on
  a parameter, and limit cycles.
\newblock {\em Funktsional. Anal. i Prilozhen.}, 18(2):14--25, 1984.

\bibitem{wilkie:Rexp}
A.~J. Wilkie.
\newblock Model completeness results for expansions of the ordered field of
  real numbers by restricted {P}faffian functions and the exponential function.
\newblock {\em J. Amer. Math. Soc.}, 9(4):1051--1094, 1996.

\bibitem{wilkie:new-omin}
A.~J. Wilkie.
\newblock A theorem of the complement and some new o-minimal structures.
\newblock {\em Selecta Math. (N.S.)}, 5(4):397--421, 1999.

\bibitem{yakovenko:functions-and-curves}
Sergei Yakovenko.
\newblock On functions and curves defined by ordinary differential equations.
\newblock In {\em The {A}rnoldfest ({T}oronto, {ON}, 1997)}, volume~24 of {\em
  Fields Inst. Commun.}, pages 497--525. Amer. Math. Soc., Providence, RI,
  1999.

\bibitem{zannier:book}
Umberto Zannier.
\newblock {\em Some problems of unlikely intersections in arithmetic and
  geometry}, volume 181 of {\em Annals of Mathematics Studies}.
\newblock Princeton University Press, Princeton, NJ, 2012.
\newblock With appendixes by David Masser.

\end{thebibliography}

\end{document}